\documentclass[12pt]{amsart}
\usepackage{amsmath,amsfonts,amsthm,amsopn,amssymb,mathtools,stmaryrd}
\usepackage{cite,marginnote}
\usepackage{pdfsync}

\usepackage{color,enumitem,graphicx}
\usepackage[colorlinks=true,urlcolor=blue,
citecolor=red,linkcolor=blue,linktocpage,pdfpagelabels,
bookmarksnumbered,bookmarksopen]{hyperref}
\usepackage[english]{babel}

\usepackage[left=2.75cm,right=2.75cm,top=3cm,bottom=3cm]{geometry}

\usepackage[hyperpageref]{backref}
\usepackage[]{xcolor}

\def\sideremark#1{\ifvmode\leavevmode\fi\vadjust{\vbox to0pt{\vss
 \hbox to 0pt{\hskip\hsize\hskip1em
 \vbox{\hsize2.1cm\tiny\raggedright\pretolerance10000
  \noindent #1\hfill}\hss}\vbox to15pt{\vfil}\vss}}}%

\allowdisplaybreaks
\numberwithin{equation}{section}

\makeindex
\newtheorem{theorem}{Theorem}[section]
\newtheorem{proposition}[theorem]{Proposition}
\newtheorem{lemma}[theorem]{Lemma}
\newtheorem{remark}[theorem]{Remark}
\newtheorem{example}[theorem]{Example}
\newtheorem{corollary}[theorem]{Corollary}
\newtheorem{definition}[theorem]{Definition}

\newcommand{\bt}{\begin{theorem}}
\newcommand{\et}{\end{theorem}}
\newcommand{\bl}{\begin{lemma}}
\newcommand{\el}{\end{lemma}}
\newcommand{\bd}{\begin{definition}}
\newcommand{\ed}{\end{definition}}
\newcommand{\bc}{\begin{corollary}}
\newcommand{\ec}{\end{corollary}}
\newcommand{\bx}{\begin{example}}
\newcommand{\ex}{\end{example}}
\newcommand{\bi}{\begin{exercise}}
\newcommand{\ei}{\end{exercise}}
\newcommand{\bo}{\begin{proposition}}
\newcommand{\eo}{\end{proposition}}
\newcommand{\br}{\begin{remark}}
\newcommand{\er}{\end{remark}}
\newcommand{\be}{\begin{equation}}
\newcommand{\ee}{\end{equation}}
\newcommand{\ba}{\begin{align}}
\newcommand{\ea}{\end{align}}
\newcommand{\bn}{\begin{enumerate}}
\newcommand{\en}{\end{enumerate}}
\newcommand{\bg}{\begin{align*}}
\newcommand{\bcs}{\begin{cases}}
\newcommand{\ecs}{\end{cases}}

\newcommand{\bean}{\begin{eqnarray*}}
\newcommand{\eean}{\end{eqnarray*}}

\newcommand{\dis}{\displaystyle}

\renewcommand{\leq}{\leqslant}
\renewcommand{\geq}{\geqslant}

\title[Dirichlet problems involving critical growth]{Multi-peak solutions for singularly perturbed nonlinear Dirichlet problems involving critical growth}

\author[Y.~He]{Yi He}
\author[J. C. ~Wei]{Juncheng Wei}
\author[J. J.~Zhang]{Jianjun Zhang}

\address[Y. He]{\newline\indent School of Mathematics and Statistics
\newline\indent
South-Central University for Nationalities
\newline\indent
Wuhan 430074, P. R. China}
\email{\href{mailto:heyi19870113@163.com}{heyi19870113@163.com}}

\address[J. C. Wei]{\newline\indent
Department of Mathematics,
\newline\indent
University of British Columbia,
\newline\indent
Vancouver V6T 1Z2, Canada}
\email{\href{mailto:jcwei@math.ubc.ca}{jcwei@math.ubc.ca}}

\address[J. J. ~Zhang]{\newline\indent College of Mathematics and Statistics
\newline\indent
Chongqing Jiaotong University
\newline\indent
Chongqing 400074, PR China}
\email{\href{mailto:zhangjianjun09@tsinghua.org.cn}{zhangjianjun09@tsinghua.org.cn}}

\thanks{The research of J. Wei is partially supported by NSERC of Canada. Y. He was supported by NSFC (No. 11601530) and the Fundamental Research Funds for the Central Universities, South-Central University for Nationalities (No. CZT 20008). J. J. Zhang was supported by NSFC (No. 11871123).}

\date{April 26, 2022}

\UseRawInputEncoding

\begin{document}

\begin{abstract}
We consider the following singularly perturbed elliptic problem
\[
 - {\varepsilon ^2}\Delta u + u = f(u){\text{ in }}\Omega ,{\text{ }}u > 0{\text{ in }}\Omega ,{\text{ }}u = 0{\text{ on }}\partial \Omega ,
\]
where $\Omega$ is a domain in ${\mathbb{R}^N}(N \ge 3)$, not necessarily bounded, with boundary $\partial \Omega  \in {C^2}$ and the nonlinearity $f$ is of critical growth. In this paper, we construct a family of multi-peak solutions to the equation given above which concentrate around any prescribed finite sets of local maxima of the distance function from the boundary $\partial \Omega$.
\end{abstract}
\maketitle

\section{Introduction and main result}

\setcounter{equation}{0}

In the present paper, we are concerned with the following singularly perturbed nonlinear Dirichlet problem
\begin{equation}\label{1.1}
\left\{ \begin{gathered}
   - {\varepsilon ^2}\Delta u + u = f(u){\text{ in }}\Omega , \hfill \\
  u > 0{\text{ in }}\Omega ,~u = 0{\text{ on }}\partial \Omega , \hfill \\
\end{gathered}  \right.
\end{equation}
where $\varepsilon >0$ is a small parameter and $f \in C(\mathbb{R},\mathbb{R})$ satisfies
\begin{flalign*}
&({F_1})~f(t) = 0{\text{ for }}t < 0{\text{ and }}\mathop {\lim }\limits_{t \to {0^ + }} f(t)/t = 0.&
\end{flalign*}
\begin{flalign*}
&({F_2})~\mathop {\lim }\limits_{t \to  + \infty } f(t)/{t^{{2^ * } - 1}} = 1.&
\end{flalign*}
\begin{flalign*}
&({F_3})~{\text{there exist }}\lambda  > 0{\text{ and }}2 < q < {2^ * }{\text{ such that }}f(t) \ge \lambda {t^{q - 1}} + {t^{{2^ * } - 1}}{\text{ for }}t \ge 0,&
\end{flalign*}
where ${2^ * } = 2N/(N - 2)$. $({F_1})$-$({F_3})$ were introduced by J. Zhang and W.Zou \cite{zz} and can be regarded as an extension of the celebrated Berestycki-Lions' type nonlinearity (\cite{bl1,bl2}) to the critical growth. Furthermore, we assume that $\Omega$ is a domain in ${\mathbb{R}^N}(N \ge 3)$, not necessarily bounded, with boundary $\partial \Omega  \in {C^2}$ and there are $K$ smooth bounded sub-domains of $\Omega$, ${\Lambda ^1}$, ${\Lambda ^2}$,...,${\Lambda ^K}$ compactly contained in $\Omega$ satisfying
\begin{flalign*}
&({H_1})~{D_k}: = \mathop {\max }\limits_{x \in {\Lambda ^k}} d(x,\partial \Omega ) > \mathop {\max }\limits_{x \in \partial {\Lambda ^k}} d(x,\partial \Omega ),{\text{ }}k = 1,2,...,K.&
\end{flalign*}
\begin{flalign*}
&({H_2})~\mathop {\min }\limits_{i \ne j} d({\Lambda ^i},{\Lambda ^j}) > 2\mathop {\max }\limits_{1 \le k \le K} \mathop {\max }\limits_{x \in {\Lambda ^k}} d(x,\partial \Omega ),
&
\end{flalign*}
where $d({\Lambda ^i},{\Lambda ^j})$ is the distance between ${\Lambda ^i}$, ${\Lambda ^j}$, i.e.
\[
d({\Lambda ^i},{\Lambda ^j}) = \mathop {\inf }\limits_{x \in {\Lambda ^i},y \in {\Lambda ^j}} d(x,y).
\]
This kind of hypothesis was first introduced by M. del Pino, P. L. Felmer and J. Wei in \cite{dfw}.

There have been many results on equation \eqref{1.1} and related elliptic systems (see \cite{b,b1,bzz,cdny,df2,dfw,dfw1,ln,lw,nw,ntw,ny,w,wz} and references therein). For global cases, W. M. Ni and J. Wei \cite{nw} proved that for  $\varepsilon > 0$ small, there exists a positive ground state solution to \eqref{1.1} which possesses a single spike-layer with its unique peak in the interior of $\Omega$. Moreover, this unique peak must concentrate around  the ``most-centered" part of $\Omega$, i.e. the set of global maximum of the distance function $d(x,\partial \Omega )$, $x \in \Omega $. Note that, they required that $\Omega $ is a smooth bounded domain in ${\mathbb{R}^N}(N \ge 2)$ and $f:\mathbb{R} \to \mathbb{R}$ is of ${C^{1,\alpha }}(\mathbb{R},\mathbb{R})$ with $0 < \alpha  < 1$ satisfying the following conditions.
\begin{itemize}
\item [$({f_1})$] $f(t) \equiv 0$ for $t \le 0$ and $f(t) \to  + \infty $ as $t \to  + \infty $.
\item [$({f_2})$] For $t \ge 0$, $f$ admits the decomposition in ${C^{1,\alpha }}(\mathbb{R},\mathbb{R})$, $f(t) = {f_1}(t) - {f_2}(t)$, where
\begin{itemize}
\item [$(i)$] ${f_1}(t) \ge 0$ and ${f_2}(t) \ge 0$ with ${f_1}(0) = {f'_1}(0) = 0$ and ${f_2}(0) = {f'_1}(0) = 0$;
\item [$(ii)$] ${f_1}(t)/{t^q}$ is non-decreasing, ${f_2}(t)/{t^q}$ is non-increasing for $t>0$ respectively for some $q \ge 1$ and ${f_1}(t)/t$ is strictly increasing for $t>0$.
\end{itemize}
\item [$({f_3})$] $f(t) = O({t^p})$ as $t \to +\infty$, where $1 < p < (N + 2)/(N - 2)$ if $N \ge 3$ and $1 < p < +\infty$ if $N = 2$.
\item [$({f_4})$] For some $ \mu  > 2$, $\mu \int_0^ t {f(s)} ds \le tf(t), t \ge 0$.
\item [$({f_5})$] there exists a unique radially symmetric solution $w \in {H^1}({\mathbb{R}^N})$ for $ - \Delta u + u = f(u)$, $u>0$ in ${\mathbb{R}^N}$ such that if $v \in {H^1}({\mathbb{R}^N})$ and satisfies $ - \Delta v + v = f'(w)v$, then $v = \sum\nolimits_{i = 1}^n {{a_i}\frac{{\partial w}}
{{\partial {x_i}}}} $ for some ${a_1},...,{a_n} \in \mathbb{R}$.
\end{itemize}
M. del Pino and P. L. Felmer \cite{df2} proved that the same concentration phenomenon can be obtained in a simpler way even without the non-degeneracy condition $({f_5})$. Their approach depends strongly on the monotonicity condition, i.e. $f(t)/t$ is strictly increasing for $t>0$. Furthermore, J. Byeon \cite{b} showed the concentration behavior of positive ground state solutions to (1.1) without conditions (ii) of $({f_2})$ and $({f_5})$ by some ODE technique and symmetrization methods for $N \ge 3$. When the nonlinearity $f$ is typically of the form $f(u) = u - u(u - a)(u - 1)$ with $0 < a < 1/2$, which has appeared in various models in applied mathematics, including population genetics and chemical reactor theory. W. M. Ni, I. Takagi and J. Wei \cite{ntw} proved that if $\Omega$ satisfies a geometric condition, that is, $\partial \Omega  = {\Gamma _1} \cup {\Gamma _2}$, where both ${\Gamma _1}$ and ${\Gamma _2}$ are closed (one of them could be empty) and satisfy (i) At every point of ${\Gamma _1}$, all sectional curvatures of ${\Gamma _1}$ are bounded below by a positive constant $\alpha > 0$ and (ii) There exists a point ${x_0}$ in ${\mathbb{R}^N}$ such that $(x - {x_0},\nu (x)) \le 0$ for all $x \in {\Gamma _2}$, where $\nu (x)$ denotes the unit outer normal to $\partial \Omega $ at $x$, the same concentration phenomenon as \cite{nw} occurs. Their assumptions for $f(t) = t + g(t)$ are the following
\begin{itemize}
\item [$({g_1})$] $g \in {C^{1,\alpha }}(\mathbb{R},\mathbb{R})$ with $0 < \alpha  < 1$ and $g(0) = 0$, $g'(0) < 0$.
\item [$({g_2})$] $g$ has two positive zero $z_1$ and $z_2$, such that ${z_1} < {z_2}$ and has no other positive zeros.
\item [$({g_3})$] $\int_0^{{z_2}} {g(s)} ds > 0$.
\item [$({g_4})$] The function $t \to f(t)/(t - {t_0})$ is decreasing in the interval $({t_0},{z_2})$, where ${t_0}$ is the unique number in $({z_1},{z_2})$ such that$\int_0^{{t_0}} {f(s)} ds = 0$.
\end{itemize}
J. Byeon \cite{b1} developed a new variational method to construct a family of single-peak solutions to \eqref{1.1} concentrating around the set of global maximum of the distance function $d(x,\partial \Omega )$, $x \in \Omega $ under an almost optimal conditions on $f$. More precisely, the author assumed that $\Omega $ is a smooth bounded domain in ${\mathbb{R}^N}(N \ge 3)$ and $f$ satisfies the celebrated Berestycki-Lions' type assumptions (\cite{bl1,bl2}), i.e.
\begin{itemize}
\item [$({h_1})$] $f \in C(\mathbb{R},\mathbb{R})$ such that $f(t) = 0$ for $t \le 0$ and $\mathop {\lim }\nolimits_{t \to {0^ + }} f(t)/t = 0$.
\item [$({h_2})$] There exists $p \in (1,(N + 2)/(N - 2))$ such that $\mathop {\overline {\lim } }\nolimits_{t \to  + \infty } f(t)/{t^p} < \infty $.
\item [$({h_3})$] There exists $T>0$ such that ${T^2}/2 < F(T): = \int_0^t {f(t)} dt$.
\end{itemize}
J. Byeon, J. Zhang and W. Zou \cite{bzz} extended the results of \cite{b1} to the critical case, i.e. $f$ satisfies $({F_1})$-$({F_3})$.

For local cases, J. Wei \cite{w} investigated a related problem
\begin{equation}\label{1.2}
\left\{ \begin{gathered}
   - {\varepsilon ^2}\Delta u + u = u^p{\text{ in }}\Omega , \hfill \\
  u > 0{\text{ in }}\Omega ,~u = 0{\text{ on }}\partial \Omega , \hfill \\
\end{gathered}  \right.
\end{equation}
where $\Omega $ is a smooth bounded domain in ${\mathbb{R}^N}(N \ge 2)$, $1 < p < (N + 2)/(N - 2)$ if $N \ge 3$ and $1 < p < +\infty$ if $N = 2$. The author shows a local version of the results in \cite{nw}, namely there exists a family of local single-peak solutions to \eqref{1.2} which concentrate around any prescribed strict local maximum point of the distance function $d(x,\partial \Omega )$, $x \in \Omega $. D. Cao, N. Dancer, E. Noussair and S. Yan \cite{cdny} constructed a family of double-peak solutions to \eqref{1.2} which concentrate around two strict local maximum points (denoted by $P_1$ and $P_2$) of the the distance function $d(x,\partial \Omega )$, $x \in \Omega $. Moreover, they required that
\begin{center}
(i) $|{P_1} - {P_2}| > \max \{ 1/(p - 1),2\} d({P_1},\Omega )$  and  (ii) $d({P_1},\partial \Omega ) = d({P_2},\partial \Omega )$.
\end{center}
M. del Pino, P. L. Felmer and J. Wei \cite{dfw} constructed a family of multi-peak solutions to \eqref{1.1} concentrating around any prescribed finite sets of local maxima of the distance function $d(x,\partial \Omega )$, $x \in \Omega $ via the penalization method due to \cite{df1}. In \cite{dfw}, the authors assumed that $\Omega$ is a smooth, not necessarily bounded domain in ${\mathbb{R}^N}(N \ge 2)$, $f$ satisfies $(f_1)$-$(f_5)$ and the distance function $d(x,\partial \Omega )$, $x \in \Omega $ satisfies $(H_1)$-$(H_2)$.

Furthermore, J. Byeon and K. Tanaka \cite{bt} considered the following singularly perturbed elliptic problem
\begin{equation}\label{add20}
 - {\varepsilon ^2}\Delta u + V(x)u = f(u){\text{ in }}{\mathbb{R}^N}.
\end{equation}
They developed a local variational and deformation argument to establish the existence of a family of positive solutions to \eqref{add20} with multiple peaks cluster near a local maximum of $V(x)$ when the nonlinearity $f$ is subcritical, actually satisfies the Berestycki-Lions' type assumptions $({h_1})$-$({h_3})$. An instability property of critical points of $V(x)$ plays an crucial role in the gluing argument. And then, J. Byeon and K. Tanaka \cite{bt1} applied the arguments in \cite{bt} to show the existence of multi-bump positive solutions for a nonlinear elliptic problem in expanding tubular domains. Compared to the results by J. Byeon and K. Tanaka \cite{bt}, our main interest in this paper is the existence of solutions with concentration relying on the effect from the boundary of domains.
\vskip0.1in
Now, we are ready to state our main result as follows.
\begin{theorem}\label{1.1.}
Let $\Omega$ be a domain in ${\mathbb{R}^N}(N \ge 3)$, not necessarily bounded, with boundary $\partial \Omega  \in {C^2}$, which satisfies $({H_1})$-$({H_2})$. Assume that $f$ satisfies $({F_1})$-$({F_3})$ and one of the following two conditions holds
\begin{flalign*}
&{({F_4})_1}{\text{ }}{\text{For each }}\lambda  > 0{\text{ fixed, let }}2 < q < {2^ * }{\text{ for }}N \ge 4{\text{ and }}4 < q < 6{\text{ for }}N = 3.&
\end{flalign*}
\begin{flalign*}
&{({F_4})_2}{\text{ }}{\text{For}}{\text{ }}\lambda  > 0{\text{ sufficiently large, let }}2 < q \le 4{\text{ for }}N = 3.&
\end{flalign*}
Then, there exists ${\varepsilon _0} > 0$ such that for every $\varepsilon  \in (0,{\varepsilon _0}]$, \eqref{1.1} possesses a positive bound state solution ${u_\varepsilon } \in C_{{\text{loc}}}^{2,\alpha }({\mathbb{R}^N})$ for some $\alpha  \in (0,1)$. Moreover,\\
$(i)$ ${u_\varepsilon }$ possesses $K$ local maximum points $x_\varepsilon ^k \in {\Lambda ^k}$$(k = 1,2,...,K)$ such that for each $k$,
\[
d(x_\varepsilon ^k,\partial \Omega ) \to \mathop {\max }\limits_{x \in {\Lambda ^k}} d(x,\partial \Omega ){\text{ as }}\varepsilon  \to 0.
\]
$(ii)$ There holds that
\[
{u_\varepsilon }(x) \le {C_1}\exp \Bigl( { - \frac{{{C_2}}}
{\varepsilon }\mathop {\min }\limits_{1 \le k \le K} |x - x_\varepsilon ^k|} \Bigr),
\]
where ${C_1},{C_2} > 0$ are independent of $\varepsilon$.
\end{theorem}

As far as we know, the existence and concentration of multi-peak solutions to problem \eqref{1.1} in the critical case have not ever been studied by variational methods. The proof of Theorem~\ref{1.1.} is mainly based on variational methods. The main difficulties are three-fold.
\begin{itemize}
\item [$(1)$] Under the conditions $({F_1})$-$({F_3})$, neither the so called Ambrosetti-Rabinowitz condition nor the monotonicity condition of $f(s)/s$ holds. It prevent us from obtaining a bounded Palais-Smale sequence or using the Nehari manifold. So the arguments in \cite{dfw} seems difficult to be applied directly in this paper.
\item [$(2)$] The possible unboundedness of the domain $\Omega$ and the critical growth of the nonlinearity $f$ lead to the lack of compactness. To perform the proof of Theorem \ref{1.1}, we mainly use a deformation argument as in \cite{g,cr,bt} to glue together mountain pass solutions of the associated limit problems. The existence of multi-peak solutions heavily rely on the behavior of the energy functional in some annular neighborhood. In contrast with \cite{bt}, due to the presence of the Sobolev critical exponent, it seems more complicated to adopt the Lions lemma to establish the behavior of $\varepsilon$-dependent Palais-Smale sequences. To bypass this obstacle, a Liouville type result and the concentration compactness principle II due to P. L. Lions are involved.
\item [$(3)$] For constructing multi-peak solutions to \eqref{1.1} which concentrate around any prescribed finite set of local maxima of the distance function without the non-degeneracy condition, energy estimates plays a crucial role. Actually, one need to find sharp estimates of the difference between ${c_\varepsilon }$ (an energy level of the auxiliary functional ${J_\varepsilon }$) and $c$ (the mountain pass level of the limiting equation \eqref{2.1}) (Lemma~\ref{3.1.}, Lemma~\ref{3.2.} and \eqref{3.126} below).
\end{itemize}
As we shall see later, the above three aspects prevent us from using variational methods in a standard way and more delicate analysis should be carried out.

We emphasize that in \cite{dfw}, M. del Pino, P. L. Felmer and J. Wei \cite{dfw} investigated a subcritical problem with a more or less restricted class of nonlinearities $f$ and obtained a family of multi-peak solutions to \eqref{1.1} concentrating around any prescribed finite sets of local maxima of the distance function $d(x,\partial \Omega )$, $x \in \Omega $ via the penalization method due to \cite{df1}. Different from \cite{dfw}, in the present paper, we are concerned with the critical case with the nonlinearity $f$, which is considered to be almost optimal in some sense. The proof is mainly based on the variational approach which is initiated by \cite{b1,bzz}. To complete this section, we give the sketch of our proof.

First, we need to study the ``limiting problem"
\begin{equation}\label{1.3}
 - \Delta u + u = f(u){\text{ in }}{\mathbb{R}^N},~u > 0,~u \in {H^1}({\mathbb{R}^N}).
\end{equation}
It is unknown that whether positive solutions to \eqref{1.3} are unique or not under our assumptions above. Nevertheless, J. Byeon, J. Zhang and W. Zou \cite{bzz} indicated that the set of positive ground state solutions to \eqref{1.3} satisfies some compactness properties (Proposition~\ref{2.1.} below). This is crucial for finding multi-peak solutions which are close to a set of prescribed functions.

Second, to study \eqref{1.1}, we work with the following equivalent equation
\begin{equation}\label{1.4}
\left\{ \begin{gathered}
   - \Delta v + v = f(v){\text{ in }}{\Omega _\varepsilon }, \hfill \\
  v > 0{\text{ in }}{\Omega _\varepsilon },~v = 0{\text{ on }}\partial {\Omega _\varepsilon, } \hfill \\
\end{gathered}  \right.
\end{equation}
with the corresponding energy functional
\[
{I_\varepsilon }(v): = \frac{1}
{2}\int_{{\Omega _\varepsilon }} {|\nabla v{|^2}}  + \frac{1}
{2}\int_{{\Omega _\varepsilon }} {{v^2}}  - \int_{{\Omega _\varepsilon }} {F(v)} ,~v \in H_0^1({\Omega _\varepsilon }),
\]
where ${\Omega _\varepsilon }: = \{ y \in {\mathbb{R}^N}:\varepsilon y \in \Omega \} $. In contrast with \cite{b,b1,bzz}, where the mountain pass theorem was used in some bounded domain $\Omega$ and only the global single-peak solutions were considered, in the present paper, the conditions $(H_1)$-$(H_2)$ are local. So we introduce a local penalization method in \cite{df,df1,g}, which helps us to overcome the difficulty caused by the non-compactness due to the unboundedness of the domain $\Omega$.

Third, in order to construct multi-peak solutions to \eqref{1.1} with a critical nonlinearity $f$, inspired by \cite{bw,bj,bj1} but with some modifications, we introduce nonlocal penalization functionals ${Q_\varepsilon }$, $Q_\varepsilon ^k$ $((k = 1,2,...K)$. We modify the functional ${I_\varepsilon }$ and define auxiliary functionals ${J_\varepsilon }$, $J_\varepsilon ^k$ $(k = 1,2,...K)$ respectively (see Section 3 for details). It will be shown that this type of penalization forces the concentration phenomena to occur inside $\Lambda : = \mathop  \cup \nolimits_{k = 1}^K {\Lambda ^k}$ (Proposition~\ref{3.4.} below).

Finally, in order to get a critical point ${v_\varepsilon }$ of ${J_\varepsilon }$, we use a version of quantitative deformation lemma (Lemma~\ref{3.7.} below) to construct a special convergent Palais-Smale sequence of ${J_\varepsilon }$ for $\varepsilon > 0$ small. To prove that ${v_\varepsilon }$ is indeed a solution to the original problem \eqref{1.4}, we need to establish a uniform ${L^\infty }$-estimate of ${v_\varepsilon }$ by using the classical Brezis-Kato's type argument (see ``Proof of  Theorem \ref{1.1.}", \textbf{Step~2} below) and prove the concentration results of ${v_\varepsilon }$ by some more delicate energy estimates (see ``Proof of  Theorem \ref{1.1.}", \textbf{Step~3} below).

Noting that ${v_\varepsilon }$ consists essentially of $K$ disjoints parts, each part of which is close to a ground state solution to \eqref{1.3}.  Namely, in our approach we take into account the shape and location of the solutions we expect to find. Thus, on the one hand, we benefit from the advantage of the Lyapunov-Schmidt reduction type approach, which is to discover the solution around a small neighbourhood of a well chosen first approximation. On the other hand, our approach, which is purely variational, does not require any uniqueness or non-degeneracy conditions. Furthermore, since $f$ is merely required to be in $C(\mathbb{R},\mathbb{R})$, the corresponding energy functional ${I_\varepsilon }$ to \eqref{1.4} may be only of $C^1$-smoothness. For this reason, it becomes tough to use the Lyapunov-Schmidt reduction method, but the variational approach seems to be the right tool.

This paper is organized as follows. In Section 2, some preliminary results are introduced. In Section 3, we give the proof of Theorem~\ref{1.1.}.

\vspace{.5cm}

\section{Preliminary results}
\noindent The equation
\begin{equation}\label{2.1}
 - \Delta u + u = f(u){\text{ in }}{\mathbb{R}^N},~u > 0,~u \in {H^1}({\mathbb{R}^N})
\end{equation}
is called the limiting problem of \eqref{1.1} with the corresponding energy functional
\[
I(u) = \frac{1}
{2}\int_{{\mathbb{R}^N}} {|\nabla u{|^2}}  + \frac{1}
{2}\int_{{\mathbb{R}^N}} {{u^2}}  - \int_{{\mathbb{R}^N}} {F(u)} ,~u \in {H^1}({\mathbb{R}^N}),
\]
where $F(t) = \int_0^t {f(\tau)} d\tau$. In view of \cite{ps}, if $u \in {H^1}({\mathbb{R}^N})$ is a weak solution to \eqref{2.1}, the following Pohozaev's identity holds
\begin{equation}\label{add12}
P(u) = \frac{{N - 2}}
{2}\int_{{\mathbb{R}^N}} {|\nabla u{|^2}}  + N\Bigl( {\frac{1}
{2}\int_{{\mathbb{R}^N}} {{u^2}}  - \int_{{\mathbb{R}^N}} {F(u)} } \Bigr) = 0.
\end{equation}
Denoting by $c$ the ground state level of \eqref{2.1}, that is
\[
c = \inf \{ I(u):u \in {H^1}({\mathbb{R}^N})\backslash \{ 0\} {\text{ is a solution to }} \eqref{2.1}\}.
\]
$u \in {H^1}({\mathbb{R}^N})\backslash \{ 0\} $ is called a ground state solution to \eqref{2.1} if $u$ is a solution to \eqref{2.1} and satisfies $I(u) = c$. J. Zhang and W. Zou \cite{zz} proved that \eqref{2.1} has a ground state solution with the ground state level $c < \frac{1}
{N}{S^{N/2}}$ under the assumptions $(F_1)$-$(F_3)$, $(F_4)_1$ and $(F_4)_2$, where $S$ is the best Sobolev constant for the imbedding ${D^{1,2}}({\mathbb{R}^N}) \hookrightarrow {L^{{2^ * }}}({\mathbb{R}^N})$. Furthermore, arguing as in \cite{hit,jt}, we see that the ground state level $c$ and the mountain pass value coincide, i.e.
\[
c = \mathop {\inf }\limits_{\gamma  \in \Gamma } \mathop {\sup }\limits_{t \in [0,1]} I(\gamma (t)),
\]
where the set of paths is defined as
\[
\Gamma : = \{ {\gamma  \in C([0,1],{H^1}({\mathbb{R}^N})):\gamma (0) = 0{\text{ and }}I(\gamma (1)) < 0} \}.
\]
Let $S$ be the set of ground state solutions $U$ to \eqref{2.1} satisfying $U(0) = \mathop {\max }\nolimits_{x \in {\mathbb{R}^N}} U(x)$. From \cite{bjm,bl3,bzz,b}, we have the following results on $S$.

\begin{proposition}\label{2.1.}
$(i)$ For any $U \in S$, $U$ is radially symmetric and strictly decreasing with respect to $|x| = r > 0$.\\
$(ii)$ $S$ is compact in ${H^1}({\mathbb{R}^N})$.\\
$(iii)$ $0 < \inf \bigl\{ {{{\| U \|}_{{L^\infty }({\mathbb{R}^N})}}:U \in S} \bigr\} \le \sup \bigl\{ {{{\| U \|}_{{L^\infty }({\mathbb{R}^N})}}:U \in S} \bigr\} < \infty $.\\
$(iv)$ For any $\mu \in (0,1)$, there exists a constant ${C_\mu } > 0$ independent of $U \in S$, such that for any $U \in S$,
\[U(x) + |\nabla U(x)| \le {C_\mu }{e^{ - \mu |x|}}, x\in\mathbb{R}^N.\]
\end{proposition}


\begin{lemma}\label{2.2.}
(Lemma~2.1 of \cite{rww}) Let $R$ be a positive number and $\{ {u_n}\} _{n = 1}^\infty $ a bounded sequence in ${H^1}({\mathbb{R}^N})$$(N \ge 3)$. If
\[
\mathop {\lim }\limits_{n \to \infty } \mathop {\sup }\limits_{x \in {\mathbb{R}^N}} \int_{{B_R}(x)} {|{u_n}{|^{{2^ * }}}}  = 0,
\]
then ${u_n} \to 0$ in ${L^{{2^ * }}}({\mathbb{R}^N})$ as $n \to \infty $.
\end{lemma}

\vspace{.5cm}

\section{The singularly perturbed problem}

\subsection{Variational setting and Penalization} Letting $v(x) = u(\varepsilon x)$, \eqref{1.1} is equivalent to
\begin{equation}\label{3.1}
\left\{ \begin{gathered}
   - \Delta v + v = f(v){\text{ in }}{\Omega _\varepsilon }, \hfill \\
  v > 0{\text{ in }}{\Omega _\varepsilon },~v = 0{\text{ on }}\partial {\Omega _\varepsilon }, \hfill \\
\end{gathered}  \right.
\end{equation}
where ${\Omega _\varepsilon }: = \{ y \in {\mathbb{R}^N}:\varepsilon y \in \Omega \} $. The corresponding energy functional to \eqref{3.1} is
\[
{I_\varepsilon }(v): = \frac{1}
{2}\int_{{\Omega _\varepsilon }} {|\nabla v{|^2}}  + \frac{1}
{2}\int_{{\Omega _\varepsilon }} {{v^2}}  - \int_{{\Omega _\varepsilon }} {F(v)} ,~v \in H_0^1({\Omega _\varepsilon }).
\]
From now on, we define $\Lambda : = \mathop  \cup \nolimits_{k = 1}^K {\Lambda ^k}$ and
\[
{\mathcal{M}^k}: = \Bigl\{ {x \in {\Lambda ^k}:d(x,\partial \Omega ) = \mathop {\max }\limits_{x \in {\Lambda ^k}} d(x,\partial \Omega ): = {D_k}} \Bigr\},{\text{ }}k = 1,2,...K.
\]
Similar to \cite{dfw}, choose $R_0 > 0$ such that
\begin{equation}\label{add1}
\mathop {\max }\limits_{1 \le k \le K} \mathop {\max }\limits_{x \in {\Lambda ^k}} d(x,\partial \Omega ) < {R_0} < \frac{1}
{2}\mathop {\min }\limits_{i \ne j} d({\Lambda ^i},{\Lambda ^j})
\end{equation}
and let
\[
\widetilde{{\Lambda ^k}}: = \Bigl\{ {x \in \Omega :d(x,{\Lambda ^k}) < R_0}\Bigr\},~k = 1,2,...K.
\]
We note that for each $1 \le k \le K$, the set $\widetilde{{\Lambda ^k}}$ contains ${\Lambda ^k}$ compactly and the sets $\{ {\widetilde{{\Lambda ^k}}} \}_{k = 1}^K$ are mutually disjoint with $\mathop {\min }\nolimits_{i \ne j} d(\widetilde{{\Lambda ^i}},\widetilde{{\Lambda ^j}}) > 0$. Similar to \cite{df,df1,dfw,g}, we define the truncated function ${g_\varepsilon }(x,t)$ of $f(t)$ by
\[
{g_\varepsilon }(x,t): = {\chi _{{\Lambda _\varepsilon }}}(x)f(t) + (1 - {\chi _{{\Lambda _\varepsilon }}}(x))\tilde f(t)
\]
and set
\[
{G_\varepsilon }(x,t): = \int_0^t {{g_\varepsilon }(x,\tau )} d\tau,
\]
where
\[
\tilde f(t): = \min \Bigl\{ {\frac{1}
{2}{t^ + },f(t)} \Bigr\}.
\]
Similarly, we define
\[
g_\varepsilon ^k(x,t): = {\chi _{{{({\Lambda ^k})}_\varepsilon }}}(x)f(t) + (1 - {\chi _{{{({\Lambda ^k})}_\varepsilon }}}(x))\tilde f(t)
\]
and
\[
G_\varepsilon ^k(x,t): = \int_0^t {g_\varepsilon ^k(x,\tau )} d\tau  = {\chi _{{{({\Lambda ^k})}_\varepsilon }}}(x)F(t) + (1 - {\chi _{{{({\Lambda ^k})}_\varepsilon }}}(x))\tilde F(t),
\]
where $\tilde F(t) = \int_0^t {\tilde f(\tau )} d\tau $. Set the functionals  ${E_\varepsilon }$, $E_\varepsilon ^k$ $(k = 1,2,...K)$ by
\[
{E_\varepsilon }(v): = \frac{1}
{2}\int_{{\Omega _\varepsilon }} {|\nabla v{|^2}}  + \frac{1}
{2}\int_{{\Omega _\varepsilon }} {{v^2}}  - \int_{{\Omega _\varepsilon }} {{G_\varepsilon }(x,v)} ,~v \in H_0^1({\Omega _\varepsilon })
\]
and
\[
E_\varepsilon ^k(v): = \frac{1}
{2}\int_{{\Omega _\varepsilon }} {|\nabla v{|^2}}  + \frac{1}
{2}\int_{{\Omega _\varepsilon }} {{v^2}}  - \int_{{\Omega _\varepsilon }} {G_\varepsilon ^k(x,v)} ,~v \in H_0^1({\Omega _\varepsilon }).
\]
Inspired by \cite{bw,bj,bj1} but with some modifications, we define
\[
{\chi _\varepsilon }(x): = \left\{ \begin{array}{ll}
  0,&x \in {(\widetilde\Lambda )_\varepsilon }, \hfill \\
  {\varepsilon ^{ - 1}},&x \in {\Omega _\varepsilon }\backslash {(\widetilde\Lambda )_\varepsilon }, \hfill \\
\end{array}  \right.{\text{ }}{Q_\varepsilon }(v): = \left( {\int_{{\Omega _\varepsilon }} {{\chi _\varepsilon }{|v|^{{2^ * }}}}  - 1} \right)_ + ^2
\]
and
\[
\chi _\varepsilon ^k(x): = \left\{ \begin{array}{ll}
  0,&x \in {(\widetilde{{\Lambda ^k}})_\varepsilon }, \hfill \\
  {\varepsilon ^{ - 1}},&x \in {\Omega _\varepsilon }\backslash {(\widetilde{{\Lambda ^k}})_\varepsilon }, \hfill \\
\end{array}  \right.{\text{ }}Q_\varepsilon ^k(v): = \left( {\int_{{\Omega _\varepsilon }} {\chi _\varepsilon ^k|v{|^{{2^ * }}}}  - 1} \right)_ + ^2,
\]
where $\widetilde\Lambda : = \mathop  \cup \nolimits_{k = 1}^K \widetilde{{\Lambda ^k}}$. Finally, we define auxiliary functionals ${J_\varepsilon }$, $J_\varepsilon ^k$ $(k = 1,2,...K)$ on $H_0^1({\Omega _\varepsilon })$ by
\[{J_\varepsilon }(v): = {E_\varepsilon }(v) + {Q_\varepsilon }(v){\text{ and }}J_\varepsilon ^k(v): = E_\varepsilon ^k(v) + Q_\varepsilon ^k(v).\]

\noindent\textbf{Remark}: We note that, J. Byeon and Z.-Q. Wang \cite{bw} introduced the penalization functional by
\[
\Bigl( {{\varepsilon ^{ - 6/\mu }}\int_{{\mathbb{R}^N}\backslash {\Lambda _\varepsilon }} {{u^2}}  - 1} \Bigr)_ + ^\beta
\]
for some $\mu  > 0$, $1 < \beta  < q/2$ and $q \in (2,{2^*})$.  However, in the present paper, we construct multi-peak solutions to \eqref{1.1} with a critical nonlinearity $f$. In the process of ``gluing solution", we need to ensure that for $2 < p \le {2^ * }$,
$$
\int_{{\Omega _\varepsilon }\backslash {{(\widetilde\Lambda )}_\varepsilon }} {|{\xi _\varepsilon }(s){|^p}}  = o(1)
$$
(see Lemma~\ref{3.7.} for detail). For this purpose, we introduce ${Q_\varepsilon }$, $Q_\varepsilon ^k$ $(k = 1,2,...K)$ with some modifications. As far as we know, it seems the first attempt to adopt the nonlocal penalization method due to J. Byeon and Z-Q. Wang \cite{bw} to construct multi-peak solutions of singularly perturbed nonlinear Dirichlet problems. However, in the previous related works, only considered were solutions with concentration around local minimal, local maximal, saddle points and other type of  critical points to potential $V(x)$(See \cite{bj1,bt} and references therein).

\subsection{Local deformation}
For any $\Omega  \subset {\mathbb{R}^N}$ and $a > 0$, we denote
\[
{\Omega ^{ - a}}: = \{ {x \in \Omega :d(x,\partial \Omega ) > a} \}
\]
and
\[
{\Omega ^a}: = \{ x \in {\mathbb{R}^N}:d(x,\Omega ) \le a\}.
\]
For each $1 \le k \le K$, choosing a cut-off function ${\varphi _{\varepsilon ,k}}(x) \in C_c^\infty ({\mathbb{R}^N},[0,1])$ such that ${\varphi _{\varepsilon ,k}}(x) = 0$ for $x \notin (\widetilde{{\Lambda ^k}})_\varepsilon ^{ - 1/{\varepsilon ^{1/2}}}$, ${\varphi _{\varepsilon ,k}}(x) = 1$ for $x \in (\widetilde{{\Lambda ^k}})_\varepsilon ^{ - 2/{\varepsilon ^{1/2}}}$ and $|\nabla {\varphi _{\varepsilon ,k}}(x)| \le C{\varepsilon ^{1/2}}$. We will find a solution to \eqref{3.1} near the set
\[
{X_\varepsilon }: = \left\{ {\sum\limits_{k = 1}^K {{\varphi _{\varepsilon ,k}}(x){U^k}(x - ({z^k}/\varepsilon ))} :{z^k} \in  {{\Lambda ^k}} {\text{ and }}{U^k} \in S} \right\}.
\]
Moreover, for any $A \subset H_0^1({\Omega _\varepsilon })$ and $a>0$, we use the notation
\[
{A^a}: = \bigl\{ {u \in {H_0^1({\Omega _\varepsilon }) }:\mathop {\inf }\limits_{v \in A} {{\| {u - v} \|}_{{H_0^1({\Omega _\varepsilon })}}} \le a} \bigr\}.
\]
For each $1 \le k \le K$, choosing $U_*^k \in S$ and $z_*^k \in {\mathcal{M}^k}$, then for each $\varepsilon  > 0$, we define
\begin{equation}\label{3.2}
W_{\varepsilon ,t}^k(x): = {\varphi _{\varepsilon ,k}}(x)U_*^k((x - (z_*^k/\varepsilon ))/t),~t \in [0,+\infty),
\end{equation}
where we denote $W_{\varepsilon ,0}^k = \mathop {\lim }\nolimits_{t \to 0} W_{\varepsilon ,t}^k$ in $H_0^1({\Omega _\varepsilon })$ sense, then $W_{\varepsilon ,0}^k = 0$. Since for each $U \in S$,
\[
\begin{array}{ll}
  I(U(x/t))&\dis = I(U(x/t)) - \frac{1}
{N}P(U){t^N} \hfill \\
   &=\dis \Bigl( {\frac{1}
{2}{t^{N - 2}} - \frac{{N - 2}}
{{2N}}{t^N}} \Bigr)\int_{{\mathbb{R}^N}} {|\nabla U{|^2}}  \hfill \\
   &=\dis \Bigl( {\frac{1}
{2}{t^{N - 2}} - \frac{{N - 2}}
{{2N}}{t^N}} \Bigr)Nc, \hfill \\
\end{array}
\]
then we choose a $t_0 >0$ large which is independent of $U \in S$ such that $I(U(x/{t_0})) <  - 3$. Moreover, we see that
\begin{equation}\label{3.3}
J_\varepsilon ^k(W_{\varepsilon ,{t_0}}^k) = I(U_*^k(x/{t_0})) + o(1) <  - 2{\text{ for }}\varepsilon  > 0{\text{ small.}}
\end{equation}
We define
\[
\widetilde{{c_\varepsilon }}: = \mathop {\max }\limits_{s \in {{[0,1]}^K}} {J_\varepsilon }({\gamma _\varepsilon }(s)),
\]
where
\begin{equation}\label{3.4}
{\gamma _\varepsilon }(s): = \sum\limits_{k = 1}^K {W_{\varepsilon ,{s_k}{t_0}}^k} {\text{ for }}s \in ({s_1},{s_2},...,{s_K}) \in {[0,1]^K}.
\end{equation}
Furthermore, for each $1 \le k \le K$, we denote
\[
\widetilde{c_\varepsilon ^k}: = \mathop {\max }\limits_{t \in [0,{t_0}]} J_\varepsilon ^k(W_{\varepsilon ,t}^k).
\]
Similar to \cite{df1}, we have the following first estimates of $\widetilde{{c_\varepsilon }}$:
\begin{lemma}\label{3.1.}
(i) $\widetilde{{c_\varepsilon }} = \sum\limits_{k = 1}^K {\widetilde{c_\varepsilon ^k}}  = Kc + o(1)$;\\
(ii) $\mathop {\overline {\lim } }\limits_{\varepsilon  \to 0} \mathop {\max }\limits_{s \in \partial {{[0,1]}^k}} {J_\varepsilon }({\gamma _\varepsilon }(s)) \le Kc - \sigma$,\\
where $0 < \sigma  < c $ is a fixed number.
\end{lemma}
From \eqref{3.3}, we define the following minimax value of $J_\varepsilon ^k$
\[
c_\varepsilon ^k: = \mathop {\inf }\limits_{\gamma  \in \Gamma _\varepsilon ^k} \mathop {\max }\limits_{t \in [0,1]} J_\varepsilon ^k(\gamma (t)),
\]
where the set of paths is defined by
\[
\Gamma _\varepsilon ^k: = \{ \gamma (t) \in C([0,1],H_0^1({\Omega _\varepsilon })):\gamma (0) = 0{\text{ and }}\gamma (1) = W_{\varepsilon ,{t_0}}^k\}.
\]
Next, we give the following more delicate upper estimate of $c_\varepsilon ^k$.

\begin{lemma}\label{3.2.}

For any $\delta  \in (0,1)$, there exists ${C_\delta } > 0$ such that for $\varepsilon >0$ small, there holds
\[
c_\varepsilon ^k \le \widetilde{c_\varepsilon ^k} \le c + {C_\delta }\exp \Bigl( { - 2\delta  \cdot \frac{{{D_k}-2{\varepsilon ^{1/2}}}}
{{\varepsilon t_\varepsilon ^k}}} \Bigr),{\text{ }}k = 1,2,...,K,
\]
where ${t_\varepsilon ^k}>0$ and satisfies $t_\varepsilon ^k \to 1$ as $\varepsilon \to 0$.
\end{lemma}
\begin{proof}
By the exponential decay of $U_ * ^k$ and $\nabla U_ * ^k$, it is standard to prove this lemma. Similar results can be found in Proposition~4.1 of \cite{b}. But for readers' convenience and completeness, we give a detailed proof. Similar to \eqref{3.3}, we have
\begin{equation}\label{3.5}
c_\varepsilon ^k \le \widetilde{c_\varepsilon ^k} = \mathop {\max }\limits_{t \in [0,{t_0}]} I(U_*^k(x/t)) + o(1) = I(U_*^k) + o(1) = c + o(1).
\end{equation}
Letting $t_\varepsilon ^k \in [0,{t_0}]$ be such that
\[
\widetilde{c_\varepsilon ^k}: = \mathop {\max }\limits_{t \in [0,{t_0}]} J_\varepsilon ^k(W_{\varepsilon ,t}^k) = J_\varepsilon ^k(W_{\varepsilon ,t_\varepsilon ^k}^k),
\]
by \eqref{3.5}, we see that $t_\varepsilon ^k \to 1$ as $\varepsilon \to 0$. For any $\delta  \in (0,1)$, we choose $\delta ' = \delta '(\delta ) > 0$ such that $\delta  + \delta ' \in (0,1)$, by Proposition~\ref{2.1.} $(iv)$, we see that
\begin{equation}\label{add10}
U_ * ^k (x)+ |\nabla U_ * ^k(x)| \le {C_\delta }{e^{ - (\delta  + \delta ')|x|}}, x\in\mathbb{R}^N.
\end{equation}
In view of Proposition~\ref{2.1.} $(iii)$, \eqref{add10} and $(F_1)$, $(F_2)$, we see that
\begin{equation}\label{3.6}
\begin{gathered}
\quad  \int_{(\widetilde{{\Lambda ^k}} - z_*^k)_{\varepsilon t_\varepsilon ^k}^{ - 1/{\varepsilon ^{1/2}}t_\varepsilon ^k}\backslash (\widetilde{{\Lambda ^k}} - z_*^k)_{\varepsilon t_\varepsilon ^k}^{ - 2/{\varepsilon ^{1/2}}t_\varepsilon ^k}} {|\nabla U_*^k| + |U_*^k| + F(U_*^k)}  \hfill \\
   \le {C_\delta }\int_{(\widetilde{{\Lambda ^k}} - z_*^k)_{\varepsilon t_\varepsilon ^k}^{ - 1/{\varepsilon ^{1/2}}t_\varepsilon ^k}\backslash (\widetilde{{\Lambda ^k}} - z_*^k)_{\varepsilon t_\varepsilon ^k}^{ - 2/{\varepsilon ^{1/2}}t_\varepsilon ^k}} {{e^{ - 2(\delta  + \delta ')|x|}}}  \hfill \\
   \le {C_\delta }\exp \Bigl( { - 2(\delta  + \delta ')\cdot\Bigl( {\frac{{{D_k}}}
{{\varepsilon t_\varepsilon ^k}} - \frac{2}
{{\varepsilon ^{1/2}}{t_\varepsilon ^k}}} \Bigr)} \Bigr){\Bigl( {\frac{1}
{{\varepsilon t_\varepsilon ^k}}} \Bigr)^N} \hfill \\
   \le {C_\delta }\exp \Bigl( { - 2\delta \cdot\frac{{D_k}-2{\varepsilon ^{1/2}}}
{{\varepsilon t_\varepsilon ^k}}} \Bigr)\exp \Bigl( { - 2\delta '\cdot\frac{{D_k}-2{\varepsilon ^{1/2}}}
{{\varepsilon t_\varepsilon ^k}}} \Bigr){\Bigl( {\frac{1}
{{\varepsilon t_\varepsilon ^k}}} \Bigr)^N} \hfill \\
   \le {C_\delta }\exp \Bigl( { - 2\delta \cdot\frac{{D_k}-2{\varepsilon ^{1/2}}}
{{\varepsilon t_\varepsilon ^k}}} \Bigr) \hfill \\
\end{gathered}
\end{equation}
for $\varepsilon  \in (0,{\varepsilon _\delta })$ with some ${\varepsilon _\delta } > 0$. Note from $(F_1)$ that
\[
 \frac{1}
{2}|U_*^k{|^2} - F(U_*^k) > 0{\text{ for }}x \in {\mathbb{R}^N}\backslash (\widetilde{{\Lambda ^k}} - z_*^k)_{\varepsilon t_\varepsilon ^k}^{ - 2/{{\varepsilon ^{1/2}}}t_\varepsilon ^k}{\text{ and }}\varepsilon  > 0{\text{ small.}}
\]
then by \eqref{3.6}, we see that
\[\begin{gathered}
\quad  J_\varepsilon ^k(W_{\varepsilon ,t_\varepsilon ^k}^k) \hfill \\
   \le \int_{(\widetilde{{\Lambda ^k}} - z_*^k)_{\varepsilon t_\varepsilon ^k}^{ - 2/{\varepsilon ^{1/2}}t_\varepsilon ^k}}
 {\frac{1}
{2}|\nabla U_*^k| + \frac{1}
{2}|U_*^k| - F(U_*^k)}  + {C_\delta }\exp \Bigl( { - 2\delta  \cdot \frac{{D_k} - 2{\varepsilon ^{1/2}}}
{{\varepsilon t_\varepsilon ^k}}} \Bigr) \hfill \\
   \le I(U_*^k) + {C_\delta }\exp \Bigl( { - 2\delta  \cdot \frac{{D_k} - 2{\varepsilon ^{1/2}}}
{{\varepsilon t_\varepsilon ^k}}} \Bigr). \hfill \\
\end{gathered} \]
This proves the claim.
\end{proof}

\begin{proposition}\label{3.4.}
There exists ${d_0} > 0$ such that for any $\{ {\varepsilon _j}\} _{j = 1}^\infty $, $\{ {u_{{\varepsilon _j}}}\}  _{j = 1}^\infty$ satisfying $\mathop {\lim }\limits_{j \to \infty } {\varepsilon _j} = 0$, ${u_{{\varepsilon _j}}} \in X_{{\varepsilon _j}}^{{d_0}},\mathop {\lim }\limits_{j \to \infty } {J_{{\varepsilon _j}}}({u_{{\varepsilon _j}}}) \le Kc$ and $\mathop {\lim }\limits_{j \to \infty } {\| {{J'_{{\varepsilon _j}}}({u_{{\varepsilon _j}}})} \|_{{{(H_0^1({\Omega _{{\varepsilon _j}}}))}^{ - 1}}}} = 0$, there exist, up to a subsequence, $\{ y_{{\varepsilon _j}}^k\} _{j = 1}^\infty  \subset {({\Lambda ^k})_{{\varepsilon _j}}}$, ${U^k} \in S$ $(k = 1,2,...,K)$ such that
\[
\mathop {\lim }\limits_{j \to \infty } {\left\| {{u_{{\varepsilon _j}}} - \sum\limits_{k = 1}^K {{\varphi _{{\varepsilon _j},k}}(x){U^k}(x - y_{{\varepsilon _j}}^k)} } \right\|_{H_0^1({\Omega _{{\varepsilon _j}}})}} = 0.
\]
\end{proposition}

\begin{proof}
For notational simplicity, we write $\varepsilon $ for ${\varepsilon _j}$ and still use $\varepsilon $ after taking a subsequence. By the definition of $X_\varepsilon ^{{d_0}}$ and the compactness of  $S$, for each $1 \le k \le K$, there exist ${\{ x_\varepsilon ^k\} _{\varepsilon  > 0}} \subset {\Lambda ^k}$ and ${U^k} \in S$ such that

\begin{equation}\label{3.8}
{\left\| {{u_\varepsilon } - \sum\limits_{k = 1}^K {{\varphi _{\varepsilon ,k}}(x){U^k}(x - (x_\varepsilon ^k/\varepsilon ))} } \right\|_{H_0^1({\Omega _\varepsilon })}} \le 2{d_0}.
\end{equation}
In the following of the proof, we regard $u \in H_0^1({\Omega _\varepsilon })$ as a member of $u \in {H^1}({\mathbb{R}^N})$ by defining $u=0$ outside ${\Omega _\varepsilon }$ if necessary. Next, we divide the proof into the following steps.\\
\textbf{Step~1}. We claim that
\begin{equation}\label{3.9}
\mathop {\lim }\limits_{\varepsilon  \to 0} \mathop {\sup }\limits_{y \in {A_\varepsilon }} \int_{{B_1}(y)} {(u_\varepsilon ^ + )^{{2^*}}} = 0,
\end{equation}
where ${A_\varepsilon } = \mathop  \cup \nolimits_{k = 1}^K ((\widetilde{{\Lambda ^k}})_\varepsilon ^{ - 1/(2{\varepsilon ^{1/2}})}\backslash {B_{\beta /2\varepsilon }}(x_\varepsilon ^k/\varepsilon ))$ and $\beta >0$ is small but fixed.

Assuming on the contrary that there exists $r>0$ such that
\[
\mathop {\underline {\lim } }\limits_{\varepsilon  \to 0} \mathop {\sup }\limits_{y \in {A_\varepsilon }} \int_{{B_1}(y)} {(u_\varepsilon ^ + )^{{2^*}}}= 2r > 0,
\]
then there exists ${y_\varepsilon } \in {A_\varepsilon }$ such that for $\varepsilon  > 0$ small,
\begin{equation}\label{3.10}
\int_{{B_1}({y_\varepsilon })} {(u_\varepsilon ^ + )^{{2^*}}}  \ge r > 0.
\end{equation}
Letting ${v_\varepsilon }(x): = {u_\varepsilon }(x + {y_\varepsilon })$, up to a subsequence, there exist $v \in {H^1}({\mathbb{R}^N})$ and $1 \le {k_0} \le K$ such that ${v_\varepsilon } \rightharpoonup v$ in ${H^1}({\mathbb{R}^N})$ and $\varepsilon {y_\varepsilon } \to {y_0} \in \overline {\widetilde{{\Lambda ^{{k_0}}}}} $ as $\varepsilon  \to 0$. In the following, we consider two cases.

{\bf Case~1.} $v \ne 0$. We have the following subcases.

{\bf Case~1-1.} $\varepsilon {y_\varepsilon } \in {\Lambda ^{{k_0}}}$ and $\mathop {\lim }\nolimits_{\varepsilon  \to 0} d({y_\varepsilon },\partial {({\Lambda ^{{k_0}}})_\varepsilon }) =  + \infty $. We see that for any compact set $K$, $\varepsilon K + \varepsilon {y_\varepsilon } \subset {\Lambda ^{{k_0}}}$ for $\varepsilon > 0$ small. Thus, for each $\varphi (x) \in C_c^\infty ({\mathbb{R}^N})$, $\varphi (x - {y_\varepsilon }) \in C_c^\infty ({({\Lambda ^{{k_0}}})_\varepsilon })$ for $\varepsilon > 0$ small, then
\[
\mathop {\lim }\limits_{\varepsilon  \to 0} \left\langle {{J'_\varepsilon }({u_\varepsilon }),\varphi (x - {y_\varepsilon })} \right\rangle  = 0
\]
implies that $0 < v \in {H^1}({\mathbb{R}^N})$ and satisfies \eqref{2.1}. Using Pohozaev's identity \eqref{add12}, we see that
\begin{equation}\label{3.11}
I(v) = I(v) - \frac{1}
{N}P(v) = \frac{1}
{N}\int_{{\mathbb{R}^N}} {|\nabla v{|^2}}  \ge c.
\end{equation}

{\bf Case~1-2.} $\mathop {\lim }\nolimits_{\varepsilon  \to 0} d({y_\varepsilon },\partial {({\Lambda ^{{k_0}}})_\varepsilon }) <  + \infty $. Up to a translation and a rotation, we see that $v \in {H^1}({\mathbb{R}^N})$ and satisfies
\begin{equation}\label{3.12}
 - \Delta v + v = {\chi _{\mathbb{R}_ + ^N}}f(v) + (1 - {\chi _{\mathbb{R}_ + ^N}})\tilde f(v){\text{ in }}{\mathbb{R}^N}.
\end{equation}
The corresponding energy functional to \eqref{3.12} is
\[
\tilde I(v) = \frac{1}
{2}\int_{{\mathbb{R}^N}} {|\nabla v{|^2}}  + \frac{1}
{2}\int_{{\mathbb{R}^N}} {{v^2}}  - \int_{{\mathbb{R}^N}} {G(x,v)} ,{\text{ }}v \in {H^1}({\mathbb{R}^N}),
\]
where
\[
G(x,v) = {\chi _{\mathbb{R}_ + ^N}}(x)F(v) + (1 - {\chi _{\mathbb{R}_ + ^N}}(x))\tilde F(v){\text{ and }}\tilde F(v) = \int_0^v {\tilde f(\tau )} d\tau.
\]
Moreover, $v$ satisfies the following Pohozaev's type identity
\[
\tilde P(v) = \frac{{N - 2}}
{2}\int_{{\mathbb{R}^N}} {|\nabla v{|^2}}  + N\Bigl( {\frac{1}
{2}\int_{{\mathbb{R}^N}} {{v^2}}  - \int_{{\mathbb{R}^N}} {G(x,v)} } \Bigr) = 0.
\]
Now, we divide\textbf{ Case~1-2} into the following two subcases.

{\bf Case~1-2-1.} If $\tilde f(v(x)) < f(v(x))$ for some $x \in {\mathbb{R}^N}\backslash \mathbb{R}_ + ^N$, arguing as in Proposition~2.2 of \cite{b}, we see that
\[
\tilde I(v) = \tilde I(v) - \frac{1}
{N}\tilde P(v) = \frac{1}
{N}\int_{{\mathbb{R}^N}} {|\nabla v{|^2}}  > c.
\]

{\bf Case~1-2-2.} If $\tilde f(v(x)) = f(v(x))$ for all $x \in {\mathbb{R}^N}\backslash \mathbb{R}_ + ^N$, we see that $v$ satisfies \eqref{2.1} and
\[
I(v) = \frac{1}
{N}\int_{{\mathbb{R}^N}} {|\nabla v{|^2}}  \ge c.
\]
Next, we divide \textbf{Case~1-2-2} into the following two subcases.
\begin{itemize}
\item \textbf{Case~1-2-2-1.} $I(v) > c$.
\item \textbf{Case~1-2-2-2.} $I(v) = c$, then $v$ is a ground state solution to \eqref{2.1}, i.e. $\exists U \in S$, ${z_0} \in {\mathbb{R}^N}$ such that $v(x) = U(x - {z_0})$. Note that $ - \Delta U(0) \ge 0$, then $f(U(0)) \ge U(0) > \frac{1}
{2}U(0)$, it follows that, in fact, ${z_0} \in \mathbb{R}_ + ^N$.
\end{itemize}

{\bf Case~1-3.} $\varepsilon {y_\varepsilon } \in {(\widetilde{{\Lambda ^{{k_0}}}})^{ - 1/(2{\varepsilon ^{1/2}})}}\backslash {\Lambda ^{{k_0}}}$ and $\mathop {\lim }\nolimits_{\varepsilon  \to 0} d({y_\varepsilon },\partial {({\Lambda ^{{k_0}}})_\varepsilon }) =  + \infty $. We see that for any compact set $K$, $\varepsilon K + \varepsilon {y_\varepsilon } \subset \widetilde{{\Lambda ^{{k_0}}}}$ for $\varepsilon > 0$ small. Thus, for each $\varphi  \in C_c^\infty ({\mathbb{R}^N})$, it holds that
\[
\mathop {\lim }\limits_{\varepsilon  \to 0} \left\langle {{J'_\varepsilon }({u_\varepsilon }),\varphi (x - {y_\varepsilon })} \right\rangle  = 0,
\]
which implies that $0 \le v \in {H^1}({\mathbb{R}^N})$ and satisfies
\[
 - \Delta v + v = \tilde f(v){\text{ in }}{\mathbb{R}^N}.
\]
Note that $\tilde f(v) \le \frac{1}
{2}v$, we see that $v=0$, which contradicts $v \ne 0$. \textbf{Case~1-3} is impossible.

To sum up, we see that
\[
\int_{{\mathbb{R}^N}} {|\nabla v{|^2}}  \ge Nc,
\]
which implies that for some $R>0$ large but fixed,
\begin{align}\label{3.13}
\lim\limits_{\varepsilon  \to 0} \int_{{B_R}({y_\varepsilon })} {|\nabla {u_\varepsilon }{|^2}}  = \mathop {\underline {\lim } }\limits_{\varepsilon  \to 0} \int_{{B_R}(0)} {|\nabla {v_\varepsilon }{|^2}}
\ge \int_{{B_R}(0)} {|\nabla v{|^2}}  \ge \frac{1}
{2}\int_{{\mathbb{R}^N}} {|\nabla v{|^2}}  \ge \frac{N}
{2}c > 0.
\end{align}
On the other hand, by \eqref{3.8}, we have
\begin{equation}\label{3.14}
\begin{array}{ll}
  \dis\int_{{B_R}({y_\varepsilon })} {|\nabla {u_\varepsilon }{|^2}}  & \dis \le \int_{{B_R}({y_\varepsilon })} {|\nabla ({\varphi _{\varepsilon ,{k_0}}}(x){U^{{k_0}}}(x - (x_\varepsilon ^{{k_0}}/\varepsilon ))){|^2}}  + C{d_0} \hfill \\
   &\dis \le C\int_{{B_R}({y_\varepsilon } - (x_\varepsilon ^{{k_0}}/\varepsilon ))} {|{U^{{k_0}}}{|^2}}  + C\int_{{B_R}({y_\varepsilon } - (x_\varepsilon ^{{k_0}}/\varepsilon ))} {|\nabla {U^{{k_0}}}{|^2}}  + C{d_0} \hfill \\
   &\dis = C{d_0} + o(1), \hfill \\
\end{array}
\end{equation}
where $o(1) \to 0$ as $\varepsilon  \to 0$ and we have used the fact that  $|{y_\varepsilon } - (x_\varepsilon ^{{k_0}}/\varepsilon )| \ge \beta /2\varepsilon $. We see that \eqref{3.14} contradicts \eqref{3.13} if $d_0 >0$ is small.

{\bf Case~2.} $v = 0$, i.e. ${v_\varepsilon } \rightharpoonup 0$ in ${H^1}({\mathbb{R}^N})$ as $\varepsilon \to 0$. Let ${\varphi _0}(x) \in C_c^\infty ({\mathbb{R}^N},[0,1])$ satisfying ${\varphi _0}(x) = 1$ for $x \in {B_1}(0)$, ${\varphi _0}(x) = 0$ for $x \in {\mathbb{R}^N}\backslash {B_2}(0)$ and $|\nabla {\varphi _0}(x)| \le C$. By the concentration compactness principle II (Lemma~I.1 of \cite{l}), we obtain an at most countable index set $\Gamma$, sequence ${\{ {x_j}\} _{j \in \Gamma }} \subset {\mathbb{R}^N}$, and ${\{ {\nu _j}\} _{j \in \Gamma }} \subset (0,+\infty )$ such that
\[
{({\varphi _0}v_\varepsilon ^ + )^{{2^*}}} \rightharpoonup \sum\limits_{j \in \Gamma } {{\nu _j}{\delta _{{x_j}}}} {\text{ as }}\varepsilon  \to 0.
\]
Then, there is at least one ${j_0} \in \Gamma $ such that ${x_{{j_0}}} \in \overline {{B_1}(0)} $ with ${\nu_{{j_0}}} > 0$. Otherwise, ${v_\varepsilon ^ +} \to 0$ in ${L^{{2^ * }}}({B_1}(0))$, which contradicts \eqref{3.10}.

Inspired by \cite{s,y}, we define the concentration function
\[
{G_\varepsilon }(r) = \mathop {\sup }\limits_{x \in \overline {{B_1}(0)} } \int_{{B_r}(x)} {{{(v_\varepsilon ^ + )}^{{2^*}}}} .
\]
Fixing a small $\tau  \in (0,{S^{N/2}})$ and choosing ${\sigma _\varepsilon } = {\sigma _\varepsilon }(\tau ) > 0$, ${z_\varepsilon } \in \overline {{B_1}(0)} $ such that
\begin{equation}\label{3.15}
\int_{{B_{{\sigma _\varepsilon }}}({z_\varepsilon })} {{{(v_\varepsilon ^ + )}^{{2^*}}}}  = {G_\varepsilon }({\sigma _\varepsilon }) = \tau .
\end{equation}
Letting ${w_\varepsilon }(y) = \sigma _\varepsilon ^{(N - 2)/2}v_\varepsilon ({\sigma _\varepsilon }y + {z_\varepsilon })$, we see that
\begin{equation}\label{3.16}
{{\tilde G}_\varepsilon }(r): = \mathop {\sup }\limits_{x \in \overline {{B_{1/{\sigma _\varepsilon }}}( - {z_\varepsilon }/{\sigma _\varepsilon })} } \int_{{B_r}(x)} {(w_\varepsilon ^ + )^{{2^*}}}
  = \mathop {\sup }\limits_{x \in \overline {{B_1}(0)} } \int_{{B_{{\sigma _\varepsilon }r}}(x)} {{{(v_\varepsilon ^ + )}^{{2^*}}}}  = {G_\varepsilon }({\sigma _\varepsilon }r).
\end{equation}
\eqref{3.15} and \eqref{3.16} imply that
\begin{equation}\label{3.17}
{{\tilde G}_\varepsilon }(1) = \int_{{B_1}(0)} {(w_\varepsilon ^ + )^{{2^*}}}  = \int_{{B_{{\sigma _\varepsilon }}}({z_\varepsilon })} {{{(v_\varepsilon ^ + )}^{{2^*}}}}  = {G_\varepsilon }({\sigma _\varepsilon }) = \tau .
\end{equation}

Next, we prove that there is $\tau  \in (0,{S^{N/2}})$ small, such that, up to a subsequence, ${\sigma _\varepsilon }(\tau ) \to 0$ as $\varepsilon  \to 0 $. Otherwise, for any $\delta  > 0$, there exists ${M_\delta } > 0$ such that ${\sigma _\varepsilon }(\delta ) \ge {M_\delta }$, then
\[
\int_{{B_{{M_\delta }}}({x_{{j_0}}})} {{{({\varphi _0}v_\varepsilon ^ + )}^{{2^*}}}}  \le \mathop {\sup }\limits_{x \in \overline {{B_1}(0)} } \int_{{B_{{\sigma _\varepsilon }(\delta )}}(x)} {{{(v_\varepsilon ^ + )}^{{2^*}}}}  = {G_\varepsilon }({\sigma _\varepsilon }(\delta )) = \delta .
\]
In particular,
\begin{equation}\label{3.18}
{\nu _{{j_0}}} \le \int_{{B_{{M_\delta }}}({x_{{j_0}}})} {{{({\varphi _0}v_\varepsilon ^ + )}^{{2^*}}}}  + o(1) \le \delta  + o(1),
\end{equation}
where $o(1) \to 0$ as $\varepsilon \to 0 $. Letting $\varepsilon \to 0 $ and $\delta  \to 0$ in \eqref{3.18}, we see that ${\nu _{{j_0}}} \le 0$, which contradicts ${\nu _{{j_0}}} > 0$.

Since
\[
\int_{{\mathbb{R}^N}} {|\nabla {w_\varepsilon}{|^2}}  = \int_{{\mathbb{R}^N}} {|\nabla v_\varepsilon {|^2}}
\]
and ${\{ {v_\varepsilon }\} _{\varepsilon  > 0}}$ is bounded in  ${H^1}({\mathbb{R}^N})$, up to a subsequence, there exists $ w \in {D^{1,2}}({\mathbb{R}^N})$ such that
\begin{equation}\label{3.19}
{w_\varepsilon } \rightharpoonup w{\text{ in }}{D^{1,2}}({\mathbb{R}^N}).
\end{equation}
For any $\varphi  \in C_c^\infty ({\mathbb{R}^N})$, we denote ${{\tilde \varphi }_\varepsilon }(x): = \sigma _\varepsilon ^{(2 - N)/2}\varphi ((x - {z_\varepsilon })/{\sigma _\varepsilon })$. Note that ${\sigma _\varepsilon} \to 0$ and ${z_\varepsilon } \in \overline {{B_1}(0)} $ imply that ${{\tilde \varphi }_\varepsilon }(x) \in C_c^\infty ({B_2}(0))$ for $\varepsilon > 0$ small. Then, we see from the fact ${J'_\varepsilon }({u_\varepsilon }) \to 0$ as $\varepsilon \to 0$ that
\begin{equation}\label{add3}
\left\langle {{J'_\varepsilon }({u_\varepsilon }),{{\tilde \varphi }_\varepsilon }(x - {y_\varepsilon })} \right\rangle  = o(1){\| {{{\tilde \varphi }_\varepsilon }(x - {y_\varepsilon })} \|_{{H^1}({\mathbb{R}^N})}} = o(1){\| \varphi  \|_{{H^1}({\mathbb{R}^N})}},
\end{equation}
where
\begin{equation}\label{add4}
\begin{gathered}
 \quad \left\langle {{J'_\varepsilon }({u_\varepsilon }),{{\tilde \varphi }_\varepsilon }(x - {y_\varepsilon })} \right\rangle  \hfill \\
   = \int_{{\Omega _\varepsilon } - {y_\varepsilon }} {\nabla {v_\varepsilon } \cdot \nabla {{\tilde \varphi }_\varepsilon }}  + \int_{{\Omega _\varepsilon } - {y_\varepsilon }} {{v_\varepsilon }{{\tilde \varphi }_\varepsilon }}  \hfill \\
\quad   - \int_{{\Omega _\varepsilon } - {y_\varepsilon }} {{\chi _{{\Lambda _\varepsilon }}}(x + {y_\varepsilon })f({v_\varepsilon }){{\tilde \varphi }_\varepsilon }}  - \int_{{\Omega _\varepsilon } - {y_\varepsilon }} {(1 - {\chi _{{\Lambda _\varepsilon }}}(x + {y_\varepsilon }))\tilde f({v_\varepsilon }){{\tilde \varphi }_\varepsilon }}. \hfill \\
\end{gathered}
\end{equation}
We note that
\begin{equation}\label{add5}
\begin{array}{ll}
\dis  \Bigl| {\int_{{\Omega _\varepsilon } - {y_\varepsilon }} {{v_\varepsilon }{{\tilde \varphi }_\varepsilon }} } \Bigr| &\le\dis {\Bigl( {\int_{{\Omega _\varepsilon } - {y_\varepsilon }} {|{v_\varepsilon }{|^2}} } \Bigr)^{1/2}}{\Bigl( {\int_{{\Omega _\varepsilon } - {y_\varepsilon }} {|{{\tilde \varphi }_\varepsilon }{|^2}} } \Bigr)^{1/2}} \hfill \\
  &\dis \le C{\sigma _\varepsilon }{\Bigl( {\int_{{\Omega _{\varepsilon {\sigma _\varepsilon }}} - (({y_\varepsilon } + {z_\varepsilon })/\varepsilon )} {|\varphi {|^2}} } \Bigr)^{1/2}} \le o(1){\| \varphi  \|_{{H^1}({\mathbb{R}^N})}}, \hfill \\
\end{array}
\end{equation}
Similar to \eqref{add5} and by the definition of ${\tilde f}$ ,
\begin{equation}\label{add6}
\Bigl| {\int_{{\Omega _\varepsilon } - {y_\varepsilon }} {(1 - {\chi _{{\Lambda _\varepsilon }}}(x + {y_\varepsilon }))\tilde f({v_\varepsilon }){{\tilde \varphi }_\varepsilon }} } \Bigr| \le \frac{1}
{2}\int_{{\Omega _\varepsilon } - {y_\varepsilon }} {|{v_\varepsilon }{{\tilde \varphi }_\varepsilon }|}  \le o(1){\| \varphi  \|_{{H^1}({\mathbb{R}^N})}}.
\end{equation}
By $(F_1)$ and $(F_2)$, we see that for any $\delta > 0$, there exists ${C_\delta } > 0$ such that
\begin{equation}\label{add7}
\begin{array}{ll}
\dis  \Bigl| {\int_{{\Omega _\varepsilon } - {y_\varepsilon }} {{\chi _{{\Lambda _\varepsilon }}}(x + {y_\varepsilon })(f({v_\varepsilon }) - {{(v_\varepsilon ^ + )}^{{2^ * } - 1}}){{\tilde \varphi }_\varepsilon }} } \Bigr| &\le \dis \delta \int_{{\Omega _\varepsilon } - {y_\varepsilon }} {|{v_\varepsilon }{|^{{2^ * } - 1}}|{{\tilde \varphi }_\varepsilon }|}  + {C_\delta }\int_{{\Omega _\varepsilon } - {y_\varepsilon }} {|{v_\varepsilon }{{\tilde \varphi }_\varepsilon }|}  \hfill \\
   &\le \dis (C\delta  + {C_\delta } \cdot o(1)){\| \varphi  \|_{{H^1}({\mathbb{R}^N})}}. \hfill \\
\end{array}
\end{equation}
Letting $\varepsilon \to 0$ and $\delta \to 0$, we get
\begin{equation}\label{add8}
\Bigl| {\int_{{\Omega _\varepsilon } - {y_\varepsilon }} {{\chi _{{\Lambda _\varepsilon }}}(x + {y_\varepsilon })(f({v_\varepsilon }) - {{(v_\varepsilon ^ + )}^{{2^ * } - 1}}){{\tilde \varphi }_\varepsilon }} } \Bigr| = o(1){\| \varphi  \|_{{H^1}({\mathbb{R}^N})}}.
\end{equation}
By \eqref{add3}, \eqref{add4}, \eqref{add5}, \eqref{add6}, \eqref{add7} and \eqref{add8}, we see that
\[
\int_{{\Omega _{\varepsilon {\sigma _\varepsilon }}} - (({y_\varepsilon } + {z_\varepsilon })/\varepsilon )} {\nabla {w_\varepsilon } \cdot \nabla \varphi }  - \int_{{\Omega _{\varepsilon {\sigma _\varepsilon }}} - (({y_\varepsilon } + {z_\varepsilon })/\varepsilon )} {{\chi _{{\Lambda _\varepsilon }}}({\sigma _\varepsilon }x + {y_\varepsilon } + {z_\varepsilon }){{(w_\varepsilon ^ + )}^{{2^*} - 1}}\varphi }  = o(1){\left\| \varphi  \right\|_{{H^1}({\mathbb{R}^N})}}.
\]
Now, we divide \textbf{Case~2} into the following three subcases:\\
\begin{itemize}
\item \textbf{Case~2-1.} $\varepsilon {y_\varepsilon } \in {\Lambda ^{{k_0}}}$ and $\mathop {\lim }\nolimits_{\varepsilon  \to 0} d({y_\varepsilon },\partial {({\Lambda ^{{k_0}}})_\varepsilon }) =  + \infty $. We see that  $0 \le w \in {D^{1,2}}({\mathbb{R}^N})$ and satisfies
\begin{equation}\label{3.20}
 - \Delta w = |w{|^{{2^ * } - 2}}w{\text{ in }}{\mathbb{R}^N}.
\end{equation}
\item \textbf{Case~2-2.} $\mathop {\lim }\nolimits_{\varepsilon  \to 0} d({y_\varepsilon },\partial {({\Lambda ^{{k_0}}})_\varepsilon }) <  + \infty $. Up to translation and rotation, we see that $0 \le w \in {D^{1,2}}({\mathbb{R}^N})$ and satisfies
\[
\left\{ \begin{array}{ll}
   - \Delta w = |w{|^{{2^*} - 2}}w&{\text{ in }}\mathbb{R}_ + ^N, \hfill \\
  w = 0&{\text{ on }}\partial \mathbb{R}_ + ^N. \hfill \\
\end{array}  \right.
\]
Due to Pohozaev's non-existence result \cite{p} (see also \cite{d}), $w \equiv 0$ on ${\mathbb{R}^N}$.
\noindent\item \textbf{Case~2-3.} $\varepsilon {y_\varepsilon } \in {(\widetilde{{\Lambda ^{{k_0}}}})^{ - 1/(2{\varepsilon ^{1/2}})}}\backslash {\Lambda ^{{k_0}}}$ and $\mathop {\lim }\nolimits_{\varepsilon  \to 0} d({y_\varepsilon },\partial {({\Lambda ^{{k_0}}})_\varepsilon }) =  + \infty $. We see that $0 \le w \in {D^{1,2}}({\mathbb{R}^N})$ and satisfies
$ - \Delta w = 0$ in ${\mathbb{R}^N}$, i.e. $w \equiv 0$ on ${\mathbb{R}^N}$.
\end{itemize}
Next, we claim that,
\begin{equation}\label{3.21}
w_\varepsilon ^ +  \to w {\text{ in }}{L^{{2^*}}}({B_1}(0)).
\end{equation}
For a fixed ${r_0} > 1$, let $\varphi(x)  \in C_c^\infty ({\mathbb{R}^N},[0,1])$ satisfy $\varphi (x) = 1$ for $x \in {B_{{r_0}}}(0)$ and ${\text{supp}}\varphi (x) \subset {B_{2{r_0}}}(0)$. By \eqref{3.19}, we may assume that
\[
|\nabla (\varphi w_\varepsilon ^ + ){|^2} \rightharpoonup |\nabla (\varphi w){|^2} + \mu {\text{ and }}{(\varphi w_\varepsilon ^ + )^{{2^*}}} \rightharpoonup {(\varphi w)^{{2^*}}} + \nu ,
\]
where $\mu$ and $\nu$ are two bounded nonnegative measures on ${\mathbb{R}^N}$.  By the concentration compactness principle II (Lemma I.1 of \cite{l}), we obtain an at most countable index set $\Gamma $, sequences ${\{ {x_j}\} _{j \in \Gamma }} \subset {\mathbb{R}^N}$ and ${\{ {\mu _j}\} _{j \in \Gamma }},{\{ {\nu _j}\} _{j \in \Gamma }} \subset (0,\infty )$ such that
\begin{equation}\label{3.22}
\mu  \ge \sum\limits_{j \in \Gamma } {{\mu _j}} {\delta _{{x_j}}},\nu  = \sum\limits_{j \in \Gamma } {{\nu _j}} {\delta _{{x_j}}}{\text{ and }}S{({\nu _j})^{(N - 2)/N}} \le {\mu _j}.
\end{equation}
To prove \eqref{3.21}, it suffices to show that ${\{ {x_j}\} _{j \in \Gamma }} \cap \overline {{B_1}(0)}  = \emptyset $. Suppose that there is ${x_{j'}} \in \overline {{B_1}(0)} $ for some ${j'} \in \Gamma $ and define the function ${\phi _\rho }(x): = \phi ( {(x - {x_{{j'}}})/\rho } )$, where $\rho>0$ and $\phi(x)  \in C_c^\infty ({\mathbb{R}^N},[0,1])$ such that $\phi (x) = 1$ for $x \in {B_1}(0)$ and ${\text{supp}}\phi (x) \subset {B_2}(0)$ and$|{\nabla \phi }(x) | \le C$. We denote ${{\tilde \phi }_{\rho ,\varepsilon }}(x): = {\phi _\rho }((x - {z_\varepsilon })/{\sigma _\varepsilon })$, since ${z_\varepsilon } \in \overline {{B_1}(0)} $, ${x_{{j'}}} \in \overline {{B_1}(0)} $ and ${\sigma _\varepsilon} \to 0$ as $\varepsilon \to 0 $, we see that for $\varepsilon > 0$ small, ${\text{supp}}{{\tilde \phi }_{\rho ,\varepsilon }} \subset {B_{2{\sigma _\varepsilon }\rho }}({z_\varepsilon } + {\sigma _\varepsilon }{x_{j'}}) \subset {B_2}(0)$, then ${{\tilde \phi }_{\rho ,\varepsilon }}{v_\varepsilon } \in H_0^1({B_2}(0))$. Direct computations show that
\[
\int_{{\mathbb{R}^N}} {|{{\tilde \phi }_{\rho ,\varepsilon }}{v_\varepsilon }{|^2}}  \le \int_{{\mathbb{R}^N}} {|{v_\varepsilon }{|^2}}  \le C
\]
and
\[\begin{gathered}
\quad  \int_{{\mathbb{R}^N}} {|\nabla ({{\tilde \phi }_{\rho ,\varepsilon }}{v_\varepsilon }){|^2}}  \hfill \\
   \le C\int_{{\mathbb{R}^N}} {|\nabla {{\tilde \phi }_{\rho ,\varepsilon }}{|^2}|{v_\varepsilon }{|^2}}  + C\int_{{\mathbb{R}^N}} {|{{\tilde \phi }_{\rho ,\varepsilon }}{|^2}|\nabla {v_\varepsilon }{|^2}}  \hfill \\
   \le C{\Bigl( {\int_{{B_{2{\sigma _\varepsilon }\rho }}({z_\varepsilon } + {\sigma _\varepsilon }{x_{j'}})} {|\nabla {{\tilde \phi }_{\rho ,\varepsilon }}{|^N}} } \Bigr)^{2/N}}{\Bigl( {\int_{{\mathbb{R}^N}} {|{v_\varepsilon }{|^{{2^*}}}} } \Bigr)^{(N - 2)/N}} + C\int_{{\mathbb{R}^N}} {|\nabla {v_\varepsilon }{|^2}}  \le C. \hfill \\
\end{gathered} \]
Thus ${\{ {{\tilde \phi }_{\rho ,\varepsilon }}{v_\varepsilon }\} _{\varepsilon  > 0}}$ is bounded in ${H^1}({\mathbb{R}^N})$ uniformly for $\rho $, hence
$$\langle {{J'_\varepsilon }({u_\varepsilon }),{{\tilde \phi }_{\rho ,\varepsilon }}(x - {y_\varepsilon }){u_\varepsilon }} \rangle  = o(1),$$
then we get
\begin{equation}\label{3.23}
\begin{gathered}
\quad  \int_{{\Omega _\varepsilon } - {y_\varepsilon }} {|\nabla {v_\varepsilon }{|^2}{{\tilde \phi }_{\rho ,\varepsilon }}}  + \int_{{{\Omega _\varepsilon } - {y_\varepsilon }}} {(\nabla {v_\varepsilon }\cdot\nabla {{\tilde \phi }_{\rho ,\varepsilon }}){v_\varepsilon }}  + \int_{{\Omega _\varepsilon } - {y_\varepsilon }} {|{v_\varepsilon }{|^2}{{\tilde \phi }_{\rho ,\varepsilon }}}  \hfill \\
   = \int_{{\Omega _\varepsilon } - {y_\varepsilon }} {{\chi _{{\Lambda _\varepsilon }}}(x + {y_\varepsilon })f({v_\varepsilon }){v_\varepsilon }{{\tilde \phi }_{\rho ,\varepsilon }}}  + \int_{{\Omega _\varepsilon } - {y_\varepsilon }} {(1 - {\chi _{{\Lambda _\varepsilon }}}(x + {y_\varepsilon }))\tilde f({v_\varepsilon }){v_\varepsilon }{{\tilde \phi }_{\rho ,\varepsilon }}}  + o(1). \hfill \\
\end{gathered}
\end{equation}
Note that
\begin{equation}\label{3.24}
\mathop {\overline {\lim } }\limits_{\varepsilon  \to 0} \int_{{\Omega _\varepsilon } - {y_\varepsilon }} {|{v_\varepsilon }{|^2}{{\tilde \phi }_{\rho ,\varepsilon }}}  \le \mathop {\overline {\lim } }\limits_{\varepsilon  \to 0} \sigma _\varepsilon ^2\int_{{B_{2\rho }}({x_{j'}})} {|{w_\varepsilon }{|^2}{\phi _\rho }}  = 0.
\end{equation}
Similar to \eqref{3.24}, we see from the definition of ${\tilde f}$ that
\[
\mathop {\overline {\lim } }\limits_{\varepsilon  \to 0} \int_{{\Omega _\varepsilon } - {y_\varepsilon }} {(1 - {\chi _{{\Lambda _\varepsilon }}}(x + {y_\varepsilon }))\tilde f({v_\varepsilon }){v_\varepsilon }{{\tilde \phi }_{\rho ,\varepsilon }}}  = 0.
\]
We see from $(F_1)$ and $(F_2)$ that for any $\delta > 0$, there exists ${C_\delta } > 0$ such that
\[\begin{gathered}
\quad  \Bigl| {\int_{{\Omega _\varepsilon } - {y_\varepsilon }} {{\chi _{{\Lambda _\varepsilon }}}(x + {y_\varepsilon })(f({v_\varepsilon }){v_\varepsilon } - {{(v_\varepsilon ^ + )}^{{2^*} }}){{\tilde \phi }_{\rho ,\varepsilon }}} } \Bigr| \hfill \\
   \le \delta \int_{{B_{2{\sigma _\varepsilon }\rho }}({z_\varepsilon } + {\sigma _\varepsilon }{x_{j'}})} {|{v_\varepsilon }{|^{{2^*}}}{{\tilde \phi }_{\rho ,\varepsilon }}}  + {C_\delta }\int_{{B_{2{\sigma _\varepsilon }\rho }}({z_\varepsilon } + {\sigma _\varepsilon }{x_{j'}})} {|{v_\varepsilon }{|^2}{{\tilde \phi }_{\rho ,\varepsilon }}}  \hfill \\
   \le  C\delta  + {C_\delta }\sigma _\varepsilon ^2\int_{{B_{2\rho }}({x_{j'}})} {|{w_\varepsilon }{|^2}{\phi _\rho }}  \hfill \\
\end{gathered} \]
Letting $\varepsilon \to 0$ and $\delta \to 0$, we get
\[
\mathop {\lim }\limits_{\varepsilon  \to 0} \int_{{\Omega _\varepsilon } - {y_\varepsilon }} {{\chi _{{\Lambda _\varepsilon }}}(x + {y_\varepsilon })(f({v_\varepsilon }){v_\varepsilon } - {{(v_\varepsilon ^ + )}^{{2^*} }}){{\tilde \phi }_{\rho ,\varepsilon }}}  = 0.
\]
Moreover,
\[\begin{gathered}
\quad \mathop {\overline {\lim } }\limits_{\rho  \to 0} \mathop {\overline {\lim } }\limits_{\varepsilon  \to 0} \Bigl| {\int_{{\Omega _\varepsilon } - {y_\varepsilon }} {(\nabla {v_\varepsilon }\cdot\nabla {{\tilde \phi }_{\rho ,\varepsilon }}){v_\varepsilon }} } \Bigr| \hfill \\
   \le \mathop {\overline {\lim } }\limits_{\rho  \to 0} \mathop {\overline {\lim } }\limits_{\varepsilon  \to 0} {\Bigl( {\int_{{\Omega _\varepsilon } - {y_\varepsilon }} {|\nabla {v_\varepsilon }{|^2}} } \Bigr)^{1/2}}{\Bigl( {\int_{{B_{2{\sigma _\varepsilon }\rho }}({z_\varepsilon } + {\sigma _\varepsilon }{x_{j'}})} {|{v_\varepsilon }{|^2}|\nabla {{\tilde \phi }_{\rho ,\varepsilon }}{|^2}} } \Bigr)^{1/2}} \hfill \\
   \le C\mathop {\overline {\lim } }\limits_{\rho  \to 0} \mathop {\overline {\lim } }\limits_{\varepsilon  \to 0} {\Bigl( {\int_{{B_{2\rho }}({x_{j'}})} {|{w_\varepsilon }{|^2}|\nabla {\phi _\rho }{|^2}} } \Bigr)^{1/2}} \hfill \\
   = C\mathop {\overline {\lim } }\limits_{\rho  \to 0} {\Bigl( {\int_{{B_{2\rho }}({x_{j'}})} {w^2}|\nabla {\phi _\rho }{|^2}} \Bigr)^{1/2}} \hfill \\
   \le C\mathop {\overline {\lim } }\limits_{\rho  \to 0} {\Bigl( {\int_{{B_{2\rho }}({x_{j'}})} {|\nabla {\phi _\rho }{|^N}} } \Bigr)^{1/N}}{\Bigl( {\int_{{B_{2\rho }}({x_{j'}})} {{w^{{2^*}}}} } \Bigr)^{1/{2^ * }}} = 0. \hfill \\
\end{gathered} \]
We note that ${\text{supp}}{\phi _\rho } \subset {B_{2\rho }}({x_{j'}}) \subset {B_{{r_0}}}(0)$ for $\rho > 0$ small, by \eqref{3.22},
\[
\mathop {\overline {\lim } }\limits_{\rho  \to 0} \mathop {\overline {\lim } }\limits_{\varepsilon  \to 0} \int_{{\Omega _\varepsilon } - {y_\varepsilon }} {|\nabla {v_\varepsilon }{|^2}{{\tilde \phi }_{\rho ,\varepsilon }}}  \ge \mathop {\overline {\lim } }\limits_{\rho  \to 0} \mathop {\overline {\lim } }\limits_{\varepsilon  \to 0} \int_{{\mathbb{R}^N}} {|\nabla w_\varepsilon ^ + {|^2}{\phi _\rho }}  = \mathop {\overline {\lim } }\limits_{\rho  \to 0} \Bigl( {\int_{{\mathbb{R}^N}} {|\nabla w{|^2}{\phi _\rho }}  + \int_{{\mathbb{R}^N}} {\mu {\phi _\rho }} } \Bigr) \ge {\mu _{j'}}
\]
and
\[
\begin{array}{ll}
  \mathop  {\overline {\lim } }\limits_{\rho  \to 0} \mathop {\overline {\lim } }\limits_{\varepsilon  \to 0}\dis \int_{{\Omega _\varepsilon } - {y_\varepsilon }} {{\chi _{{\Lambda _\varepsilon }}}(x + {y_\varepsilon }){{(v_\varepsilon ^ + )}^{{2^*}}}{{\tilde \phi }_{\rho ,\varepsilon }}}  &\le \mathop {\overline {\lim } }\limits_{\rho  \to 0} \mathop {\overline {\lim } }\limits_{\varepsilon  \to 0} \dis\int_{{\Omega _\varepsilon } - {y_\varepsilon }} {{{(w_\varepsilon ^ + )}^{{2^*}}}{\phi _\rho }}  \hfill \\
   &= \mathop {\overline {\lim } }\limits_{\rho  \to 0}\dis \Bigl( {\int_{{\mathbb{R}^N}} {w^{{2^*}}{\phi _\rho }}  + \int_{{\mathbb{R}^N}} {\nu {\phi _\rho }} } \Bigr) = {\nu _{j'}}. \hfill \\
\end{array}
\]
By \eqref{3.23}, ${\mu _{j'}} \le {\nu _{j'}}$, then by \eqref{3.22}, ${\nu _{j'}} \ge {S^{N/2}}$. From \eqref{3.17}, we see that
\begin{equation}\label{3.25}
{S^{N/2}} \le {\nu _{j'}} \le \int_{{B_1}(0)} {{{(w_\varepsilon ^ + )}^{{2^*}}}}  + o(1) = \tau  + o(1),
\end{equation}
where $o(1) \to 0$ as $\varepsilon  \to 0$. Letting $\varepsilon  \to 0$ in \eqref{3.25}, we see that ${S^{N/2}} \le \tau $ which contradicts $\tau < {S^{N/2}}$. Hence, ${\{ {x_j}\} _{j \in \Gamma }} \cap \overline {{B_1}(0)}  = \emptyset $, then \eqref{3.21} holds. \eqref{3.17} and \eqref{3.21} imply that
\begin{equation}\label{add17}
\int_{{B_1}(0)} {{w^{{2^*}}}}  = \mathop {\lim }\limits_{\varepsilon  \to 0} \int_{{B_1}(0)} {{{(w_\varepsilon ^ + )}^{{2^*}}}}  = \tau  > 0,
\end{equation}
which means that $w$ is nontrivial. Hence, neither {\bf Case~2-2} nor {\bf Case~2-3} holds and only {\bf Case~2-1} is ture. Thus, $w$ is a positive solution to \eqref{3.20} and satisfies
\begin{equation}\label{3.26}
\int_{{\mathbb{R}^N}} {|\nabla w{|^2}}  = {S^{N/2}} = \int_{{\mathbb{R}^N}} {|w{|^{{2^*}}}} ,
\end{equation}
then for some $R > 0$ such that
\begin{equation}\label{3.27}
\int_{{B_R}(0)} {|\nabla w{|^2}}  \ge \frac{1}
{2}\int_{{\mathbb{R}^N}} {|\nabla w{|^2}}  = \frac{1}
{2}{S^{N/2}} > 0.
\end{equation}
On the other hand,
\begin{equation}\label{3.28}
\begin{array}{ll}
\dis  \int_{{B_R}(0)} {|\nabla w{|^2}} & \dis \le \mathop {\underline {\lim } }\limits_{\varepsilon  \to 0} \int_{{B_R}(0)} {|\nabla {w_\varepsilon^+ }{|^2}}  \le \mathop {\underline {\lim } }\limits_{\varepsilon  \to 0} \int_{{B_{{\sigma _\varepsilon }R}}({z_\varepsilon })} {|\nabla v_\varepsilon {|^2}}  \hfill \\
  & \dis = \mathop {\underline {\lim } }\limits_{\varepsilon  \to 0} \int_{{B_{{\sigma _\varepsilon }R}}({z_\varepsilon } + {y_\varepsilon })} {|\nabla u_\varepsilon  {|^2}}  \le \mathop {\underline {\lim } }\limits_{\varepsilon  \to 0} \int_{{B_2}({y_\varepsilon })} {|\nabla {u_\varepsilon }{|^2}} . \hfill \\
\end{array}
\end{equation}
Similar to \eqref{3.14}, we see that \eqref{3.27} and \eqref{3.28} lead to a contradiction for ${d_0} > 0$ small. Hence, \eqref{3.9} holds. We note that
\[
\mathop {\sup }\limits_{y \in {A_\varepsilon }} \int_{{B_1}(y)} {{{(u_\varepsilon ^ + )}^{{2^*}}}}  \ge \mathop {\sup }\limits_{y \in {\mathbb{R}^N}} \int_{{B_1}(y)} {{{({\varphi _{A_\varepsilon ^1}}u_\varepsilon ^ + )}^{{2^*}}}} ,
\]
where ${\varphi _{A_\varepsilon ^1}} \in C_c^\infty ({\mathbb{R}^N},[0,1])$ satisfies ${\varphi _{A_\varepsilon ^1}} = 1$ on $\mathop  \cup \nolimits_{k = 1}^K \bigl( {(\widetilde{{\Lambda ^k}})_\varepsilon ^{ - 1/(2{\varepsilon ^{1/2}}) - 2}\backslash {B_{(\beta /2\varepsilon ) + 2}}(x_\varepsilon ^k/\varepsilon )} \bigr)$, ${\text{supp}}{\varphi _{A_\varepsilon ^1}} \subset \mathop  \cup \nolimits_{k = 1}^K \bigl( {(\widetilde{{\Lambda ^k}})_\varepsilon ^{- 1/(2{\varepsilon ^{1/2}}) - 1}\backslash {B_{(\beta /2\varepsilon ) + 1}}(x_\varepsilon ^k/\varepsilon )} \bigr)$ and $|\nabla {\varphi _{A_\varepsilon ^1}}| \le C$. And since ${\{ {\varphi _{A_\varepsilon ^1}}u_\varepsilon  \} _{\varepsilon  > 0}}$ is bounded in ${H^1}({\mathbb{R}^N})$, we see from \eqref{3.9} and Lemma~\ref{2.2.} that
\[
\mathop {\lim }\limits_{\varepsilon  \to 0} \int_{{\mathbb{R}^N}} {{{({\varphi _{A_\varepsilon ^1}}u_\varepsilon ^ + )}^{{2^*}}}}  = 0,
\]
then, it follows from the interpolation inequality that
\begin{equation}\label{3.29}
\mathop {\lim }\limits_{\varepsilon  \to 0} \int_{\mathop  \cup \limits_{k = 1}^K \bigl( {(\widetilde{{\Lambda ^k}})_\varepsilon ^{ - 1/{\varepsilon ^{1/2}}}\backslash {B_{\beta /\varepsilon }}(x_\varepsilon ^k/\varepsilon )} \bigr)} {({u_\varepsilon^+})^p}  = 0 \text{ for all } p \in (2,{2^ * }].
\end{equation}
\textbf{Step~2}. Letting $${u_{\varepsilon ,1}}(x): = \sum\nolimits_{k = 1}^K {u_{\varepsilon ,1}^k(x)} : = \sum\nolimits_{k = 1}^K {{\varphi _{\varepsilon ,k}}(x){u_\varepsilon }(x)} $$ and ${u_{\varepsilon ,2}}(x): = {u_\varepsilon }(x) - {u_{\varepsilon ,1}}(x)$. Direct computations show that
\[
\int_{{\mathbb{R}^N}} {|\nabla {u_\varepsilon }{|^2}}  \ge \int_{{\mathbb{R}^N}} {|\nabla u_{\varepsilon ,1}{|^2}}  + \int_{{\mathbb{R}^N}} {|\nabla u_{\varepsilon ,2}{|^2}}  + o(1),
\]
\[
\int_{{\mathbb{R}^N}} {|{u_\varepsilon }{|^2}}  \ge \int_{{\mathbb{R}^N}} {|u_{\varepsilon ,1}{|^2}}  + \int_{{\mathbb{R}^N}} {|u_{\varepsilon ,2}{|^2}},
\]
\[
{Q_\varepsilon }({u_{\varepsilon ,1}}) = 0,{\text{ }}{Q_\varepsilon }({u_{\varepsilon ,2}}) = {Q_\varepsilon }({u_\varepsilon }) \ge 0.
\]
Moreover, we derive from $(F_1)$, $(F_2)$, \eqref{3.29} and the definition of ${\tilde f}$ that
\[
\begin{gathered}
  \int_{{\Omega _\varepsilon }} {{G_\varepsilon }(x,{u_\varepsilon })}  = \int_{{\Omega _\varepsilon }} {{G_\varepsilon }(x,{u_{\varepsilon ,1}})}  + \int_{{\Omega _\varepsilon }} {{G_\varepsilon }(x,{u_{\varepsilon ,2}})}  + o(1). \hfill \\
   \hfill \\
\end{gathered}
\]
Hence, we get
\begin{equation}\label{3.30}
{J_\varepsilon }({u_\varepsilon }) \ge {J_\varepsilon }(u_{\varepsilon ,1}) + {J_\varepsilon }(u_{\varepsilon ,2}) + o(1).
\end{equation}
Next, we claim that ${\| {{u_{\varepsilon ,2}}} \|_{H_0^1({\Omega _\varepsilon })}} \to 0$ as $\varepsilon \to 0$. Indeed, by \eqref{3.8}, it holds that

\begin{eqnarray*}
  {\| {{u_{\varepsilon ,2}}} \|_{H_0^1({\Omega _\varepsilon })}} &\le& {\left\| {{u_{\varepsilon ,1}} - \sum\limits_{k = 1}^K {{\varphi _{\varepsilon ,k}}(x){U^k}(x - (x_\varepsilon ^k/\varepsilon ))} } \right\|_{{H^1}\bigl( {\mathop  \cup \limits_{k = 1}^K (\widetilde{{\Lambda ^k}})_\varepsilon ^{  - 1/{\varepsilon ^{1/2}} }} \bigr)}} + 2{d_0} \hfill \\
   &\le& {\| {{u_{\varepsilon ,2}}} \|_{{H^1}\bigl( {\mathop  \cup \limits_{k = 1}^K \bigl( {(\widetilde{{\Lambda ^k}})_\varepsilon ^{  - 1/{\varepsilon ^{1/2}} }\backslash (\widetilde{{\Lambda ^k}})_\varepsilon ^{  - 2/{\varepsilon ^{1/2}} }} \bigr)} \bigr)}} + 4{d_0} \hfill \\
   &\le& C{\| {{u_\varepsilon }} \|_{{H^1}\bigl( {\mathop  \cup \limits_{k = 1}^K \bigl( {(\widetilde{{\Lambda ^k}})_\varepsilon ^{  - 1/{\varepsilon ^{1/2}} }\backslash (\widetilde{{\Lambda ^k}})_\varepsilon ^{  - 2/{\varepsilon ^{1/2}} }} \bigr)} \bigr)}} + 4{d_0} \hfill \\
   &\le& C\sum\limits_{k = 1}^K {{{\| {{\varphi _{\varepsilon ,k}}(x){U^k}(x - (x_\varepsilon ^k/\varepsilon ))} \|}_{{H^1}\bigl( {(\widetilde{{\Lambda ^k}})_\varepsilon ^{  - 1/{\varepsilon ^{1/2}} }\backslash (\widetilde{{\Lambda ^k}})_\varepsilon ^{  - 2/{\varepsilon ^{1/2}} }} \bigr)}}}  + C{d_0} \hfill \\
   &\le& C\sum\limits_{k = 1}^K {{{\| {{U^k}} \|}_{{H^1}( {{\mathbb{R}^N}\backslash ({B_{\beta /\varepsilon }}(0)} )}}}  + C{d_0} = C{d_0} + o(1). \hfill \\
\end{eqnarray*}
Thus, $\mathop {\overline {\lim } }\nolimits_{\varepsilon  \to 0} {\| {u_{\varepsilon ,2}} \|_{{H_0^1({\Omega _\varepsilon })}}} \le C{d_0}$. Since $\langle {{J'_\varepsilon }({u_\varepsilon }),{u_{\varepsilon ,2}}} \rangle  \to 0$ as $\varepsilon  \to 0$ and $\langle {{Q'_\varepsilon }({u_\varepsilon }),{u_{\varepsilon ,2}}} \rangle  = \langle {{Q'_\varepsilon }({u_{\varepsilon ,2}}),{u_{\varepsilon ,2}}} \rangle  \ge 0$, we see from $(F_1)$ and $(F_2)$ that
\[\begin{array}{ll}
\dis  \int_{{\Omega _\varepsilon }} {\nabla {u_\varepsilon }\cdot\nabla {u_{\varepsilon ,2}}}  + \int_{{\Omega _\varepsilon }} {{u_\varepsilon }{u_{\varepsilon ,2}}} & \le \dis \int_{{\Omega _\varepsilon }} {f({u_\varepsilon }){u_{\varepsilon ,2}}}  + o(1) \hfill \\
   &\le \dis\frac{1}
{2}\int_{{\Omega _\varepsilon }} {u_\varepsilon ^+u_{\varepsilon ,2}^ + }  + C\int_{{\Omega _\varepsilon }} {{{(u_\varepsilon ^+)}^{{2^*} - 1}}(u_{\varepsilon ,2}^+)}  + o(1), \hfill \\
\end{array} \]
then by \eqref{3.29} and the Sobolev's imbedding theorem, we have
\begin{equation}\label{add9}
\| {{u_{\varepsilon ,2}}} \|_{H_0^1({\Omega _\varepsilon })}^2 \le C\| {{u_{\varepsilon ,2}}} \|_{H_0^1({\Omega _\varepsilon })}^{{2^ * }} + o(1).
\end{equation}
Choosing ${d_0} > 0$ small, \eqref{add9} yields ${\| {{u_{\varepsilon ,2}}} \|_{{H_0^1({\Omega _\varepsilon })}}} = o(1)$, by \eqref{3.30}, we have
\begin{equation}\label{3.31}
{J_\varepsilon }({u_\varepsilon }) \ge {J_\varepsilon }(u_{\varepsilon ,1}) + o(1).
\end{equation}
\textbf{Step~3}. For each $1 \le k \le K$, letting $$\tilde w_\varepsilon ^k(x): = u_{\varepsilon ,1}^k(x + (x_\varepsilon ^k/\varepsilon )) = {\varphi _{\varepsilon ,k}}(x + (x_\varepsilon ^k/\varepsilon )){u_\varepsilon }(x + (x_\varepsilon ^k/\varepsilon )).$$ Up to a subsequence, as $\varepsilon \to 0$, for some $ {{\tilde w}^k} \in {H^1}({\mathbb{R}^N})$ such that $\tilde w_\varepsilon ^k \rightharpoonup {{\tilde w}^k}$ in ${H^1}({\mathbb{R}^N})$. Next, we claim that
\begin{equation}\label{3.32}
\tilde w_\varepsilon ^k \to {{\tilde w}^k}{\text{ in }}{L^{{2^*}}}({\mathbb{R}^N}).
\end{equation}
Indeed, in view of Lemma~\ref{2.2.}, we assume on the contrary that for some $ r > 0$ such that
\[
\mathop {\underline {\lim } }\limits_{\varepsilon  \to 0} \mathop {\sup }\limits_{z \in {\mathbb{R}^N}} \int_{{B_1}(z)} {|\tilde w_\varepsilon ^k - {{\tilde w}^k}{|^{{2^*}}}}  = 2r > 0,
\]
then for $\varepsilon > 0$ small and $ z_\varepsilon ^k \in {\mathbb{R}^N}$ such that
\begin{equation}\label{3.33}
\int_{{B_1}(z_\varepsilon ^k)} {|\tilde w_\varepsilon ^k - {{\tilde w}^k}{|^{{2^*}}}}  \ge r > 0.
\end{equation}
There are two cases.

{\bf Case~I.} ${\{ z_\varepsilon ^k\} _{\varepsilon  > 0}}$ is bounded, i.e. $|z_\varepsilon ^k| \le \alpha $ for some $\alpha > 0$, then
\begin{equation}\label{3.34}
\int_{{B_{\alpha  + 1}}(0)} {|\tilde v_\varepsilon ^k{|^{{2^*}}}}  \ge r > 0,
\end{equation}
where
\begin{equation}\label{3.35}
\tilde v_\varepsilon ^k: = \tilde w_\varepsilon ^k - {{\tilde w}^k} \rightharpoonup 0{\text{ in }}{H^1}({\mathbb{R}^N}){\text{ as }}\varepsilon  \to 0.
\end{equation}
Arguing as in {\bf Case~2} of \textbf{Step~1}, we see from \eqref{3.34} and \eqref{3.35} that for each $1 \le k \le K$, there exist ${\{ \tilde z_\varepsilon ^k\} _{\varepsilon  > 0}} \subset \overline {{B_{\alpha  + 1}}(0)} $ and ${\{ \tilde \sigma _\varepsilon ^k\} _{\varepsilon  > 0}}$ with $\tilde \sigma _\varepsilon ^k > 0$ and $\tilde \sigma _\varepsilon ^k \to 0$ as $\varepsilon \to 0$ such that
\begin{center}
$\hat w_\varepsilon ^k(x): = {(\tilde \sigma _\varepsilon ^k)^{(N - 2)/2}}\tilde v_\varepsilon ^k(\tilde \sigma _\varepsilon ^kx + \tilde z_\varepsilon ^k) \rightharpoonup {{\hat w}^k}(x)$ in ${D^{1,2}}({\mathbb{R}^N})$,
\end{center}
where ${{\hat w}^k}$ is a nontrivial solution to \eqref{3.20} and satisfies \eqref{3.26}. From \eqref{3.8} and the fact that ${\| {{u_{\varepsilon ,2}}} \|_{H_0^1({\Omega _\varepsilon })}} = o(1)$, we get
\begin{equation}\label{3.36}
\begin{array}{ll}
\dis  \int_{{\mathbb{R}^N}} {|\nabla u_{\varepsilon ,1}^k{|^2}}  &\le\dis \int_{{\mathbb{R}^N}} {|\nabla ({\varphi _{\varepsilon ,k}}(x){U^k}(x - (x_\varepsilon ^k/\varepsilon ))){|^2}}  + C{d_0} \hfill \\
   &\le \dis\int_{{\mathbb{R}^N}} {|\nabla {U^{{k}}}{|^2}}  + C{d_0} + o(1). \hfill \\
\end{array}
\end{equation}
On the other hand,
\begin{equation}\label{3.37}
\begin{array}{ll}
 \dis \int_{{\mathbb{R}^N}} {|\nabla {{\hat w}^k}{|^2}}  &\le \dis \mathop {\underline {\lim } }\limits_{\varepsilon  \to 0} \int_{{\mathbb{R}^N}} {|\nabla \hat w_\varepsilon ^k{|^2}}  = \mathop {\underline {\lim } }\limits_{\varepsilon  \to 0} \int_{{\mathbb{R}^N}} {|\nabla \tilde v_\varepsilon ^k{|^2}}  \hfill \\
  & \dis= \mathop {\underline {\lim } }\limits_{\varepsilon  \to 0} \int_{{\mathbb{R}^N}} {|\nabla \tilde w_\varepsilon ^k{|^2}}  - \int_{{\mathbb{R}^N}} {|\nabla {{\tilde w}^k}{|^2}}  \le \mathop {\underline {\lim } }\limits_{\varepsilon  \to 0} \int_{{\mathbb{R}^N}} {|\nabla u_{\varepsilon ,1}^k{|^2}} . \hfill \\
\end{array}
\end{equation}
By \eqref{3.36} and \eqref{3.37}, we have
\begin{equation}\label{3.38}
\int_{{\mathbb{R}^N}} {|\nabla {{\hat w}^k}{|^2}}  \le \int_{{\mathbb{R}^N}} {|\nabla {U^{{k}}}{|^2}}  + C{d_0}.
\end{equation}
Thus, we see from \eqref{3.38} and \eqref{3.26} that
\[
c = I({U^k}) - \frac{1}
{N}P({U^k}) = \frac{1}
{N}\int_{{\mathbb{R}^N}} {|\nabla {U^k}{|^2}}  \ge \frac{1}
{N}\int_{{\mathbb{R}^N}} {|\nabla {{\hat w}^k}{|^2}}  - C{d_0} = \frac{1}
{N}{S^{N/2}} - C{d_0},
\]
which contradicts $c < \frac{1}
{N}{S^{N/2}}$ for ${d_0} > 0$ small.

{\bf Case~II.} ${\{ z_\varepsilon ^k\} _{\varepsilon  > 0}}$ is unbounded, i.e. up to a subsequence, $\mathop {\lim }\nolimits_{\varepsilon  \to 0} |z_\varepsilon ^k| =  + \infty $. Thus, by \eqref{3.33}, we see that
\begin{equation}\label{3.39}
\mathop {\underline {\lim } }\limits_{\varepsilon  \to 0} \int_{{B_1}(z_\varepsilon ^k)} {|\tilde w_\varepsilon ^k{|^{{2^ * }}}}  \ge r > 0,
\end{equation}
i.e.
\begin{equation}\label{3.40}
\mathop {\underline {\lim } }\limits_{\varepsilon  \to 0} \int_{{B_1}(z_\varepsilon ^k)} {|{\varphi _{\varepsilon ,k}}(x + (x_\varepsilon ^k/\varepsilon )){u_\varepsilon }(x + (x_\varepsilon ^k/\varepsilon )){|^{{2^ * }}}}  \ge r > 0.
\end{equation}
Since ${\varphi _{\varepsilon ,k}}(x) = 0$ for $x \notin (\widetilde{{\Lambda ^k}})_\varepsilon ^{ - 1/{\varepsilon ^{1/2}}}$, we see from \eqref{3.40} that $z_\varepsilon ^k + (x_\varepsilon ^k/\varepsilon ) \in (\widetilde{{\Lambda ^k}})_\varepsilon ^{ - 1/(2{\varepsilon ^{1/2}})}$ for $\varepsilon  > 0$ small. If $|z_\varepsilon ^k| \ge \beta /2\varepsilon $, then $z_\varepsilon ^k + (x_\varepsilon ^k/\varepsilon ) \in (\widetilde{{\Lambda ^k}})_\varepsilon ^{ - 1/(2{\varepsilon ^{1/2}})}\backslash {B_{\beta /2\varepsilon }}(x_\varepsilon ^k/\varepsilon )$. By \eqref{3.9}, we get
\[
\begin{array}{ll}
 \dis \mathop {\underline {\lim } }\limits_{\varepsilon  \to 0} \int_{{B_1}(z_\varepsilon ^k)} {|\tilde w_\varepsilon ^k{|^{{2^ * }}}}  &\dis\le \mathop {\underline {\lim } }\limits_{\varepsilon  \to 0} \int_{{B_1}(z_\varepsilon ^k)} {|{u_\varepsilon }(x + (x_\varepsilon ^k/\varepsilon )){|^{{2^ * }}}}  \hfill \\
   &\dis\le \mathop {\underline {\lim } }\limits_{\varepsilon  \to 0} \int_{{B_1}(z_\varepsilon ^k + (x_\varepsilon ^k/\varepsilon ))} {|{u_\varepsilon }{|^{{2^ * }}}}  \le \mathop {\underline {\lim } }\limits_{\varepsilon  \to 0} \mathop {\sup }\limits_{y \in {A_\varepsilon }} \int_{{B_1}(y)} {|{u_\varepsilon }{|^{{2^*}}}}  = 0, \hfill \\
\end{array}
\]
which contradicts \eqref{3.39}. Thus, $|z_\varepsilon ^k| \le \beta /2\varepsilon $ for $\varepsilon  > 0$ small. Assume that $\bar w_\varepsilon ^k(x): = \tilde w_\varepsilon ^k(x + z_\varepsilon ^k) \rightharpoonup {{\bar w}^k}(x)$ in ${H^1}({\mathbb{R}^N})$. If $\bar w_\varepsilon ^k \ne 0$, arguing as in {\bf Case~1} of \textbf{Step~1}, we get a contradiction for ${d_0} > 0$ small. If $\bar w_\varepsilon ^k = 0$, similar to {\bf Case~2} of \textbf{Step~1}, we see from \eqref{3.39} that for some $ \tilde x_\varepsilon ^k \in \overline {{B_1}(0)} $, $\gamma _\varepsilon ^k > 0$ with ${\gamma _\varepsilon ^k } \to 0$ as $\varepsilon  \to 0$ and
\[
w_{\varepsilon , * }^k(x): = {(\gamma _\varepsilon ^k)^{(N - 2)/2}}\bar w_\varepsilon ^k(\gamma _\varepsilon ^kx + \tilde x_\varepsilon ^k) \rightharpoonup w_ * ^k(x){\text{ in }}{D^{1,2}}({\mathbb{R}^N}),
\]
where $w_ * ^k$ is a nontrivial solution to \eqref{3.20} and satisfies \eqref{3.26}. Arguing as \eqref{3.27}, \eqref{3.28} and \eqref{3.14}, we get a contradiction by \eqref{3.8} for ${d_0} > 0$ small. Thus, \eqref{3.32} holds. The interpolation inequality shows that for each $1 \le k \le K$,
\begin{equation}\label{3.41}
\tilde w_\varepsilon ^k \to {{\tilde w}^k}{\text{ in }}{L^p}({\mathbb{R}^N})  \text{ for all } p \in (2,{2^ * }].
\end{equation}
Thanks to \eqref{3.8}, \eqref{3.41} and the fact that ${\| {{u_{\varepsilon ,2}}} \|_{H_0^1({\Omega _\varepsilon })}} = o(1)$, we see that ${{{\tilde w}^k}} \ne 0$ for ${d_0} > 0$ small. Arguing as in {\bf Case~1} of \textbf{Step~1}, we see that ${{{\tilde w}^k}}$ satisfies \eqref{2.1} or \eqref{3.12} with $\tilde f({{\tilde w}^k}(x)) < f({{\tilde w}^k}(x))$ for some $x \in {\mathbb{R}^N}\backslash \mathbb{R}_ + ^N$, then we have
\begin{equation}\label{3.42}
I({{\tilde w}^k}) \ge c{\text{ or }}\tilde I({{\tilde w}^k}) > c
\end{equation}
 respectively. Moreover, by \eqref{3.31},
\[
Kc + o(1) \ge {J_\varepsilon }({u_\varepsilon }) \ge {J_\varepsilon }({u_{\varepsilon ,1}}) + o(1) = \sum\limits_{k = 1}^K {J_\varepsilon ^k(u_{\varepsilon ,1}^k)}  + o(1),
\]
then by \eqref{3.41}, we get
\begin{equation}\label{3.43}
 \sum\limits_{k = 1}^K {I({{\tilde w}^k})({\text{or }}\tilde I({{\tilde w}^k}))}  \le Kc.
\end{equation}
From \eqref{3.42} and \eqref{3.43}, we see that for each $1 \le k \le K$,
\begin{equation}\label{3.44}
I({{\tilde w}^k}) = c.
\end{equation}
To sum up, only the situation which is analogous to {\bf Case~1-1} or {\bf Case~1-2-2-2} in \textbf{Step~1} may occur. That is, for each $1 \le k \le K$, one of the following two cases will occur.
\begin{itemize}
\item {\bf Case~A}. $x_\varepsilon ^k \in {\Lambda ^k}$ and $\mathop {\lim }\nolimits_{\varepsilon  \to 0} d(x_\varepsilon ^k/\varepsilon ,\partial {({\Lambda ^k})_\varepsilon }) =  + \infty $. ${{\tilde w}^k}$ satisfies \eqref{2.1}, by \eqref{3.44}, ${{\tilde w}^k}$ is a ground state solution to \eqref{2.1}, i.e. for some $ {U^k} \in S$, ${z^k} \in {\mathbb{R}^N}$ such that ${{\tilde w}^k}(x) = {U^k}(x - {z^k})$.
\item {\bf Case~B}. $x_\varepsilon ^k \in {\Lambda ^k}$ and $\mathop {\lim }\nolimits_{\varepsilon  \to 0} d(x_\varepsilon ^k/\varepsilon ,\partial {({\Lambda ^k})_\varepsilon }) < + \infty $. Up to a translation and a rotation, ${{\tilde w}^k}$ satisfies \eqref{3.12} and $\tilde f({{\tilde w}^k}(x)) = f({{\tilde w}^k}(x))$ for all $x \in {\mathbb{R}^N}\backslash \mathbb{R}_ + ^N$, i.e. ${{\tilde w}^k}$ is actually a ground state solution to \eqref{2.1}, similar to {\bf Case~1-2-2-2}, for some ${U^k} \in S$, ${z^k} \in \mathbb{R}_ + ^N$ such that ${{\tilde w}^k}(x) = {U^k}(x - {z^k})$.
\end{itemize}
Since $\langle {{J'_\varepsilon }({u_\varepsilon }),u_{\varepsilon ,1}^k} \rangle  \to 0$, ${\| {{u_{\varepsilon ,2}}} \|_{H_0^1({\Omega _\varepsilon })}} \to 0$ as $\varepsilon  \to 0$ and $\langle {{Q'_\varepsilon }({u_\varepsilon }),u_{\varepsilon ,1}^k} \rangle  = 0$, we see from \eqref{3.41} that in either {\bf Case~A} or {\bf Case~B},
\begin{equation}\label{3.45}
\begin{gathered}
\quad  \int_{{\Omega _\varepsilon } - (x_\varepsilon ^k/\varepsilon )} {|\nabla \tilde w_\varepsilon ^k{|^2}}  + \int_{{\Omega _\varepsilon } - (x_\varepsilon ^k/\varepsilon )} {|\tilde w_\varepsilon ^k{|^2}}  \hfill \\
   = \int_{{\Omega _\varepsilon } - (x_\varepsilon ^k/\varepsilon )} {{\chi _{{{({\Lambda ^k})}_\varepsilon }}}(x + (x_\varepsilon ^k/\varepsilon ))f(\tilde w_\varepsilon ^k)\tilde w_\varepsilon ^k}  \hfill \\
\quad   + \int_{{\Omega _\varepsilon } - (x_\varepsilon ^k/\varepsilon )} {(1 - {\chi _{{{({\Lambda ^k})}_\varepsilon }}}(x + (x_\varepsilon ^k/\varepsilon )))\tilde f(\tilde w_\varepsilon ^k)\tilde w_\varepsilon ^k}  + o(1) \hfill \\
   = \int_{{\mathbb{R}^N}} {f({{\tilde w}^k}){{\tilde w}^k}}  + o(1). \hfill \\
\end{gathered}
\end{equation}
Since ${{\tilde w}^k}$ satisfies \eqref{2.1}, then
\begin{equation}\label{3.46}
\int_{{\mathbb{R}^N}} {|\nabla {{\tilde w}^k}{|^2}}  + \int_{{\mathbb{R}^N}} {|{{\tilde w}^k}{|^2}}  = \int_{{\mathbb{R}^N}} {f({{\tilde w}^k}){{\tilde w}^k}} .
\end{equation}
From \eqref{3.45} and \eqref{3.46}, we see that for each $1 \le k \le K$,
\begin{equation}\label{3.47}
\tilde w_\varepsilon ^k \to {{\tilde w}^k}{\text{ in }}{H^1}({\mathbb{R}^N}).
\end{equation}
Note that, in either {\bf Case~A} or {\bf Case~B}, for each $1 \le k \le K$ and $\varepsilon >0$ small, $(x_\varepsilon ^k/\varepsilon ) + {z^k} \in {({\Lambda ^k})_\varepsilon }$, by \eqref{3.47}, we have
\[
{\left\| {{u_\varepsilon } - \sum\limits_{k = 1}^K {{\varphi _{\varepsilon ,k}}(x){U^k}(x - ((x_\varepsilon ^k/\varepsilon ) + {z^k}))} } \right\|_{H_0^1({\Omega _\varepsilon })}} \to 0{\text{ as }}\varepsilon  \to 0.
\]
\end{proof}

\begin{lemma}\label{3.5.}
For ${d_0}$ given in Proposition~\ref{3.4.} and any $d \in (0,{d_0})$, there exist ${\varepsilon _d} > 0$, ${\rho _d} > 0$ and ${\omega _d} > 0$ such that
\[
{\| {{J'_\varepsilon }(u)} \|_{{{(H_0^1({\Omega _\varepsilon }))}^{ - 1}}}} \ge {\omega _d}
\]
for all $u \in J_\varepsilon ^{Kc  + {\rho _d}} \cap (X_\varepsilon ^{{d_0}}\backslash X_\varepsilon ^d)$ with $\varepsilon  \in (0,{\varepsilon _d})$.
\end{lemma}
\begin{proof}
Assuming on the contrary that, there exist $d \in (0,{d_0})$, $\{ {\varepsilon _j}\}  _{j = 1}^\infty$, $\{ {\rho _j}\}  _{j = 1}^\infty$ with ${\varepsilon _j}$, ${\rho _j} \to 0$ and ${u_{{\varepsilon _j}}} \in J_{{\varepsilon _j}}^{Kc + {\rho _j}} \cap (X_{{\varepsilon _j}}^{{d_0}}\backslash X_{{\varepsilon _j}}^d)$ such that
\[
{\| {{J'_{{\varepsilon _j}}}({u_{{\varepsilon _j}}})} \|_{{{(H_0^1({\Omega _{{\varepsilon _j}}}))}^{ - 1}}}} \to 0{\text{ as }}j \to \infty .
\]
By Proposition~\ref{3.4.}, for each $1 \le k \le K$, there exist, up to a subsequence, $\{ y_{{\varepsilon _j}}^k\} _{j = 1}^\infty  \subset {({\Lambda ^k})_{{\varepsilon _j}}}$, ${U^k} \in S$ such that
\[
\mathop {\lim }\limits_{j \to \infty } {\left\| {{u_{{\varepsilon _j}}} - \sum\limits_{k = 1}^K {{\varphi _{{\varepsilon _j},k}}(x){U^k}(x - y_{{\varepsilon _j}}^k)} } \right\|_{H_0^1({\Omega _{{\varepsilon _j}}})}} = 0,
\]
which gives that ${u_{{\varepsilon _j}}} \in X_{{\varepsilon _j}}^d$ for large $j$. This contradicts ${u_{{\varepsilon _j}}} \notin X_{{\varepsilon _j}}^d$.
\end{proof}

\begin{lemma}\label{3.6.}
There exists ${T_0} > 0$ with the following property: for any $\delta  > 0$ small, there exist ${\alpha _\delta } > 0$ and ${\varepsilon _\delta } > 0$ such that ${\gamma _{\varepsilon} }(s) \in X_\varepsilon ^{{T_0}\delta }$ if ${J_\varepsilon }({\gamma _{\varepsilon}}(s)) \ge Kc  - {\alpha _\delta }$ and $\varepsilon  \in (0,{\varepsilon _\delta })$.
\end{lemma}
\begin{proof}
First, there is ${T'_0} > 0$ such that for each $1 \le k \le K$, $z_ * ^k \in {\mathcal{M}^k}$ and $u \in {H^1}({\mathbb{R}^N})$,
\begin{equation}\label{3.48}
{\| {{\varphi _{\varepsilon ,k}}(x)u(x - (z_*^k/\varepsilon ))} \|_{H_0^1({\Omega _\varepsilon })}} \le {T'_0}{\| {u(x)} \|_{{H^1}({\mathbb{R}^N})}}.
\end{equation}
Denoting
\[
\begin{array}{ll}
  {\alpha _\delta } = &\dis\frac{1}
{4}\min \Bigg \{ Kc - \sum\limits_{k = 1}^K {I(U_ * ^k(s_k^{ - 1}t_0^{ - 1}x))} :s = ({s_1},...,{s_K}) \in {[0,1]^K}, \hfill \\
 &\dis \sum\limits_{k = 1}^K {{{\| {U_ * ^k(s_k^{ - 1}t_0^{ - 1}x) - U_ * ^k(x)} \|}_{{H^1}({\mathbb{R}^N})}}}  \ge k\delta \Bigg\}  > 0, \hfill \\
\end{array}
\]
we have
\begin{equation}\label{3.49}
\sum\limits_{k = 1}^K {I(U_ * ^k(s_k^{ - 1}t_0^{ - 1}x))}  \ge Kc - 2{\alpha _\delta }{\text{ implies }}\sum\limits_{k = 1}^K {{{\| {U_ * ^k(s_k^{ - 1}t_0^{ - 1}x) - U_ * ^k(x)} \|}_{{H^1}({\mathbb{R}^N})}}}  \le k\delta .
\end{equation}
Direct computations show that
\[
\mathop {\max }\limits_{s \in {{[0,1]}^k}} \left| {{J_\varepsilon }({\gamma _\varepsilon }(s)) - \sum\limits_{k = 1}^K {I(U_ * ^k(s_k^{ - 1}t_0^{ - 1}x))} } \right| \to 0{\text{ as }}\varepsilon  \to 0,
\]
then, there exists an $\varepsilon _\delta >0$ such that
\begin{equation}\label{3.50}
\mathop {\max }\limits_{s \in {{[0,1]}^k}} \left| {{J_\varepsilon }({\gamma _\varepsilon }(s)) - \sum\limits_{k = 1}^K {I(U_ * ^k(s_k^{ - 1}t_0^{ - 1}x))} } \right| \le {\alpha _\delta }
\end{equation}
for all $\varepsilon  \in (0,{\varepsilon _\delta })$. Thus if $\varepsilon  \in (0,{\varepsilon _\delta })$ and ${J_\varepsilon }({\gamma _\varepsilon }(s)) \ge Kc - {\alpha _\delta }$, by \eqref{3.49} and \eqref{3.50}, we see that
\[
\sum\limits_{k = 1}^K {{{\left\| {U_ * ^k(s_k^{ - 1}t_0^{ - 1}x) - U_ * ^k(x)} \right\|}_{{H^1}({\mathbb{R}^N})}}}  \le k\delta ,
\]
and by \eqref{3.48},
\[
\begin{gathered}
\quad  {\left\| {\sum\limits_{k = 1}^K {W_{\varepsilon ,{s_k}{t_0}}^k}  - \sum\limits_{k = 1}^K {{\varphi _{\varepsilon ,k}}(x)U_ * ^k(x - (z_*^k/\varepsilon ))} } \right\|_{H_0^1({\Omega _\varepsilon })}} \hfill \\
   \le \sum\limits_{k = 1}^K {{{\left\| {{\varphi _{\varepsilon ,k}}(x)(U_ * ^k(s_k^{ - 1}t_0^{ - 1}(x - (z_*^k/\varepsilon ))) - U_ * ^k(x - (z_*^k/\varepsilon )))} \right\|}_{H_0^1({\Omega _\varepsilon })}}}  \hfill \\
   \le {{T'}_0}\sum\limits_{k = 1}^K {{{\left\| {U_ * ^k(s_k^{ - 1}t_0^{ - 1}x) - U_ * ^k(x)} \right\|}_{{H^1}({\mathbb{R}^N})}}}  \le k{T'_0}\delta : = {T_0}\delta . \hfill \\
\end{gathered}
\]
Thus ${\gamma _\varepsilon }(s) \in X_\varepsilon ^{{T_0}\delta }$.
\end{proof}

Choosing ${\delta _1} > 0$ to ensure that ${T_0}{\delta _1} < {d_0}/4$, let $\bar \alpha  = \min \{ {\alpha _{{\delta _1}}},\sigma \} $ and $d = {d_0}/4: = {d_1}$ in Lemma~\ref{3.5.}. To prove the next lemma, we use an idea developed in \cite{fis}. However, to construct multi-peak solutions which concentrate around the prescribed finite sets of local maxima of the distance function $d( \cdot ,\partial \Omega )$, we give a more complicated proof than that in \cite{fis}, where only the single-peak solutions which concentrate around local minimum points of the potential function were considered.
\begin{lemma}\label{3.7.}
There exists $\bar \varepsilon  > 0$ such that for each $\varepsilon  \in (0,\bar \varepsilon ]$, there exists a sequence $\{ {v_{n,\varepsilon }}\} _{n = 1}^\infty  \subset J_\varepsilon ^{\widetilde{{c_\varepsilon }} + \sum\nolimits_{k = 1}^K {\exp ( - 2{D_k}/\varepsilon )} } \cap X_\varepsilon ^{{d_0}}$
 such that ${J'_\varepsilon }({v_{n,\varepsilon }}) \to 0$ in ${(H_0^1({\Omega _\varepsilon }))^{ - 1}}$ as $n \to \infty $.
\end{lemma}
\begin{proof}
Assuming on the contrary that there always exist $\varepsilon  > 0$ small and $\gamma (\varepsilon) > 0$ such that
\begin{equation}\label{3.51}
{\| {{J'_\varepsilon }(u)} \|_{{{(H_0^1({\Omega _\varepsilon }))}^{ - 1}}}} \ge \gamma (\varepsilon ) > 0
\end{equation}
for $u \in J_\varepsilon ^{\widetilde{{c_\varepsilon }} + \sum\nolimits_{k = 1}^K {\exp ( - 2{D_k}/\varepsilon )} } \cap X_\varepsilon ^{{d_0}}$.

Let $Y$ be a pseudo-gradient vector field for ${{J'_{\varepsilon }}}$ in $H_0^1({\Omega _\varepsilon })$, i.e. $Y:H_0^1({\Omega _\varepsilon }) \to H_0^1({\Omega _\varepsilon })$ is a locally Lipschitz continuous vector field such that for every $u \in H_0^1({\Omega _\varepsilon })$,
\begin{equation}\label{3.52}
{\| {Y(u)} \|_{H_0^1({\Omega _\varepsilon })}} \le 2{\| {{J'_\varepsilon }(u)} \|_{{{(H_0^1({\Omega _\varepsilon }))}^{ - 1}}}}{\text{ and }}\langle {J'_\varepsilon }(u),Y(u)\rangle  \ge \| {{J'_\varepsilon }(u)} \|_{{{(H_0^1({\Omega _\varepsilon }))}^{ - 1}}}^2,
\end{equation}
and ${\psi _1}$, ${\psi _2}$ be locally Lipschitz continuous functions in $H_0^1({\Omega _\varepsilon })$ such that $0 \le {\psi _1},{\psi _2} \le 1$ and
\[
{\psi _1}(u) = \left\{ \begin{array}{ll}
  1, &Kc  - \frac{1}
{2}{\bar \alpha} \le {J_\varepsilon }(u) \le \widetilde{{c_\varepsilon }},
\vspace{0.2cm}\\
  0, &{J_\varepsilon }(u) \le Kc  - {{\bar \alpha}}{\text{ or }}  {{\widetilde{{c_\varepsilon }} + \sum\nolimits_{k = 1}^K {\exp ( - 2{D_k}/\varepsilon )} }}  \le {J_\varepsilon }(u),
\end{array}  \right.
\]
\[
{\psi _2}(u) = \left\{ \begin{array}{ll}
  1,&u \in X_\varepsilon ^{3{d_0}/4},
\vspace{0.2cm}\\
  0,&u \notin X_\varepsilon ^{{d_0}}.
\end{array} \right.
\]
Consider the following Cauchy problem
\begin{equation}\label{3.53}
\left\{ \begin{gathered}
  \frac{d}
{{dr}}\eta (r,u) =  - \frac{{Y(\eta (r,u))}}
{{{{\| {Y(\eta (r,u))} \|}_{H_0^1({\Omega _\varepsilon })}}}}{\psi _1}(\eta (r,u)){\psi _2}(\eta (r,u)), \hfill \\
  \eta (0,u) = u. \hfill \\
\end{gathered}  \right.
\end{equation}
By \eqref{3.52} and \eqref{3.53}, we have
\[
\begin{array}{ll}
  \frac{d}
{{dr}}{J_\varepsilon }(\eta (r,u)) &=\dis \left\langle {{J'_\varepsilon }(\eta (r,u)),\frac{d}
{{dr}}\eta (r,u)} \right\rangle  \hfill \\
   &=\dis \left\langle {{J'_\varepsilon }(\eta (r,u)), - \frac{{Y(\eta (r,u))}}
{{{{\| {Y(\eta (r,u))} \|}_{H_0^1({\Omega _\varepsilon })}}}}{\psi _1}(\eta (r,u)){\psi _2}(\eta (r,u))} \right\rangle  \hfill \\
   &\le \dis  - \frac{{{\psi _1}(\eta (r,u)){\psi _2}(\eta (r,u))}}
{{{{\| {Y(\eta (r,u))} \|}_{H_0^1({\Omega _\varepsilon })}}}}\| {{J'_\varepsilon }(\eta (r,u))} \|_{(H_0^1({\Omega _\varepsilon }))^{ - 1}}^2 \hfill \\
   &\le \dis - \frac{1}
{2}{\psi _1}(\eta (r,u)){\psi _2}(\eta (r,u)){\| {{J'_\varepsilon }(\eta (r,u))} \|_{(H_0^1({\Omega _\varepsilon }))^{ - 1}}}. \hfill \\
\end{array}
\]
In view of Lemma~\ref{3.5.}, \eqref{3.51}, \eqref{3.53} and the definition of ${\psi _1}$, ${\psi _2}$, we see that $\eta  \in C([0, + \infty ) \times H_0^1({\Omega _\varepsilon }),H_0^1({\Omega _\varepsilon }))$ and satisfies that for $\varepsilon  > 0$ small,\\
$(i)$ $\frac{d}
{{dr}}{J_{\varepsilon }}(\eta (r,u)) \le 0$ for each $r \in [0,+\infty )$ and $u \in H_0^1({\Omega _\varepsilon })$;\\
$(ii)$ $\frac{d}
{{dr}}{J_\varepsilon }(\eta (r,u)) \le  -{\omega _{{d_1}}}/2$ if $\eta (r,u) \in \overline {J_\varepsilon ^{{{\tilde c}_\varepsilon }}\backslash J_\varepsilon ^{Kc - \frac{1}
{2}\bar \alpha }}  \cap \overline {X_\varepsilon ^{3{d_0}/4}\backslash X_\varepsilon ^{{d_0}/4}} $;\\
$(iii)$ $\frac{d}
{{dr}}{J_\varepsilon }(\eta (r,u)) \le -\gamma (\varepsilon )/2$ if $\eta (r,u) \in \overline {J_\varepsilon ^{{{\tilde c}_{\varepsilon }}}\backslash J_\varepsilon ^{Kc  - \frac{1}
{2}{\bar \alpha}}}  \cap X_\varepsilon ^{3{d_0}/4}$;\\
$(iv)$ $\eta (r,u) = u$ if ${J_\varepsilon }(u) \le Kc  - {\bar \alpha}$.

Letting ${r_1}: = {\omega _{{d_1}}}{d_0}/\gamma (\varepsilon )$ and ${\xi _\varepsilon }(s): = \eta ({r_1},{\gamma _{\varepsilon }}(s))$, where ${\gamma _{\varepsilon }}(s)$ has been mentioned in \eqref{3.4}, we have the following cases.

\textbf{Case 1}. ${\gamma _\varepsilon }(s) \in J_\varepsilon ^{Kc - \bar \alpha }$. By $(iv)$, we see that
\begin{equation}\label{3.54}
\eta (r,{\gamma _{\varepsilon}}(s)) = {\gamma _{\varepsilon }}(s).
\end{equation}

\textbf{Case 2}. ${\gamma _\varepsilon }(s) \notin J_\varepsilon ^{Kc - \bar \alpha }$. By Lemma~\ref{3.6.} and the definition of $\widetilde{{c_\varepsilon }}$, we see that
\[
{\gamma _\varepsilon }(s) \in \overline {J_\varepsilon ^{\widetilde{{c_\varepsilon }}}\backslash J_\varepsilon ^{Kc - \bar \alpha }}  \cap X_\varepsilon ^{{d_0}/4}.
\]
Moreover, we have
\begin{equation}\label{3.55}
\eta (r,{\gamma _\varepsilon }(s)) \in X_\varepsilon ^{{d_0}}{\text{ for }}r \in [0,{r_1}].
\end{equation}
Indeed, if not, for some $ r' \in [0,{r_1}]$, $\eta (r',{\gamma _\varepsilon }(s)) \notin X_\varepsilon ^{{d_0}}$. Denote
\[
r'': = \sup \left\{ {r \in [0,r']:\eta (r,{\gamma _{\varepsilon }}(s)) \in X_\varepsilon ^{{d_0}}} \right\},
\]
by \eqref{3.53} and the definition of ${\psi _2}$, we see $\eta (r',{\gamma _{\varepsilon }}(s)) = \eta (r'',{\gamma _{\varepsilon }}(s)) \in X_\varepsilon ^{{d_0}}$, which is a contradiction.

Next, we divide \textbf{Case 2} into the following three subcases.

\textbf{Case 2.1}. $\eta ({r_1},{\gamma _\varepsilon }(s)) \in J_\varepsilon ^{Kc - \frac{1}
{2}\bar \alpha }$;

\textbf{Case 2.2}. $\eta ({r_1},{\gamma _\varepsilon }(s)) \in J_\varepsilon ^{\widetilde{{c_\varepsilon }}}\backslash J_\varepsilon ^{Kc - \frac{1}
{2}\bar \alpha }$ and $\eta (r,{\gamma _\varepsilon }(s)) \notin X_\varepsilon ^{3{d_0}/4}$ for some $r \in [0,{r_1}]$;

\textbf{Case 2.3}. $\eta ({r_1},{\gamma _\varepsilon }(s)) \in J_\varepsilon ^{\widetilde{{c_\varepsilon }}}\backslash J_\varepsilon ^{Kc - \frac{1}
{2}\bar \alpha }$ and $\eta (r,{\gamma _\varepsilon }(s)) \in X_\varepsilon ^{3{d_0}/4}$ for all $r \in [0,{r_1}]$.

In \textbf{Case 2.2}, denote
\[
{r_2}: = \inf \left\{ {r \in [0,{r_1}]:\eta (r,{\gamma _{\varepsilon }}(s)) \notin X_\varepsilon ^{3{d_0}/4}} \right\}
\]
and
\[
{r_3}: = \sup \left\{ {r \in [0,{r_2}]:\eta (r,{\gamma _{\varepsilon }}(s)) \in X_\varepsilon ^{{d_0}/4}} \right\},
\]
then by \eqref{3.53}, ${r_2} - {r_3} \ge \frac{1}
{2}{d_0}$ and $\eta (r,{\gamma _{\varepsilon }}(s)) \in \overline {X_\varepsilon ^{3{d_0}/4}\backslash X_\varepsilon ^{{d_0}/4}} $ for each $r \in [{r_3},{r_2}]$. By $(i)$, $(ii)$ and Lemma~\ref{3.1.}(i), we obtain
\[
\begin{array}{ll}
{J_\varepsilon }(\eta ({r_1},{\gamma _{\varepsilon }}(s))) &= \dis {J_\varepsilon }({\gamma _{\varepsilon }}(s)) + \int_0^{{r_1}} {\frac{d}
{{dr}}{J_\varepsilon }(\eta (r,{\gamma _{\varepsilon }}(s)))} ds \hfill \\
   &\le \dis \widetilde{{c_\varepsilon }} + \int_{{r_3}}^{{r_2}} {\frac{d}
{{dr}}{J_\varepsilon }(\eta (r,{\gamma _{\varepsilon }}(s)))} ds \hfill \\
   &\le \dis {\widetilde{{c_\varepsilon }}} - \frac{1}
{4}{\omega _{{d_1}}}{d_0} = Kc  - \frac{1}
{4}{\omega _{{d_1}}}{d_0} + o(1), \hfill \\
\end{array}
\]
where $ o(1) \to 0$ as $\varepsilon \to 0$.

In \textbf{Case 2.3}, by $(iii)$ and the definition of $r_1$, we have
\[
\begin{array}{ll}
  {J_\varepsilon }(\eta ({r_1},{\gamma _{\varepsilon }}(s))) \dis&= {J_\varepsilon }({\gamma _{\varepsilon }}(s)) + \dis\int_0^{{r_1}} {\frac{d}
{{dr}}{J_\varepsilon }(\eta (r,{\gamma _{\varepsilon }}(s)))} ds \hfill \\
    &\dis \le \widetilde{{c_\varepsilon }} - \frac{1}
{2}{\omega _{{d_1}}}{d_0} \dis = Kc  - \frac{1}
{2}{\omega _{{d_1}}}{d_0} + o(1). \hfill \\
\end{array}
\]
To sum up, choosing $\bar \mu  = \min \{ {{\bar \alpha}/2,{\omega _{{d_1}}}{d_0}/4} \} > 0$, we see that, for ${s \in {{[0,1]}^k}}$,
\begin{equation}\label{3.56}
{J_\varepsilon }({\xi _\varepsilon }(s)) \le Kc  - \bar \mu  + o(1).
\end{equation}
From \eqref{3.54} and \eqref{3.55}, we have
\begin{equation}\label{3.57}
{\| {{\xi _\varepsilon }(s)} \|_{H_0^1({\Omega _\varepsilon })}} \le C{\text{ for }}\varepsilon  > 0{\text{ small and }}s \in {[0,1]^k}.
\end{equation}
Choose $\beta_0  > 0$ small to ensure that
\begin{equation}\label{add2}
\beta_0  < \frac{1}
{2}\mathop {\min }\limits_{i \ne j} d({\Lambda ^i},{\Lambda ^j}) - {R_0},
\end{equation}
where $R_0$ has been mentioned in \eqref{add1} and let $\psi  \in C_c^\infty ({\mathbb{R}^N},[0,1])$ be such that $\psi (x) = 1$ for $x \in \widetilde\Lambda $, $\psi (x) = 0$ for $x \notin {(\widetilde\Lambda )^{{\beta _0}}}$ and $|\nabla \psi | \le C/\beta_0 $. We denote ${\xi _{\varepsilon ,1}}(s): = {\psi _\varepsilon }{\xi _\varepsilon }(s)$ and ${\xi _{\varepsilon ,2}}(s): = (1 - {\psi _\varepsilon }){\xi _\varepsilon }(s)$, where ${\psi _\varepsilon }(x) = \psi (\varepsilon x)$. From \eqref{3.56} and \eqref{3.57}, we see that
${Q_\varepsilon }({\xi _\varepsilon }(s))$ is uniformly bounded, then for $\varepsilon>0$ small and $s \in {[0,1]^k}$,
\begin{equation}\label{3.58}
\int_{{\Omega _\varepsilon }\backslash {{(\widetilde\Lambda )}_\varepsilon }} {|{\xi _\varepsilon }(s){|^{{2^ * }}}}  \le C\varepsilon .
\end{equation}
By interpolation inequality and \eqref{3.57}, we see that for each $p \in (2,{2^*}]$,
\begin{equation}\label{3.59}
\begin{array}{ll}
\dis  \int_{{\Omega _\varepsilon }} {|{\xi _{\varepsilon ,2}}(s){|^p}}  &\le \dis \left\| {{\xi _{\varepsilon ,2}}(s)} \right\|_{{L^{{2^*}}}({\Omega _\varepsilon })}^{\frac{{p - 2}}
{2}N}\left\| {{\xi _{\varepsilon ,2}}(s)} \right\|_{{L^2}({\Omega _\varepsilon })}^{2 + \frac{{p - 2}}
{2}(2 - N)} \hfill \\
   &\le \dis C\left\| {{\xi _\varepsilon }(s)} \right\|_{{L^{{2^*}}}({\Omega _\varepsilon }\backslash {{(\widetilde\Lambda )}_\varepsilon })}^{\frac{{p - 2}}
{2}N} = o(1), \hfill \\
\end{array}
\end{equation}
where $o(1) \to 0$ as $\varepsilon  \to 0$ uniformly for $s \in {[0,1]^k}$. By \eqref{3.57} and \eqref{3.59}, we have
\[
{J_\varepsilon }({\xi _\varepsilon }(s)) \ge {J_\varepsilon }({\xi _{\varepsilon ,1}}(s)) + {Q_\varepsilon }({\xi _\varepsilon }(s)) - {Q_\varepsilon }({\xi _{\varepsilon ,1}}(s)) + o(1).
\]
Since for $A$, $B \ge 0$, ${(A + B - 1)_ + } \ge {(A - 1)_ + } + {(B - 1)_ + }$, it follows that
\begin{equation}\label{3.60}
\begin{array}{ll}
{Q_\varepsilon }({\xi _\varepsilon }(s)) &= \dis\Bigl( {{\varepsilon ^{ - 1}}\int_{{\Omega _\varepsilon }\backslash {{(\widetilde\Lambda )}_\varepsilon }} {|{\xi _{\varepsilon ,1}}(s) + {\xi _{\varepsilon ,2}}(s){|^{{2^ * }}}}  - 1} \Bigr)_ + ^2 \hfill \\
   &\ge \dis\Bigl( {{\varepsilon ^{ - 1}}\int_{{\Omega _\varepsilon }\backslash {{(\widetilde\Lambda )}_\varepsilon }} {|{\xi _{\varepsilon ,1}}(s){|^{{2^ * }}}}  + {\varepsilon ^{ - 1}}\int_{{\Omega _\varepsilon }\backslash {{(\widetilde\Lambda )}_\varepsilon }} {|{\xi _{\varepsilon ,2}}(s){|^{{2^ * }}}}  - 1} \Bigr)_ + ^2 \hfill \\
  & \ge\dis \Bigl( {{\varepsilon ^{ - 1}}\int_{{\Omega _\varepsilon }\backslash {{(\widetilde\Lambda )}_\varepsilon }} {|{\xi _{\varepsilon ,1}}(s){|^{{2^ * }}}}  - 1} \Bigr)_ + ^2 + \Bigl( {{\varepsilon ^{ - 1}}\int_{{\Omega _\varepsilon }\backslash {{(\widetilde\Lambda )}_\varepsilon }} {|{\xi _{\varepsilon ,2}}(s){|^{{2^ * }}}}  - 1} \Bigr)_ + ^2 \hfill \\
   &= \dis{Q_\varepsilon }({\xi _{\varepsilon ,1}}(s)) + {Q_\varepsilon }({\xi _{\varepsilon ,2}}(s)). \hfill \\
\end{array}
\end{equation}
Thus,
\begin{equation}\label{3.61}
{J_\varepsilon }({\xi _\varepsilon }(s)) \ge {J_\varepsilon }({\xi _{\varepsilon ,1}}(s)) + o(1).
\end{equation}
We define $\xi _{\varepsilon ,1}^k(s)(x) = {\xi _{\varepsilon ,1}}(s)(x)$ for $x \in (\widetilde{{\Lambda ^k}})_\varepsilon ^{{\beta_0} /\varepsilon }$, $\xi _{\varepsilon ,1}^k(s)(x) = 0$ for $x \notin (\widetilde{{\Lambda ^k}})_\varepsilon ^{{\beta_0} /\varepsilon }$ for each $1 \le k \le K$. Similar to \eqref{3.60}, we get
\begin{equation}\label{3.62}
{J_\varepsilon }({\xi _{\varepsilon ,1}}(s)) \ge \sum\limits_{k = 1}^K {{J_\varepsilon }(\xi _{\varepsilon ,1}^k(s))}  + o(1) = \sum\limits_{k = 1}^K {J_\varepsilon ^k(\xi _{\varepsilon ,1}^k(s))}  + o(1).
\end{equation}

Next, we introduce some notations as in \cite{cr}. For $s = ({s_1},...,{s_K}) \in {[0,1]^K}$, letting ${0_k} = ({s_1},.,{s_{k - 1}},0,{s_{k + 1}},.,{s_K})$ and ${1_k} = ({s_1},.,{s_{k - 1}},1,{s_{k + 1}},.,{s_K})$, we see from Lemma~\ref{3.1.}(ii) and $(iv)$ in the proof of Lemma~\ref{3.7.} that $J_\varepsilon ^k(\xi _{\varepsilon ,1}^k({0_k})) = J_\varepsilon ^k(0) = 0$ and $J_\varepsilon ^k(\xi _{\varepsilon ,1}^k({1_k})) = J_\varepsilon ^k(W_{\varepsilon ,{t_0}}^k) < 0$ for $\varepsilon  > 0$ small by \eqref{3.3}. In view of the gluing method due to V. Coti Zelati and P. Rabinowitz (\cite{cr}, Proposition~3.4), there exists ${{\bar s}_\varepsilon } \in {[0,1]^K}$ such that
\begin{equation}\label{3.63}
J_\varepsilon ^k(\xi _{\varepsilon ,1}^k({{\bar s}_\varepsilon })) \ge c_\varepsilon ^k{\text{ for each }}1 \le k \le K.
\end{equation}
Moreover, by the definition of $J_\varepsilon ^k$, $c_\varepsilon ^k$, $I$ and $c$, we see that for $\varepsilon  > 0$ small and each $1 \le k \le K$,
\begin{equation}\label{3.64}
c_\varepsilon ^k \ge c.
\end{equation}
\eqref{3.61}, \eqref{3.62}, \eqref{3.63} and \eqref{3.64} yield
\[
\mathop {\max }\limits_{s \in {{[0,1]}^k}} {J_\varepsilon }({\xi _\varepsilon }(s)) \ge Kc + o(1),
\]
which contradicts \eqref{3.56} for $\varepsilon  > 0$ small.
\end{proof}

\subsection{Proof of  Theorem \ref{1.1.}}
\begin{proof}
\textbf{Step 1}. By Lemma~\ref{3.7.}, we see that for each $\varepsilon > 0$ small, there exists a sequence $\{ {v_{n,\varepsilon }}\} _{n = 1}^\infty  \subset J_\varepsilon ^{\widetilde{{c_\varepsilon }} + \sum\nolimits_{k = 1}^K {\exp ( - 2{D_k}/\varepsilon )} } \cap X_\varepsilon ^{{d_0}}$ such that ${J'_\varepsilon }({v_{n,\varepsilon }}) \to 0$ in ${(H_0^1({\Omega _\varepsilon }))^{ - 1}}$ as $n \to \infty $. Since $\{ {v_{n,\varepsilon }}\} _{n = 1}^\infty $ is bounded in $H_0^1({\Omega _\varepsilon })$, up to a subsequence, as $n \to \infty $,
\begin{equation}\label{3.65}
{v_{n,\varepsilon }} \rightharpoonup {v_\varepsilon }{\text{ in }}H_0^1({\Omega _\varepsilon }){\text{ and }}{v_{n,\varepsilon }} \to {v_\varepsilon }~{\text{a.e.}}~{\Omega _\varepsilon }.
\end{equation}
Next, we claim that for each $\varepsilon>0$ small,
\begin{equation}\label{3.66}
{v_{n,\varepsilon }} \to {v_\varepsilon }{\text{ in }}{L^{{2^ * }}}({\Omega _\varepsilon }){\text{ as }}n \to \infty .
\end{equation}
Indeed, for each $\varepsilon > 0$ small, choosing $R>0$ such that ${(\widetilde\Lambda )_\varepsilon } \subset {B_R}(0)$ and defining ${\phi _R}(x) \in {C^\infty }({\mathbb{R}^N},[0,1])$ satisfying ${\phi _R}(x)=0$ on ${B_R}(0)$, ${\phi _R}(x)=1$ on ${\mathbb{R}^N}\backslash {B_{2R}}(0)$ and $|\nabla {\phi _R}| \le C/R$. Since for each $\varepsilon > 0$ small, $\{ {\phi _R}{v_{n,\varepsilon }}\} _{n = 1}^\infty $ is bounded in $H_0^1({\Omega _\varepsilon })$, it follows that
\[
\langle {{J'_\varepsilon }({v_{n,\varepsilon }}),{\phi _R}{v_{n,\varepsilon }}} \rangle  \to 0{\text{ as }}n \to \infty .
\]
By the definition of ${g_\varepsilon }(x,t)$, we have
\begin{equation}\label{3.67}
\int_{{\Omega _\varepsilon }} {|\nabla {v_{n,\varepsilon }}{|^2}{\phi _R}}  + \int_{{\Omega _\varepsilon }} {v_{n,\varepsilon }^2{\phi _R}}  \le \frac{1}
{2}\int_{{\Omega _\varepsilon }} {v_{n,\varepsilon }^2{\phi _R}}  - \int_{{\Omega _\varepsilon }} {(\nabla {v_{n,\varepsilon }} \cdot \nabla {\phi _R}){v_{n,\varepsilon }}}.
\end{equation}
By \eqref{3.65}, as $R \to +\infty$,
\begin{equation}\label{3.68}
\begin{gathered}
\quad  \mathop {\overline {\lim } }\limits_{n \to \infty } \int_{{\Omega _\varepsilon }} {|\nabla {v_{n,\varepsilon }}| \cdot |\nabla {\phi _R}| \cdot |{v_{n,\varepsilon }}|}  \hfill \\
   \le \mathop {\overline {\lim } }\limits_{n \to \infty } {\Bigl( {\int_{{\Omega _\varepsilon }} {|\nabla {v_{n,\varepsilon }}{|^2}} } \Bigr)^{1/2}}{\Bigl( {\int_{{\Omega _\varepsilon } \cap ({B_{2R}}(0)\backslash {B_R}(0))} {|{v_{n,\varepsilon }}{|^2}|\nabla {\phi _R}{|^2}} } \Bigr)^{1/2}} \hfill \\
   \le C{\Bigl( {\int_{{\Omega _\varepsilon } \cap ({B_{2R}}(0)\backslash {B_R}(0))} {|{v_\varepsilon }{|^2}|\nabla {\phi _R}{|^2}} } \Bigr)^{1/2}} \hfill \\
   \le C{\Bigl( {\int_{{\Omega _\varepsilon } \cap ({B_{2R}}(0)\backslash {B_R}(0))} {|{v_\varepsilon }{|^{{2^ * }}}} } \Bigr)^{1/{2^ * }}}{\Bigl( {\int_{{B_{2R}}(0)\backslash {B_R}(0)} {|\nabla {\phi _R}{|^N}} } \Bigr)^{1/N}} \to 0. \hfill \\
\end{gathered}
\end{equation}
By \eqref{3.67} and \eqref{3.68}, we see that for $\varepsilon  > 0$ small,
\begin{equation}\label{3.69}
\mathop {\overline {\lim } }\limits_{R \to  + \infty } \mathop {\overline {\lim } }\limits_{n \to \infty } \int_{{\Omega _\varepsilon }\backslash {B_{2R}}(0)} {|\nabla {v_{n,\varepsilon }}{|^2} + |{v_{n,\varepsilon }}{|^2}}  = 0.
\end{equation}
In order to prove \eqref{3.66}, it suffices to show that for each $1 \le k \le K$ and $\varepsilon  > 0$ small,
\begin{equation}\label{3.70}
{v_{n,\varepsilon }} \to {v_\varepsilon }{\text{ in }}{L^{{2^ * }}}\Bigl( {{\Omega _\varepsilon }\backslash \Bigl( {\mathop  \cup \limits_{l \ne k} (\widetilde{{\Lambda ^l}})_\varepsilon ^{{\beta_0} /\varepsilon }} \Bigr)} \Bigr) \text{ as } n \to \infty,
\end{equation}
where ${\beta_0}$ has been mentioned in \eqref{add2}. Choosing ${\eta _{\varepsilon,k}}(x) \in {C^\infty }({\mathbb{R}^N},[0,1])$ satisfying ${\eta _{\varepsilon,k}}(x) = 1$ for $x \in {\bigl( {{\Omega _\varepsilon }\backslash \bigl( {\mathop  \cup \nolimits_{l \ne k} (\widetilde{{\Lambda ^l}})_\varepsilon ^{{\beta_0}/\varepsilon }} \bigr)} \bigr)^{{\beta_0} /2\varepsilon }}$, ${\text{supp}}{\eta _{\varepsilon,k}}(x) \subset {\bigl( {{\Omega _\varepsilon }\backslash \bigl( {\mathop  \cup \nolimits_{l \ne k} (\widetilde{{\Lambda ^l}})_\varepsilon ^{{\beta_0} /\varepsilon }} \bigr)} \bigr)^{{\beta_0} /\varepsilon }}$ and $|\nabla {\eta _{\varepsilon,k}}| \le C\varepsilon$. By \eqref{3.69}, we see that the sequence $\{ {\eta _{\varepsilon,k}}{v_{n,\varepsilon }}\} _{n = 1}^\infty $ is tight in $H_0^1({\Omega _\varepsilon })$. From \eqref{3.65}, we may assume that
\[
|\nabla ({\eta _{\varepsilon,k}}{v_{n,\varepsilon }}){|^2} \rightharpoonup |\nabla ({\eta _{\varepsilon,k}}{v_\varepsilon }){|^2} + \mu {\text{ and }}|{\eta _{\varepsilon,k}}{v_{n,\varepsilon }}{|^{{2^*}}} \rightharpoonup |{\eta _{\varepsilon,k}}{v_\varepsilon }{|^{{2^*}}} + \nu ,
\]
where $\mu$ and $\nu$ are two bounded nonnegative measures on ${\mathbb{R}^N}$.  By the concentration compactness principle II (Lemma I.1 of \cite{l}), we obtain an at most countable index set $\Gamma $, sequences ${\{ {x_j}\} _{j \in \Gamma }} \subset {\mathbb{R}^N}$ and ${\{ {\mu _j}\} _{j \in \Gamma }},{\{ {\nu _j}\} _{j \in \Gamma }} \subset (0,\infty )$ such that
\begin{equation}\label{3.71}
\mu  \ge \sum\limits_{j \in \Gamma } {{\mu _j}} {\delta _{{x_j}}},\nu  = \sum\limits_{j \in \Gamma } {{\nu _j}} {\delta _{{x_j}}}{\text{ and }}S{({\nu _j})^{(N - 2)/N}} \le {\mu _j}.
\end{equation}
To prove \eqref{3.70}, we need to show that ${\{ {x_j}\} _{j \in \Gamma }} \cap \overline {\bigl( {{\Omega _\varepsilon }\backslash \bigl( {\mathop  \cup \nolimits_{l \ne k} (\widetilde{{\Lambda ^l}})_\varepsilon ^{{\beta_0} /\varepsilon }} \bigr)} \bigr)}  = \emptyset$. Suppose by contradiction that there is an $j_0 \in \Gamma$ such that  ${x_{{j_0}}} \in \overline {\bigl( {{\Omega _\varepsilon }\backslash \bigl( {\mathop  \cup \nolimits_{l \ne k} (\widetilde{{\Lambda ^l}})_\varepsilon ^{{\beta_0} /\varepsilon }} \bigr)} \bigr)}$, arguing as \eqref{3.22}-\eqref{3.25} in Proposition~\ref{3.4.}, we get ${\nu _{{j_0}}} \ge {S^{N/2}}$. Since $\{ {v_{n,\varepsilon }}\} _{n = 1}^\infty  \subset X_\varepsilon ^{{d_0}}$, by the definition of $X_\varepsilon ^{{d_0}}$ and the compactness of $S$ and $\overline {{\Lambda ^k}} $, we see that for each $1 \le k \le K$, there exist ${U^k} \in S$, ${x^k} \in \overline {{\Lambda ^k}} $ such that for $n$ large,
\begin{equation}\label{3.72}
{\left\| {{v_{n,\varepsilon }} - \sum\limits_{k = 1}^K {{\varphi _{\varepsilon ,k}}(x){U^k}(x - ({x^k}/\varepsilon ))} } \right\|_{H_0^1({\Omega _\varepsilon })}} \le 2{d_0}.
\end{equation}
Similar to the argument in  \eqref{3.58}-\eqref{3.61}, we have
\begin{equation}\label{3.73}
\widetilde{{c_\varepsilon }} + \sum\limits_{k = 1}^K {\exp \Bigl( { - \frac{{2{D_k}}}
{\varepsilon }} \Bigr)}  \ge {J_\varepsilon }({v_{n,\varepsilon }}) \ge {J_\varepsilon }({\psi _{\varepsilon ,k}}{v_{n,\varepsilon }}) + {J_\varepsilon }((1 - {\psi _{\varepsilon ,k}}){v_{n,\varepsilon }}) + o(1),
\end{equation}
where $o(1) \to 0$ as $n \to \infty$. From \eqref{3.71} and \eqref{3.72}, we get
\begin{equation}\label{3.74}
\begin{gathered}
\quad  {J_\varepsilon }({\psi _{\varepsilon ,k}}{v_{n,\varepsilon }}) \hfill \\
   = {J_\varepsilon }({\psi _{\varepsilon ,k}}{v_{n,\varepsilon }}) - \frac{1}
{N}P({U^k}) \hfill \\
   \ge \frac{1}
{N}\int_{{\Omega _\varepsilon }} {|\nabla ({\psi _{\varepsilon ,k}}{v_{n,\varepsilon }}){|^2}}  + \frac{{N - 2}}
{{2N}}\Bigl( {\int_{{\Omega _\varepsilon }} {|\nabla ({\psi _{\varepsilon ,k}}{v_{n,\varepsilon }}){|^2}}  - \int_{{\mathbb{R}^N}} {|\nabla {U^k}{|^2}} } \Bigr) \hfill \\
\quad   + \frac{1}
{2}\Bigl( {\int_{{\Omega _\varepsilon }} {|{\psi _{\varepsilon ,k}}{v_{n,\varepsilon }}{|^2}}  - \int_{{\mathbb{R}^N}} {|{U^k}{|^2}} } \Bigr) - \Bigl( {\int_{{\Omega _\varepsilon }} {F({\psi _{\varepsilon ,k}}{v_{n,\varepsilon }})}  - \int_{{\mathbb{R}^N}} {F({U^k})} } \Bigr) \hfill \\
   \ge \frac{1}
{N}{\mu _{{j_0}}} - C{d_0} + o(1) \ge \frac{1}
{N}S{({\nu _{{j_0}}})^{(N - 2)/N}} - C{d_0} + o(1) \ge \frac{1}
{N}{S^{N/2}} - C{d_0} + o(1). \hfill \\
\end{gathered}
\end{equation}
Moreover, for each $l \ne k$, we define $w_{n,\varepsilon }^l(x) = ((1 - {\psi _{\varepsilon ,k}}(x)){v_{n,\varepsilon }}(x))$ for $x \in (\widetilde{{\Lambda ^l}})_\varepsilon ^{{\beta_0} /\varepsilon }$ and $w_{n,\varepsilon }^l(x) = 0$ for $x \notin (\widetilde{{\Lambda ^l}})_\varepsilon ^{{\beta_0}/\varepsilon }$, by \eqref{3.72}, we have
\begin{equation}\label{3.75}
{J_\varepsilon }((1 - {\psi _{\varepsilon ,k}}){v_{n,\varepsilon }}) = \sum\limits_{l \ne k} {J_\varepsilon ^l(w_{n,\varepsilon }^l)}  \ge \sum\limits_{l \ne k} {I({U^l})}  - C{d_0} + o(1) = (K - 1)c - C{d_0} + o(1).
\end{equation}
In view of \eqref{3.73}, \eqref{3.74}, \eqref{3.75} and Lemma~\ref{3.1.}(i), letting $\varepsilon \to 0$ and ${d_0} \to 0$, we see that $c \ge \frac{1}
{N}{S^{N/2}}$, a contradiction. Hence, \eqref{3.70} holds. By \eqref{3.70}, we get \eqref{3.66}. Using standard argument, we can verify that ${v_{n,\varepsilon }} \to {v_\varepsilon } \in J_\varepsilon ^{\widetilde{{c_\varepsilon }} + \sum\nolimits_{k = 1}^K {\exp ( - 2{D_k}/\varepsilon )} } \cap X_\varepsilon ^{{d_0}}$ in $H_0^1({\Omega _\varepsilon })$ and ${v_\varepsilon }$ satisfies
\begin{equation}\label{3.76}
 - \Delta {v_\varepsilon } + {v_\varepsilon } + 2 \cdot {2^ * }{\Bigl( {\int_{{\Omega _\varepsilon }} {{\chi _\varepsilon }|{v_\varepsilon }{|^{{2^ * }}}}  - 1} \Bigr)_ + }{\chi _\varepsilon }|{v_\varepsilon }{|^{{2^ * } - 2}}{v_\varepsilon } = {g_\varepsilon }(x,{v_\varepsilon }){\text{ in }}{\Omega _\varepsilon }.
\end{equation}
Since $c>0$, we see that $0 \notin X_\varepsilon ^{{d_0}}$ for ${d_0} > 0$ small. Thus ${v_\varepsilon } \ne 0$.

\noindent\textbf{Step 2}. For any sequence $\{ {\varepsilon _j}\} _{j = 1}^\infty $ with ${\varepsilon _j} \to 0$, by Proposition~\ref{3.4.}, there exist, up to a subsequence, $\{ y_{{\varepsilon _j}}^k\} _{j = 1}^\infty  \subset {({\Lambda ^k})_{{\varepsilon _j}}}$, ${U^k} \in S (1 \le k \le K)$ such that
\begin{equation}\label{3.77}
\mathop {\lim }\limits_{j \to \infty } {\left\| {{v_{{\varepsilon _j}}} - \sum\limits_{k = 1}^K {{\varphi _{{\varepsilon _j},k}}(x){U^k}(x - y_{{\varepsilon _j}}^k)} } \right\|_{H_0^1({\Omega _{{\varepsilon _j}}})}} = 0.
\end{equation}
In the following, we adopt an iteration method in \cite{Barrios15} to prove
\begin{equation}\label{3.78}
{\| {{v_{{\varepsilon _j}}}} \|_{{L^\infty }({\Omega _{{\varepsilon _j}}})}} \le C{\text{ uniformly for }}{\varepsilon _j}.
\end{equation}
For the sake of simplicity, we denote $v_{{\varepsilon _j}}$, ${\Omega _{{\varepsilon _j}}}$ by $v_j$ and $\Omega_j$ respectively. For any $T, \kappa>1$, set
$$
\psi(t)=\psi_{T,\kappa}(t)=\left\{
  \begin{array}{ll}
   0,\quad&t\leq0 \\
   t^{\kappa} \quad&0<t<T,\\
   \kappa T^{\kappa-1}(t-T)+T^{\kappa} \quad&t\geq T,
 \end{array}
\right.
$$
then $\psi$ is convex, differentiable and $-\Delta\psi(v_j)\leq\psi'(v_j)(-\Delta v_j),\ x\in\Omega_j$ in the weak sense for any $j$. Obviously, $\psi(v_j)\in H_0^1(\Omega_j)$ and $\|\nabla\psi(v_j)\|_2\leq \kappa T^{\kappa-1}\|\nabla v_j\|_2$. Then by Sobolev's embedding, $\|\nabla\psi(v_j)\|_2^2\geq S\|\psi(v_j)\|_{2^*}^2$, where $S$ is the best Sobolev constant for the imbedding ${D^{1,2}}({\mathbb{R}^N}) \hookrightarrow {L^{{2^ * }}}({\mathbb{R}^N})$. Meanwhile, noting that $v_j$ is positive and solves (\ref{3.76}), it follows from ($F_1$) and ($F_2$) that for some $c>0$ (independent of $j$) such that
$
-\Delta \psi(v_j)\leq c \psi'(v_j)v_j^{2^*-1}, x\in\Omega_j
$
in the weak sense. Then,
$
\int_{\Omega_j}|\nabla\psi(v_j)|^2\leq c\int_{\Omega_j}\psi(v_j)\psi'(v_j)v_j^{2^*-1}.
$
which implies that
$
\|\psi(v_j)\|_{2_s^*}^2\leq\frac{c}{S}\int_{\Omega_j}\psi(v_j)\psi'(v_j)v_j^{2^*-1}.
$
Since $t\psi'(t)\leq\kappa\psi(t)$, $t\ge0$,
\begin{equation}\label{guji1}
\|\psi(v_j)\|_{2_s^*}^2\leq \frac{c\kappa}{S}\int_{\Omega_j}\psi^2(v_j)v_j^{2^*-2}.
\end{equation}
Let $\kappa=:\kappa_1=2^*/2$, we claim that $v_j\in L^{\kappa_12^*}(\Omega_j)$.
Indeed, for any $R>0$,
\begin{equation*}
\int_{\Omega_j}\psi^2(v_j)v_j^{2^*-2}\leq\int_{\{v_j\leq R\}}\psi^2(v_j)R^{2^*-2}+\|\psi(v_j)\|_{2^*}^2\left(\int_{\{v_j\geq R\}}v_j^{2^*}\right)^{\frac{2^*-2}{2^*}}.
\end{equation*}
For any $j$, there exists $R>0$ such that
$$
\left(\int_{\{v_j\geq R\}}v_j^{2^*}\right)^{\frac{2^*-2}{2^*}}\le\frac{S}{2c\kappa_1},
$$
which gives that
$
\|\psi(v_j)\|_{2^*}^2\leq 2R^{2^*-2}c\kappa_1/S\int_{\{v_j\leq R\}}\psi^2(v_j).
$
Let $T\rightarrow\infty$, we have $v_j\in L^{\kappa_1 2^*}(\Omega_j)$ and
$
\|v_j\|_{\kappa_12^*}^{2\kappa_1}\leq 2R^{2^*-2}c\kappa_1/S\|v_j\|_{2^*}^{2^*}<\infty.
$
Then it follows from \eqref{guji1} that letting $T\rightarrow\infty$, we have
$
\|v_j\|_{\kappa_1 2^*}^{2\kappa_1}\leq c\kappa_1/S\int_{\Omega_j}v_j^{2\kappa_1+2^*-2}.
$
So, let $C_1=c\kappa_1/S$,
$$
\left(\int_{\Omega_j}v_j^{\kappa_1 2^*}\right)^{\frac{1}{(\kappa_1-1) 2^*}}\leq C_1^{\frac{1}{2(\kappa_1-1)}}\left(\int_{\Omega_j}v_j^{2\kappa_1+2^*-2}\right)^{\frac{1}{2(\kappa_1-1)}}.
$$
For any $m\ge2$, set $\kappa_{i}=(2^*\kappa_{i-1}-2^*+2)/2, i=2,3,\cdots,m$, i. e., $\kappa_i=\left(\frac{2^*}{2}\right)^{i-1}(\kappa_1-1)+1$. Similarly as above, for $i=2,3,\cdots,m$,
$$
\left(\int_{\Omega_j}v_j^{\kappa_{i} 2^*}\right)^{\frac{1}{(\kappa_{i}-1) 2^*}}\leq C_i^{\frac{1}{2(\kappa_{i}-1)}}\left(\int_{\Omega_j}v_j^{2^*\kappa_{i-1}}\right)^{\frac{1}{2^*(\kappa_{i-1}-1)}},
$$
where $C_i=c\kappa_i/S$. Then, by iteration $v_j\in L^{t}(\Omega_j)$ for any $t\ge2$. Let $D_i:=\left(\int_{\Omega_j}v_j^{2^*\kappa_{i}}\right)^{\frac{1}{2^*(\kappa_i-1)}}, i=1,2,\cdots,m$, we have
$$
D_{m}\leq\prod\limits_{i=1}^{m-1}C_i^{\frac{1}{2(\kappa_{i}-1)}}\cdot D_1.
$$
Let $m\rightarrow\infty$, $\|v_j\|_{L^\infty(\Omega_j)}\leq \prod\limits_{i=1}^{\infty}C_i^{\frac{1}{2(\kappa_{i}-1)}}\cdot D_1<\infty$. Due to $v_j\in X_{\varepsilon_j}^{d_0}$, the claim is concluded.

It follows from \cite[Theorem~8.17]{gt} that, for any $y\in\Omega_{\varepsilon_j}$ with ${B_2}(y)\subset\Omega_{\varepsilon_j}$,
\begin{equation}\label{3.95}
\begin{array}{ll}
  \mathop {\sup }\limits_{{B_1}(y)} {v_{{\varepsilon _j}}}\le C \bigl( {{{\| {{v_{{\varepsilon _j}}}} \|}_{{L^2}({B_2}(y))}} + \| {{v_{{\varepsilon _j}}}} \|_{{L^2}({B_2}(y))}^{2/N}} \bigr). \hfill \\
\end{array}
\end{equation}
Moreover, thanks to \eqref{3.77}, for each $R>0$,
\begin{equation}\label{3.96}
{\| {{v_{{\varepsilon _j}}}} \|_{{L^2}\bigl({\Omega _{{\varepsilon _j}}}\backslash \mathop  \cup \limits_{k = 1}^K {B_R}(y_{{\varepsilon _j}}^k)\bigr)}} \le \sum\limits_{k = 1}^K {{{\| {{U^k}} \|}_{{L^2}({\mathbb{R}^N}\backslash {B_R}(0))}}}  + o(1),
\end{equation}
where $o(1) \to 0$ as $j \to \infty$. We obtain from \eqref{3.95} and \eqref{3.96} that for any $\delta  > 0$, there exists $R = R_\delta  > 0$ such that
\begin{equation}\label{3.97}
|{v_{{\varepsilon _j}}}(x)| < \delta {\text{ uniformly for }}x \in {\Omega _{{\varepsilon _j}}}\backslash \mathop  \cup \limits_{k = 1}^K {B_{R_\delta}}(y_{{\varepsilon _j}}^k){\text{ and }}{\varepsilon _j} > 0{\text{ small}}.
\end{equation}
By \eqref{3.97}, we see from maximum principle and Schauder's estimate that for any $\mu  \in (0,1)$,
\begin{equation}\label{3.98}
{v_{{\varepsilon _j}}}(x) + |\nabla {v_{{\varepsilon _j}}}(x)| \le {C_\mu }\sum\limits_{k = 1}^K {{e^{ - \mu |x - y_{{\varepsilon _j}}^k|}}} ,
\end{equation}
where ${C_\mu}>0$ is independent of ${{\varepsilon _j}}$. Fixing $\beta ' > 0$ satisfying
$$
\beta ' < \mathop {\min }\limits_{1 \le k \le K} \mathop {\inf }\limits_{x \in {\Lambda ^k}} d(x,\partial \widetilde{{\Lambda ^k}}),
$$
then as $j \to \infty$,
\[
\varepsilon _j^{ - 1}\int_{{\Omega _{{\varepsilon _j}}}\backslash {{(\widetilde\Lambda )}_{{\varepsilon _j}}}} {v_{{\varepsilon _j}}^{{2^ * }}}  \le {C_\mu}\varepsilon _j^{ - 1}\sum\limits_{k = 1}^K {\int_{{\mathbb{R}^N}\backslash ({{(\widetilde\Lambda )}_{{\varepsilon _j}}} - y_{{\varepsilon _j}}^k)} {{e^{ - {2^ * }{\mu}|x|}}} }  \le {C_\mu}\varepsilon _j^{ - 1}\sum\limits_{k = 1}^K {\int_{{\mathbb{R}^N}\backslash {B_{\beta '/{\varepsilon _j}}}(0)} {{e^{ - {2^ * }{\mu}|x|}}} }  \to 0,
\]
i.e. ${Q_{{\varepsilon _j}}}({v_{{\varepsilon _j}}}) = 0$ for ${{\varepsilon _j}}$ small. Therefore, ${v_{{\varepsilon _j}}}$ satisfies
\begin{equation}\label{3.99}
 - \Delta {v_{{\varepsilon _j}}} + {v_{{\varepsilon _j}}} = {g_{{\varepsilon _j}}}(x,{v_{{\varepsilon _j}}}){\text{ in }}{\Omega _{{\varepsilon _j}}}.
\end{equation}
\textbf{Step 3}. Denoting ${d_{{\varepsilon _j},k}} := d({\varepsilon _j}y_{{\varepsilon _j}}^k,\partial \Omega )$, we show that
\begin{equation}\label{3.100}
{d_{{\varepsilon _j},k}} \to {D_k}{\text{ as }}j \to \infty,
\end{equation}
which plays a crucial role in proving Theorem~\ref{1.1.}. The proof is similar to \cite{b1}. However, to construct multi-peak solutions which concentrate around the prescribed finite sets of local maxima of the distance function $d( \cdot ,\partial \Omega )$, the problem becomes more complicated than that in \cite{b1}, where only the global concentrating single-peak solutions were considered.

In the following, we extend $u \in H_0^1({\Omega _{\varepsilon _j} })$ to $u \in {H^1}({\mathbb{R}^N})$ by zero outside ${\Omega _{\varepsilon _j} }$ if necessary. For each $1 \le k \le K$ and ${\beta _{1,k}} > {\beta _{2,k}} > 0$ small but fixed to ensure that $({B_{{d_{{\varepsilon _j},k + {\beta _{1,k}}}}}}({\varepsilon _j}y_{{\varepsilon _j}}^k) \cap \Omega ) \subset \widetilde{{\Lambda ^k}}$ for ${\varepsilon _j} >0$ small. Moreover, we define a smooth cut-off function ${\psi _{{\varepsilon _j},k}}(x) \in C_c ^\infty ( {{B_{({d_{{\varepsilon _j},k}} + {\beta _{1,k}})/{\varepsilon _j}}}(y_{{\varepsilon _j}}^k),[0,1]} )$ satisfying ${\psi _{{\varepsilon _j},k}}(x) = 1$ on ${B_{({d_{{\varepsilon _j},k}} + {\beta _{2,k}})/{\varepsilon _j}}}(y_{{\varepsilon _j}}^k)$, $|\nabla {\psi _{{\varepsilon _j},k}}| \le C{\varepsilon _j}$ and $|{\nabla ^2}{\psi _{{\varepsilon _j},k}}| \le C\varepsilon _j^2$. From \eqref{3.98}, we see that for any $\delta  \in (0,1)$, there exists some $C_\delta>0$ such that
\begin{equation}\label{3.101}
\left| {{E_{{\varepsilon _j}}}({v_{{\varepsilon _j}}}) - \sum\limits_{k = 1}^K {E_{{\varepsilon _j}}^k({\psi _{{\varepsilon _j},k}}{v_{{\varepsilon _j}}})} } \right| \le {C_\delta }\sum\limits_{k = 1}^K {\exp \Bigl( { - \frac{{2\delta }}
{{{\varepsilon _j}}}({d_{{\varepsilon _j},k}} + {\beta _{2,k}})} \Bigr)} .
\end{equation}
Indeed, choosing $\delta ' = \delta '(\delta ) > 0$ such that $\delta  + \delta ' \in (0,1)$, in view of \eqref{3.98}, we see that $\exists {C_\delta } > 0$ such that
\begin{equation}\label{3.102}
{v_{{\varepsilon _j}}}(x) + |\nabla {v_{{\varepsilon _j}}}(x)| \le {C_\delta }\sum\limits_{k = 1}^K {{e^{ - (\delta  + \delta ')|x - y_{{\varepsilon _j}}^k|}}} ,
\end{equation}
then \eqref{3.78}, \eqref{3.102}, $(F_1)$ and $(F_2)$ yield
\[
\begin{array}{ll}
\left| {{E_{{\varepsilon _j}}}({v_{{\varepsilon _j}}}) - \sum\limits_{k = 1}^K {E_{{\varepsilon _j}}^k({\psi _{{\varepsilon _j},k}}{v_{{\varepsilon _j}}})} } \right| &\le \dis C\int_{{\Omega _{{\varepsilon _j}}}\backslash \mathop  \cup \limits_{k = 1}^K {B_{({d_{{\varepsilon _j},k}} + {\beta _{2,k}})/{\varepsilon _j}}}(y_{{\varepsilon _j}}^k)} {|{v_{{\varepsilon _j}}}{|^2} + |\nabla {v_{{\varepsilon _j}}}{|^2}}  \hfill \\
   &\le \dis{C_\delta }\sum\limits_{k = 1}^K {\int_{{\mathbb{R}^N}\backslash {B_{({d_{{\varepsilon _j},k}} + {\beta _{2,k}})/{\varepsilon _j}}}(y_{{\varepsilon _j}}^k)} {{e^{ - 2(\delta  + \delta ')|x - y_{{\varepsilon _j}}^k|}}} }  \hfill \\
  & \le \dis {C_\delta }\Bigl( {\sum\limits_{k = 1}^K {\exp \Bigl( { - \frac{{2\delta }}
{{{\varepsilon _j}}}({d_{{\varepsilon _j},k}} + {\beta _{2,k}})} \Bigr)} } \Bigr)\int_{{\mathbb{R}^N}} {{e^{ - 2\delta '|x|}}}  \hfill \\
  & \le \dis {C_\delta }\sum\limits_{k = 1}^K {\exp \Bigl( { - \frac{{2\delta }}
{{{\varepsilon _j}}}({d_{{\varepsilon _j},k}} + {\beta _{2,k}})} \Bigr)} . \hfill \\
\end{array}
\]
Letting ${z_{{\varepsilon _j},k}}(x): = {\psi _{{\varepsilon _j},k}}(x + y_{{\varepsilon _j}}^k){v_{{\varepsilon _j}}}(x + y_{{\varepsilon _j}}^k)$, following \cite{b}, we define
\begin{equation}\label{add15}
z_{{\varepsilon _j},k}^t(x): = {z_{{\varepsilon _j},k}}(y(x,t)),~x \in {B_{({d_{{\varepsilon _j},k}} + {\beta _{1,k}})/{\varepsilon _j}}}(0),~t \in (0, + \infty )
\end{equation}
by
\begin{equation}\label{add11}
y(x,t): =\dis \frac{x}
{{t +\dis (1 - t)\frac{{{\varepsilon _j}|x|}}
{{{d_{{\varepsilon _j},k}} + {\beta _{1,k}}}}}}.
\end{equation}
As mentioned in \cite{b}, for each $t \in (0, + \infty )$, $y(x,t)$ is a diffeomorphism from ${B_{({d_{{\varepsilon _j},k}} + {\beta _{1,k}})/{\varepsilon _j}}}(0)$ onto ${B_{({d_{{\varepsilon _j},k}} + {\beta _{1,k}})/{\varepsilon _j}}}(0)$. Moreover, by \eqref{add11},
\begin{equation}\label{add13}
x(y,t): =\dis \frac{y}
{{\bigl( {1 - \frac{1}
{t}} \bigr)\dis\frac{{{\varepsilon _j}|y|}}
{{{d_{{\varepsilon _j},k}} + {\beta _{1,k}}}} +\dis \frac{1}
{t}}}
\end{equation}
for $t \in (0, + \infty )$. Since
\[
 - {z_{{\varepsilon _j},k}}(x - y_{{\varepsilon _j}}^k) + {z_{{\varepsilon _j},k}}(x - y_{{\varepsilon _j}}^k) = {g_{{\varepsilon _j}}}(x,{z_{{\varepsilon _j},k}}(x - y_{{\varepsilon _j}}^k)){\text{ in }}{\Omega _{{\varepsilon _j}}} \cap {B_{({d_{{\varepsilon _j},k}} + {\beta _{2,k}})/{\varepsilon _j}}}(y_{{\varepsilon _j}}^k),
\]
it follows that
\[\begin{gathered}
\quad {\left. {\frac{{dE_{{\varepsilon _j}}^k(z_{{\varepsilon _j},k}^t(x - y_{{\varepsilon _j}}^k))}}
{{dt}}} \right|_{t = 1}} \hfill \\
   = \int_{\partial ({B_{({d_{{\varepsilon _j},k}} + {\beta _{2,k}})/{\varepsilon _j}}}(0) \cap ({\Omega _{{\varepsilon _j}}}-  y_{{\varepsilon _j}}^k ))} {\frac{{\partial {z_{{\varepsilon _j},k}}}}
{{\partial \nu }}(x \cdot \nabla {z_{{\varepsilon _j},k}})\frac{{{\varepsilon _j}|x| - {d_{{\varepsilon _j},k}} - {\beta _{1,k}}}}
{{{d_{{\varepsilon _j},k}} + {\beta _{1,k}}}}} dS(x) \hfill \\
 \quad  + \int_{({B_{({d_{{\varepsilon _j},k}} + {\beta _{1,k}})/{\varepsilon _j}}}(0)\backslash {B_{({d_{{\varepsilon _j},k}} + {\beta _{2,k}})/{\varepsilon _j}}}(0)) \cap ({\Omega _{{\varepsilon _j}}} -  y_{{\varepsilon _j}}^k )} {\nabla {z_{{\varepsilon _j},k}}\nabla \Bigl( {(x \cdot \nabla {z_{{\varepsilon _j},k}})\frac{{{\varepsilon _j}|x| - {d_{{\varepsilon _j},k}} - {\beta _{1,k}}}}
{{{d_{{\varepsilon _j},k}} + {\beta _{1,k}}}}} \Bigr)} dx \hfill \\
\quad   + \int_{({B_{({d_{{\varepsilon _j},k}} + {\beta _{1,k}})/{\varepsilon _j}}}(0)\backslash {B_{({d_{{\varepsilon _j},k}} + {\beta _{2,k}})/{\varepsilon _j}}}(0)) \cap ({\Omega _{{\varepsilon _j}}}-  y_{{\varepsilon _j}}^k )} {({z_{{\varepsilon _j},k}} - {g_{{\varepsilon _j}}}(x + y_{{\varepsilon _j}}^k,{z_{{\varepsilon _j},k}}))}  \hfill \\
~\quad   \cdot \Bigl( {(x \cdot \nabla {z_{{\varepsilon _j},k}})\frac{{{\varepsilon _j}|x| - {d_{{\varepsilon _j},k}} - {\beta _{1,k}}}}
{{{d_{{\varepsilon _j},k}} + {\beta _{1,k}}}}} \Bigr)dx \hfill \\
   = (I) + (II) + (III), \hfill \\
\end{gathered} \]
where $\nu$ and $dS(x)$ are the outward normal and the volume element on $\partial ({B_{({d_{{\varepsilon _j},k}} + {\beta _{2,k}})/{\varepsilon _j}}}(0) \cap ({\Omega _{{\varepsilon _j}}}-  y_{{\varepsilon _j}}^k ))$, respectively. Note that $\partial \Omega  \in {C^2}$ and $\{ {\varepsilon _j}y_{{\varepsilon _j}}^k\} _{j = 1}^\infty $ and $\{ {d_{{\varepsilon _j},k}}\} _{j = 1}^\infty $ are bounded, we see that $S(\partial ({B_{({d_{{\varepsilon _j},k}} + {\beta _{2,k}})/{\varepsilon _j}}}(y_{{\varepsilon _j}}^k) \cap {\Omega _{{\varepsilon _j}}})) \le C/\varepsilon _j^{N - 1}$, where $C$ is independent of ${\varepsilon _j} >0$ small. Together with \eqref{3.102}, we have
\[
\begin{array}{ll}
  |(I)| &\le \dis C\int_{\partial ({B_{({d_{{\varepsilon _j},k}} + {\beta _{2,k}})/{\varepsilon _j}}}(y_{{\varepsilon _j}}^k) \cap {\Omega _{{\varepsilon _j}}})} {|\nabla {z_{{\varepsilon _j},k}}(x - y_{{\varepsilon _j}}^k){|^2}|x - y_{{\varepsilon _j}}^k|} dS(x) \hfill \\
   &\le \dis {C_\delta }\int_{\partial ({B_{({d_{{\varepsilon _j},k}} + {\beta _{2,k}})/{\varepsilon _j}}}(y_{{\varepsilon _j}}^k) \cap {\Omega _{{\varepsilon _j}}})} {{e^{ - 2(\delta  + \delta ')|x - y_{{\varepsilon _j}}^k|}}|x - y_{{\varepsilon _j}}^k|} dS(x) \hfill \\
   &\le \dis {C_\delta }\exp \Bigl( {\frac{{ - 2(\delta  + \delta '){d_{{\varepsilon _j},k}}}}
{{{\varepsilon _j}}}} \Bigr)\frac{{{d_{{\varepsilon _j},k}} + {\beta _{2,k}}}}
{{\varepsilon _j^N}} \le {C_\delta }\exp \Bigl( {\frac{{ - 2\delta }}
{{{\varepsilon _j}}}{d_{{\varepsilon _j},k}}} \Bigr) \hfill \\
\end{array}
\]
for ${\varepsilon _j} >0$ small. In view of ${L^p}$-estimate and \eqref{3.102},
\[
\begin{array}{ll}
  |(II)|& \le \dis \int_{({B_{({d_{{\varepsilon _j},k}} + {\beta _{1,k}})/{\varepsilon _j}}}(0)\backslash {B_{({d_{{\varepsilon _j},k}} + {\beta _{2,k}})/{\varepsilon _j}}}(0))} {|\nabla {z_{{\varepsilon _j},k}}{|^2} + |\nabla {z_{{\varepsilon _j},k}}| \cdot |{\nabla ^2}{z_{{\varepsilon _j},k}}| \cdot |x|} dx \hfill \\
  & \le \dis{C_\delta }\int_{({B_{({d_{{\varepsilon _j},k}} + {\beta _{1,k}})/{\varepsilon _j}}}(0)\backslash {B_{({d_{{\varepsilon _j},k}} + {\beta _{2,k}})/{\varepsilon _j}}}(0))} {{e^{ - 2(\delta  + \delta ')|x|}}|x|} dx \hfill \\
   &\le \dis {C_\delta }\exp \Bigl( {\frac{{ - 2\delta }}
{{{\varepsilon _j}}}({d_{{\varepsilon _j},k}} + {\beta _{2,k}})} \Bigr)\int_{{\mathbb{R}^N}} {{e^{ - 2\delta '|x|}}|x|} dx \le {C_\delta }\exp \Bigl( {\frac{{ - 2\delta }}
{{{\varepsilon _j}}}({d_{{\varepsilon _j},k}} + {\beta _{2,k}})} \Bigr). \hfill \\
\end{array}
\]
Similarly, we get
$$
|(III)| \le {C_\delta }\exp \Bigl( {\frac{{ - 2\delta }}
{{{\varepsilon _j}}}({d_{{\varepsilon _j},k}} + {\beta _{2,k}})} \Bigr).
$$
Hence, we have
\begin{equation}\label{3.103}
{\left. {\frac{{dE_{{\varepsilon _j}}^k(z_{{\varepsilon _j},k}^t(x - y_{{\varepsilon _j}}^k))}}
{{dt}}} \right|_{t = 1}} \le {C_\delta } \exp \Bigl( {\frac{{ - 2\delta }}
{{{\varepsilon _j}}}{d_{{\varepsilon _j},k}}} \Bigr).
\end{equation}
Following \cite[Proposition~3.3]{b} or \cite[Proposition~2.6]{bl}, we let
\[
{T_k}: = {T_k}(t,y,{\varepsilon _j}) = (t - 1)\frac{{{\varepsilon _j}|y|}}
{{{d_{{\varepsilon _j},k}} + {\beta _{1,k}}}} + 1{\text{ and }}{A_k}: = {A_k}(t,y,{\varepsilon _j}) = \Bigl( {\frac{{{T_k} - 1}}
{{{T_k}}}\frac{{{y_m}{y_n}}}
{{|y{|^2}}}} \Bigr)_{m,n = 1}^N.
\]
Direct computations as in \cite[Proposition~3.3]{b} show that
\begin{equation}\label{3.104}
\begin{gathered}
\quad  \int_{{\mathbb{R}^N}} {|\nabla z_{{\varepsilon _j},k}^t{|^2}} dx \hfill \\
   = {t^{N - 2}}\sum\limits_{m = 1}^N {\int_{{\mathbb{R}^N}} {{{\left| {\frac{{\partial {z_{{\varepsilon _j},k}}}}
{{\partial {y_m}}}} \right|}^2}\left| {T_k^2 + \frac{{2({T_k} - 1)y_m^2}}
{{|y{|^2}}} + \frac{{{{({T_k} - 1)}^2}y_m^2}}
{{T_k^2|y{|^2}}}} \right|\frac{{|\det (I - {A_k})|}}
{{T_k^{N + 4}}}dy} } , \hfill \\
\end{gathered}
\end{equation}
\begin{equation}\label{3.105}
\int_{{\mathbb{R}^N}} {|z_{{\varepsilon _j},k}^t{|^2}} dx = {t^N}\int_{{\mathbb{R}^N}} {|{z_{{\varepsilon _j},k}}{|^2}\frac{{|\det (I - {A_k})|}}
{{T_k^N}}dy}
\end{equation}
and
\begin{equation}\label{3.106}
\begin{gathered}
\quad  \int_{{\mathbb{R}^N}} {G_{\varepsilon _j} ^k(x + y_{{\varepsilon _j}}^k,z_{{\varepsilon _j},k}^t)} dx \hfill \\
   = {t^N}\int_{{\mathbb{R}^N}} {{\chi _{{{({\Lambda ^k})}_{{\varepsilon _j}}}}}\Bigl( {\frac{y}
{{\bigl( {1 - \frac{1}
{t}} \bigr)\frac{{{\varepsilon _j}|y|}}
{{{d_{{\varepsilon _j},k}} + {\beta _{1,k}}}} + \frac{1}
{t}}} + y_{{\varepsilon _j}}^k} \Bigr)F({z_{{\varepsilon _j},k}})\frac{{|\det (I - {A_k})|}}
{{T_k^N}}} dy \hfill \\
 \quad  + {t^N}\int_{{\mathbb{R}^N}} {\Bigl( {1 - {\chi _{{{({\Lambda ^k})}_{{\varepsilon _j}}}}}\Bigl( {\frac{y}
{{\bigl( {1 - \frac{1}
{t}} \bigr)\frac{{{\varepsilon _j}|y|}}
{{{d_{{\varepsilon _j},k}} + {\beta _{1,k}}}} + \frac{1}
{t}}} + y_{{\varepsilon _j}}^k} \Bigr)} \Bigr)\tilde F({z_{{\varepsilon _j},k}})\frac{{|\det (I - {A_k})|}}
{{T_k^N}}} dy. \hfill \\
\end{gathered}
\end{equation}
Note that, for each $l>0$ and $R>0$, $ {T_k}(t,y,{\varepsilon _j}) \to 1$ and $ {A_k}(t,y,{\varepsilon _j}) \to 0$ as $j \to \infty $ uniformly for $|t| \le l$ and $|y| \le R$. Meanwhile, for any $t \in [a,b]$ with $a > 0$, $a \le {T_k}(t,y,{\varepsilon _j}) \le b + 1$ for $|y| \le ({d_{{\varepsilon _j},k}} + {\beta _{1,k}})/{\varepsilon _j}$. Besides, by \eqref{3.77}, we see that ${z_{{\varepsilon _j},k}} \to {U^k} \in S$ in ${H^1}({\mathbb{R}^N})$ as $j \to \infty $. Therefore, from \eqref{3.104} and \eqref{3.105}, we get
\begin{equation}\label{3.107}
\mathop {\lim }\limits_{j \to \infty } \int_{{\mathbb{R}^N}} {|\nabla z_{{\varepsilon _j},k}^t{|^2}}  = {t^{N - 2}}\int_{{\mathbb{R}^N}} {|\nabla {U^k}{|^2}} {\text{ and }}\mathop {\lim }\limits_{j \to \infty } \int_{{\mathbb{R}^N}} {|z_{{\varepsilon _j},k}^t{|^2}}  = {t^N}\int_{{\mathbb{R}^N}} {|{U^k}{|^2}}
\end{equation}
uniformly for $t \in [a,b]$ with $a > 0$. Moreover, from \eqref{3.106}, we have the following two cases.

{\bf Case~I.} $\mathop {\lim }\nolimits_{j \to \infty } d(y_{{\varepsilon _j}}^k,\partial {({\Lambda ^k})_{{\varepsilon _j}}}) =  + \infty $, then
\begin{equation}\label{3.108}
\mathop {\lim }\limits_{j \to \infty } \int_{{\mathbb{R}^N}} {G_{\varepsilon _j} ^k(x + y_{{\varepsilon _j}}^k,z_{{\varepsilon _j},k}^t)}  = {t^N}\int_{{\mathbb{R}^N}} {F({U^k})}
\end{equation}
uniformly for $t \in [a,b]$ with $a > 0$.

{\bf Case~II.} $\mathop {\lim }\nolimits_{j \to \infty } d(y_{{\varepsilon _j}}^k,\partial {({\Lambda ^k})_{{\varepsilon _j}}}) <  + \infty $, then, up to a translation and a rotation, we see that $\exists {x^k} \in \mathbb{R}_ + ^N$ such that
\[
\mathop {\lim }\limits_{j \to \infty } \int_{{\mathbb{R}^N}} {G_{{\varepsilon _j}}^k(x + y_{{\varepsilon _j}}^k,z_{{\varepsilon _j},k}^t)}  = {t^N}\int_{\mathbb{R}_ + ^N} {F({U^k}(x - {x^k}))}  + {t^N}\int_{{\mathbb{R}^N}\backslash \mathbb{R}_ + ^N} {\tilde F({U^k}(x - {x^k}))} .
\]
uniformly for $t \in [a,b]$ with $a > 0$. Arguing as in Proposition~\ref{3.4.}, we see that the situation which is analogous to {\bf Case~B} in Proposition~\ref{3.4.} will occur. That is,
\begin{equation}\label{3.109}
\tilde f({U^k}(x - {x^k})) = f({U^k}(x - {x^k})){\text{ for all }}x \in {\mathbb{R}^N}\backslash \mathbb{R}_ + ^N.
\end{equation}
To sum up, in either {\bf Case~I} or {\bf Case~II}, \eqref{3.108} always holds. It follows from \eqref{3.107}, \eqref{3.108} and the Pohozaev's identity \eqref{add12} that
\begin{equation}\label{add14}
\mathop {\lim }\limits_{j \to \infty } E_{{\varepsilon _j}}^k(z_{{\varepsilon _j},k}^t(x - y_{{\varepsilon _j}}^k)) = \Bigl( {\frac{{{t^{N - 2}}}}
{2} - \frac{{N - 2}}
{{2N}}{t^N}} \Bigr)\int_{{\mathbb{R}^N}} {|\nabla {U^k}{|^2}}
\end{equation}
uniformly for $t \in [a,b]$ with $a > 0$. By \eqref{add13}, we see that
\begin{equation}\label{3.110}
\begin{gathered}
\quad  \frac{d}
{{dt}}\int_{{\mathbb{R}^N}} {G_{{\varepsilon _j}}^k(x + y_{{\varepsilon _j}}^k,z_{{\varepsilon _j},k}^t)} dx \hfill \\
   = {t^{N - 1}}\int_{{\mathbb{R}^N}} {{\chi _{{{({\Lambda ^k})}_{{\varepsilon _j}}}}}\Bigl( {\frac{y}
{{\bigl( {1 - \frac{1}
{t}} \bigr)\frac{{{\varepsilon _j}|y|}}
{{{d_{{\varepsilon _j},k}} + {\beta _{1,k}}}} + \frac{1}
{t}}} + y_{{\varepsilon _j}}^k} \Bigr)f({z_{{\varepsilon _j},k}})}  \hfill \\
~\quad   \cdot (\nabla {z_{{\varepsilon _j},k}} \cdot y)\Bigl( {\frac{{{\varepsilon _j}|y|}}
{{{d_{{\varepsilon _j},k}} + {\beta _{1,k}}}} - 1} \Bigr)\frac{{|\det (I - {A_k})|}}
{{T_k^N}}dy \hfill \\
\quad   + {t^{N - 1}}\int_{{\mathbb{R}^N}} {\Bigl( {1 - {\chi _{{{({\Lambda ^k})}_{{\varepsilon _j}}}}}\Bigl( {\frac{y}
{{\bigl( {1 - \frac{1}
{t}} \bigr)\frac{{{\varepsilon _j}|y|}}
{{{d_{{\varepsilon _j},k}} + {\beta _{1,k}}}} + \frac{1}
{t}}} + y_{{\varepsilon _j}}^k} \Bigr)} \Bigr)\tilde f({z_{{\varepsilon _j},k}})}  \hfill \\
~ \quad  \cdot (\nabla {z_{{\varepsilon _j},k}} \cdot y)\left( {\frac{{{\varepsilon _j}|y|}}
{{{d_{{\varepsilon _j},k}} + {\beta _{1,k}}}} - 1} \right)\frac{{|\det (I - {A_k})|}}
{{T_k^N}}dy. \hfill \\
\end{gathered}
\end{equation}
Differentiating both sides of \eqref{3.110}, using \cite[Proposition~3.3]{b} or \cite[Proposition~2.6]{bl} and combining with the fact that ${z_{{\varepsilon _j},k}} \to {U^k} \in S$ in ${H^1}({\mathbb{R}^N})$ as $j \to \infty $ and the exponential decays of ${z_{{\varepsilon _j},k}}$ and ${U^k}$, we see that
\[\begin{gathered}
\quad  \mathop {\lim }\limits_{j \to \infty } \frac{{{d^2}}}
{{d{t^2}}}\int_{{\mathbb{R}^N}} {G_{{\varepsilon _j}}^k(x + y_{{\varepsilon _j}}^k,z_{{\varepsilon _j},k}^t)} dx \hfill \\
   = \mathop {\lim }\limits_{j \to \infty } {t^{N - 1}}\int_{{\mathbb{R}^N}} {\nabla {\chi _{{{({\Lambda ^k})}_{{\varepsilon _j}}}}}\Bigl( {\frac{y}
{{\bigl( {1 - \frac{1}
{t}} \bigr)\frac{{{\varepsilon _j}|y|}}
{{{d_{{\varepsilon _j},k}} + {\beta _{1,k}}}} + \frac{1}
{t}}} + y_{{\varepsilon _j}}^k} \Bigr) \cdot \frac{y}
{{( - {t^2}){{\left[ {\bigl( {1 - \frac{1}
{t}} \bigr)\frac{{{\varepsilon _j}|y|}}
{{{d_{{\varepsilon _j},k}} + {\beta _{1,k}}}} + \frac{1}
{t}} \right]}^2}}}}  \hfill \\
\quad   \cdot {\left[ {\frac{{{\varepsilon _j}|y|}}
{{{d_{{\varepsilon _j},k}} + {\beta _{1,k}}}} - 1} \right]^2}(f({z_{{\varepsilon _j},k}}) - \tilde f({z_{{\varepsilon _j},k}}))(\nabla {z_{{\varepsilon _j},k}} \cdot y)\frac{{|\det (I - {A_k})|}}
{{T_k^N}}dy \hfill \\
\quad   - (N - 1){t^{N - 2}}\int_{{\Omega ^\infty }} {f({U^k})(\nabla {U^k} \cdot y)} dy - (N - 1){t^{N - 2}}\int_{{\mathbb{R}^N}\backslash {\Omega ^\infty }} {\tilde f({U^k})(\nabla {U^k} \cdot y)} dy, \hfill \\
\end{gathered} \]
where we note that ${\nabla {\chi _{{{({\Lambda ^k})}_{{\varepsilon _j}}}}}}$ is in the distributional sense and ${\Omega ^\infty } \subset {\mathbb{R}^N}$. We have the following two subcases.

{\bf Case~1.} $\mathop {\lim }\nolimits_{j \to \infty } d(y_{{\varepsilon _j}}^k,\partial {({\Lambda ^k})_{{\varepsilon _j}}}) =  + \infty $, then ${\Omega ^\infty } = {\mathbb{R}^N}$ and
\begin{equation}\label{3.111}
\begin{array}{ll}
\dis  \mathop {\lim }\limits_{j \to \infty } \frac{{{d^2}}}
{{d{t^2}}}\int_{{\mathbb{R}^N}} {G_{{\varepsilon _j}}^k(x + y_{{\varepsilon _j}}^k,z_{{\varepsilon _j},k}^t)} dx &=  - (N - 1){t^{N - 2}}\dis\int_{{\mathbb{R}^N}} {f({U^k})(\nabla {U^k} \cdot y)} dy \hfill \\
  & = N(N - 1){t^{N - 2}}\dis\int_{{\mathbb{R}^N}} {F({U^k})}  \hfill \\
\end{array}
\end{equation}
uniformly for $t \in [a,b]$ with $a > 0$.

{\bf Case~2.} $\mathop {\lim }\nolimits_{j \to \infty } d(y_{{\varepsilon _j}}^k,\partial {({\Lambda ^k})_{{\varepsilon _j}}}) <  + \infty $, then, up to a translation and a rotation, ${\Omega ^\infty } = \mathbb{R}_ + ^N - {x^k}$ for some ${x^k} \in \mathbb{R}_ + ^N$,  together with \eqref{3.109}, we get
\begin{equation}\label{3.112}
\begin{gathered}
\quad  \mathop {\lim }\limits_{j \to \infty } \frac{{{d^2}}}
{{d{t^2}}}\int_{{\mathbb{R}^N}} {G_{{\varepsilon _j}}^k(x + y_{{\varepsilon _j}}^k,z_{{\varepsilon _j},k}^t)} dx \hfill \\
   = \frac{1}
{{{t^2}}}\int_{\partial \mathbb{R}_ + ^N - t{x^k}} {(f({U^k}(y/t) - \tilde f({U^k}(y/t))(\nabla {U^k}(y/t) \cdot y)(y \cdot \nu )} dS(y) \hfill \\
\quad   - (N - 1){t^{N - 2}}\int_{\mathbb{R}_ + ^N - {x^k}} {f({U^k})(\nabla {U^k} \cdot y)} dy \hfill \\
\quad   - (N - 1){t^{N - 2}}\int_{{\mathbb{R}^N}\backslash (\mathbb{R}_ + ^N - {x^k})} {\tilde f({U^k})(\nabla {U^k} \cdot y)} dy \hfill \\
   = N(N - 1){t^{N - 2}}\int_{{\mathbb{R}^N}} {F({U^k})}  \hfill \\
\end{gathered}
\end{equation}
uniformly for $t \in [a,b]$ with $a > 0$, where $\nu$ and $dS(y)$ are the outward normal and the volume element on ${\partial \mathbb{R}_ + ^N - t{x^k}}$, respectively. Differentiating both sides of \eqref{3.104}, \eqref{3.105} twice and combining with \eqref{3.111}, \eqref{3.112}, we see from Pohozaev's identity \eqref{add12} that for $\gamma >0$ small but fixed,
\begin{equation}\label{3.113}
\mathop {\lim }\limits_{j \to \infty } \frac{{{d^2}E_{{\varepsilon _j}}^k(z_{{\varepsilon _j},k}^t(x - y_{{\varepsilon _j}}^k))}}
{{d{t^2}}} \le  - \frac{{N - 2}}
{2}\int_{{\mathbb{R}^N}} {|\nabla {U^k}{|^2}}
\end{equation}
uniformly for $t \in [1 - \gamma ,1 + \gamma ]$. It follows from \eqref{3.103}, \eqref{3.113} and Lagrange's mean value theorem that there exists ${t_{{\varepsilon _j},k}} \in (1 - \gamma ,1 + \gamma )$ such that
\[
E_{{\varepsilon _j}}^k\Bigl(z_{{\varepsilon _j},k}^{t_{{\varepsilon _j},k}}(x - y_{{\varepsilon _j}}^k)\Bigr) = \mathop {\max }\limits_{t \in [1 - \gamma ,1 + \gamma ]} E_{{\varepsilon _j}}^k(z_{{\varepsilon _j},k}^t(x - y_{{\varepsilon _j}}^k)){\text{ and }}|{t_{{\varepsilon _j},k}}- 1| \le {C_\delta }\exp \Bigl( {\frac{{ - 2\delta }}
{{{\varepsilon _j}}}{d_{{\varepsilon _j},k}}} \Bigr).
\]
Thus, it follows from Taylor expansion that
\begin{equation}\label{3.114}
\Bigl| {\mathop {\max }\limits_{t \in [1 - \gamma ,1 + \gamma ]} E_{{\varepsilon _j}}^k(z_{{\varepsilon _j},k}^t(x - y_{{\varepsilon _j}}^k)) - E_{{\varepsilon _j}}^k({z_{{\varepsilon _j},k}}(x - y_{{\varepsilon _j}}^k))} \Bigr| \le {C_\delta }\exp \Bigl( {\frac{{ - 4\delta }}
{{{\varepsilon _j}}}{d_{{\varepsilon _j},k}}} \Bigr).
\end{equation}
By $({F_1})$ and \eqref{3.107}, we see that for some small ${t_{0,k}} > 0$,
\[
\frac{{dE_{{\varepsilon _j}}^k\bigl(sz_{{\varepsilon _j},k}^{{t_{0,k}}}(x - y_{{\varepsilon _j}}^k)\bigr)}}
{{ds}} \ge 0{\text{ for }}s \in [0,1]{\text{ and }}j \in \mathbb{N}{\text{ large.}}
\]
Furthermore, in view of \eqref{3.107} and \eqref{3.108}, we see that for some ${t_{1,k}} > 0$, $E_{{\varepsilon _j}}^k(z_{{\varepsilon _j},k}^{{t_{1,k}}}(x - y_{{\varepsilon _j}}^k)) <  - 1$ for $j \in \mathbb{N}$ large. Then, we define a continuous curve $\alpha _{{\varepsilon _j}}^k(t):[0,1] \to {H^1}({\mathbb{R}^N})$ by
\begin{equation}\label{3.115}
{\alpha _{{\varepsilon _j},k}}(t) = \left\{ \begin{array}{ll}
  2tz_{{\varepsilon _j},k}^{{t_{0,k}}}(x - y_{{\varepsilon _j}}^k),&t \in [0,1/2], \hfill \\
   z_{{\varepsilon _j},k}^{(2 - 2t){t_{0,k}} + (2t - 1){t_{1,k}}}(x - y_{{\varepsilon _j}}^k),&t \in [1/2,1]. \hfill \\
\end{array}  \right.
\end{equation}
From \eqref{add14}, we see that for $j \in \mathbb{N}$ large,
\[\begin{gathered}
\quad  \mathop {\max }\limits_{t \in [0,1]} E_{{\varepsilon _j}}^k({\alpha _{{\varepsilon _j},k}}(t)) \hfill \\
   = \mathop {\max }\limits_{t \in [1/2,1]} E_{{\varepsilon _j}}^k({\alpha _{{\varepsilon _j},k}}(t)) = \mathop {\max }\limits_{t \in [{t_{0,k}},{t_{1,k}}]} E_{{\varepsilon _j}}^k(z_{{\varepsilon _j},k}^t(x - y_{{\varepsilon _j}}^k)) = \mathop {\max }\limits_{t \in [1 - \gamma ,1 + \gamma ]} E_{{\varepsilon _j}}^k(z_{{\varepsilon _j},k}^t(x - y_{{\varepsilon _j}}^k)), \hfill \\
\end{gathered} \]
then by \eqref{3.114},
\begin{equation}\label{3.116}
\mathop {\max }\limits_{t \in [0,1]} E_{{\varepsilon _j}}^k({\alpha _{{\varepsilon _j},k}}(t)) \le E_{{\varepsilon _j}}^k({z_{{\varepsilon _j},k}}(x - y_{{\varepsilon _j}}^k)) + {C_\delta }\exp \Bigl( {\frac{{ - 4\delta }}
{{{\varepsilon _j}}}{d_{{\varepsilon _j},k}}} \Bigr).
\end{equation}
Using the mollification method, we choose a sequence of mollifiers $\{ {\eta _{{\rho _{{\varepsilon _j}}}}}\} _{j = 1}^\infty $ with radius ${\rho _{{\varepsilon _j}}} \to 0$ as $j \to \infty $ and consider ${{\tilde z}_{{\varepsilon _j},k}}: = {\eta _{{\rho _{{\varepsilon _j}}}}} * {z_{{\varepsilon _j},k}}$, the mollification of ${z_{{\varepsilon _j},k}}$. Note that the support of ${z_{{\varepsilon _j},k}}$ is contained in
${B_{({d_{{\varepsilon _j},k}} + {\beta _{1,k}})/{\varepsilon _j}}}(0)$, for each ${\varepsilon _j} > 0$, we can choose ${\rho _{{\varepsilon _j}}}$ suitably small to ensure that ${{\tilde z}_{{\varepsilon _j},k}} \in C_c^\infty ({B_{({d_{{\varepsilon _j},k}} + {\beta _{1,k}})/{\varepsilon _j}}}(0))$ and satisfying
\begin{equation}\label{3.117}
{\left\| {{{\tilde z}_{{\varepsilon _j},k}} - {z_{{\varepsilon _j},k}}} \right\|_{{H^1}({\mathbb{R}^N})}} \le \exp \Bigl( {\frac{{ - 4\delta }}
{{{\varepsilon _j}}}{d_{{\varepsilon _j},k}}} \Bigr).
\end{equation}
Moreover, we mention that
\begin{equation}\label{3.118}
|a{|^p} \le |b{|^p} + {C_p}(|a{|^{p - 1}} + |b{|^{p - 1}})|a - b|,~(a,b \in \mathbb{R},~p \ge 1).
\end{equation}
Similar to \eqref{add15}, \eqref{add11} and \eqref{3.115}, we define a continuous curve ${\tilde\alpha} _{{\varepsilon _j}}^k(t):[0,1] \to {H^1}({\mathbb{R}^N})$ by
\begin{equation}\label{add16}
{{\tilde\alpha} _{{\varepsilon _j},k}}(t) = \left\{ \begin{array}{ll}
  2t{\tilde z}_{{\varepsilon _j},k}^{{t_{0,k}}}(x - y_{{\varepsilon _j}}^k),&t \in [0,1/2], \hfill \\
  {\tilde z}_{{\varepsilon _j},k}^{(2 - 2t){t_{0,k}} + (2t - 1){t_{1,k}}}(x - y_{{\varepsilon _j}}^k),&t \in [1/2,1]. \hfill \\
\end{array}  \right.
\end{equation}
Recalling that ${t_0} \le {T_k}(t,y,{\varepsilon _j}) \le {t_1} + 1$ for $t \in [{t_0},{t_1}]$ and $|y| \le ({d_{{\varepsilon _j},k}} + {\beta _{1,k}})/{\varepsilon _j}$, we see from \eqref{3.104}, \eqref{3.105}, \eqref{3.106}, \eqref{3.117} and \eqref{3.118} that for each $t \in [0,1]$,
\begin{equation}\label{3.119}
I({{\tilde \alpha }_{{\varepsilon _j},k}}(t)) \le E_{{\varepsilon _j}}^k({{\tilde \alpha }_{{\varepsilon _j},k}}(t)) \le E_{{\varepsilon _j}}^k({\alpha _{{\varepsilon _j},k}}(t)) + {C_\delta }\exp \Bigl( {\frac{{ - 4\delta }}
{{{\varepsilon _j}}}{d_{{\varepsilon _j},k}}} \Bigr),
\end{equation}
where $C>0$ is independent of $\varepsilon>0$ and $t \in [0,1]$. Letting ${S_{{\varepsilon _j},k}}(t)$ be the Schwartz symmetrization of ${\tilde \alpha _{{\varepsilon _j},k}}(t)(x + y_{{\varepsilon _j}}^k)$ and ${\beta _{3,k}}({\varepsilon _j}) > 0$ be such that
\[
\mathop {\max }\limits_{t \in [0,1]} |{\text{supp}}({{\tilde \alpha }_{{\varepsilon _j},k}}(t))| = {B_{({d_{{\varepsilon _j},k}} + {\beta _{3,k}}({\varepsilon _j}))/{\varepsilon _j}}}(0).
\]
Recalling that for each $t \in (0, + \infty )$, $y(x,t)$ is a diffeomorphism from ${B_{({d_{{\varepsilon _j},k}} + {\beta _{1,k}})/{\varepsilon _j}}}(0)$ onto ${B_{({d_{{\varepsilon _j},k}} + {\beta _{1,k}})/{\varepsilon _j}}}(0)$, it follows that $\{ {\beta _{3,k}}({\varepsilon _j})\} _{j = 1}^\infty $ is bounded away from $0$ and ${\beta _{1,k}}$, we can adjust ${\beta _{2,k}}>0$ to ensure that $0 < \mathop {\inf }\nolimits_{j \in \mathbb{N}} {\beta _{3,k}}({\varepsilon _j}) \le \mathop {\sup }\nolimits_{j \in \mathbb{N}} {\beta _{3,k}}({\varepsilon _j}) < {\beta _{2,k}} < {\beta _{1,k}}$. In view of the classical P\'{o}lya-Szeg\"{o} principle \cite{ps1}, we see that
\begin{equation}\label{3.120}
{S_{{\varepsilon _j},k}}(t) \in H_0^1({B_{({d_{{\varepsilon _j},k}} + {\beta _{3,k}}({\varepsilon _j}))/{\varepsilon _j}}}(0)),~{S_{{\varepsilon _j},k}}(0) = 0{\text{ and }}I({S_{{\varepsilon _j},k}}(1)) < 0.
\end{equation}
Inspired by \cite{b1}, we define
\[
{S^{{\delta _0}}}: = \{ U \in {H^1}({\mathbb{R}^N}):d(U,S) \le {\delta _0}\}
\]
and
\[
{S^{*,{\delta _0}}}: = \{ |U{|^*} \in H_{{\text{rad}}}^1({\mathbb{R}^N}):U \in {S^{{\delta _0}}}\}.
\]
By Proposition~\ref{2.1.} $(ii)$, it follows that for each $\delta_0 > 0$, ${S^{ * ,\delta_0}}$ is bounded. Then for each $R > 0$, we define
\[
S_R^{\delta_0 }: = \{ u \in H_{{\text{rad}}}^1({B_R}(0)):d(u,{S^{ * ,\delta_0 }}) \le \delta_0 \} ,
\]
where
$$
H_{{\text{rad}}}^1({B_R}(0)): = \{ u \in H_{\text{0}}^1({B_R}(0)):u(x) = u(|x|)\}.
$$
Note that for each $t \in [0,1]$, ${{\tilde \alpha }_{{\varepsilon _j},k}}(t) \in {C^\infty }$, the function ${{\tilde \alpha }_{{\varepsilon _j},k}}(t)$ is co-area regular (see \cite[Definition~1.2.6]{al}), then by \cite[Theorem~1.4]{al}, it follows that the path ${S_{{\varepsilon _j},k}}(t) \in C([0,1],H_{{\text{rad}}}^1({B_{({d_{{\varepsilon _j},k}} + {\beta _{3,k}}({\varepsilon _j}))/{\varepsilon _j}}}(0)))$. Next, we claim that, for ${\delta _0} > 0$ small but fixed, there exists a $\mu > 0$ such that for ${\varepsilon _j} > 0$ small,
\begin{equation}\label{3.121}
I({S_{{\varepsilon _j},k}}(t)) \le c - \mu {\text{ if }}{S_{{\varepsilon _j},k}}(t) \notin S_{({d_{{\varepsilon _j},k}} + {\beta _{3,k}}({\varepsilon _j}))/{\varepsilon _j}}^{{\delta _0}}.
\end{equation}
In view of \eqref{add14}, \eqref{3.115}, \eqref{3.117} and \eqref{add16}, it suffices to prove that there exists $\nu > 0$ such that for ${\varepsilon _j} > 0$ small,
\[
\Bigl| {t - \frac{{1 - 2{t_{0,k}} + {t_{1,k}}}}
{{2{t_{1,k}} - 2{t_{0,k}}}}} \Bigr| \ge \nu {\text{ if }}{S_{{\varepsilon _j},k}}(t) \notin S_{({d_{{\varepsilon _j},k}} + {\beta _{3,k}}({\varepsilon _j}))/{\varepsilon _j}}^{{\delta _0}}.
\]
Otherwise, there exists a sequence $\{ {t_{{\varepsilon _j}}}\} _{j = 1}^\infty $ such that ${S_{{\varepsilon _j},k}}({t_{{\varepsilon _j}}}) \notin S_{({d_{{\varepsilon _j},k}} + {\beta _{3,k}}({\varepsilon _j}))/{\varepsilon _j}}^{{\delta _0}}$ and ${t_{{\varepsilon _j}}} \to (1 - 2{t_{0,k}} + {t_{1,k}})/(2{t_{1,k}} - 2{t_{0,k}})$ as $j \to \infty $. Moreover, we see from \eqref{3.104}, \eqref{3.105} and the fact
${{\tilde z}_{{\varepsilon _j},k}} \to {U^k} \in S$ in ${H^1}({\mathbb{R}^N})$ that there exists a $\beta  > 0$ small but fixed, $\tilde z_{{\varepsilon _j},k}^t \to {U^k}(y(x,t)) \in {S^{{\delta _0}/2}}$ uniformly for $t \in [1 - \beta ,1 + \beta ]$, then for ${\varepsilon _j} > 0$ small, we see that $\tilde z_{{\varepsilon _j},k}^t \in {S^{{\delta _0}}}$ for $t \in [1 - \beta ,1 + \beta ]$. Thus we conclude that for ${\varepsilon _j} > 0$ small, ${S_{{\varepsilon _j},k}}(t) \in S_{({d_{{\varepsilon _j},k}} + {\beta _{3,k}}({\varepsilon _j}))/{\varepsilon _j}}^{{\delta _0}}$ if $t$ is close to $(1 - 2{t_{0,k}} + {t_{1,k}})/(2{t_{1,k}} - 2{t_{0,k}})$, which is a contradiction. The claim \eqref{3.121} holds. Since ${S_{{\varepsilon _j},k}}(t)$ is continuous, there exists ${t'_{{\varepsilon _j},k}} \in (0,1)$ such that
\[
I({S_{{\varepsilon _j},k}}({t'_{{\varepsilon _j},k}})) = \mathop {\max }\limits_{t \in [0,1]} I({S_{{\varepsilon _j},k}}(t)): = {c_{({d_{{\varepsilon _j},k}} + {\beta _{3,k}}({\varepsilon _j}))/{\varepsilon _j}}}.
\]
By \eqref{3.120} and the definition of the mountain pass level $c$, we see that $c \le {c_{({d_{{\varepsilon _j},k}} + {\beta _{3,k}}({\varepsilon _j}))/{\varepsilon _j}}}$. On the other hand, by \eqref{3.108}, \eqref{3.116} and \eqref{3.119}, we have
\[
{c_{({d_{{\varepsilon _j},k}} + {\beta _{3,k}}({\varepsilon _j}))/{\varepsilon _j}}} \le c + {C_\delta }\exp \Bigl( {\frac{{ - 4\delta }}
{{{\varepsilon _j}}}{d_{{\varepsilon _j},k}}} \Bigr),
\]
then it follows that
\begin{equation}\label{3.122}
{c_{({d_{{\varepsilon _j},k}} + {\beta _{3,k}}({\varepsilon _j}))/{\varepsilon _j}}} \to c{\text{ as }}j \to \infty .
\end{equation}
Moreover, by \eqref{3.121} and \eqref{3.122}, we deduce that
\begin{equation}\label{3.123}
{S_{{\varepsilon _j},k}}({t'_{{\varepsilon _j},k}}) \in S_{({d_{{\varepsilon _j},k}} + {\beta _{3,k}}({\varepsilon _j}))/{\varepsilon _j}}^{{\delta _0}}.
\end{equation}
In view of \cite[Proposition~3.1]{b1} or Proposition~3.1, Lemma~3.2 and ``Proof of (i) of Proposition~3.1" in \cite{bzz}, we see from \eqref{3.120}, \eqref{3.121}, \eqref{3.122} and \eqref{3.123} that for ${\delta _0} > 0$ small, there exists a sequence $\{ {u_n}\} _{n = 1}^\infty  \subset S_{({d_{{\varepsilon _j},k}} + {\beta _{3,k}}({\varepsilon _j}))/{\varepsilon _j}}^{{\delta _0}}$ with $I({u_n}) \le {c_{({d_{{\varepsilon _j},k}} + {\beta _{3,k}}({\varepsilon _j}))/{\varepsilon _j}}}$ and ${I'_{({d_{{\varepsilon _j},k}} + {\beta _{3,k}}({\varepsilon _j}))/{\varepsilon _j}}}({u_n}) \to 0$ in ${(H_{{\text{rad}}}^1({B_{({d_{{\varepsilon _j},k}} + {\beta _{3,k}}({\varepsilon _j}))/{\varepsilon _j}}}(0)))^{ - 1}}$ as $n \to \infty $, where
\begin{equation}\label{3.124}
{I_R}(u) = \frac{1}
{2}\int_{{B_R}(0)} {|\nabla u{|^2}}  + \frac{1}
{2}\int_{{B_R}(0)} {{u^2}}  - \int_{{B_R}(0)} {F(u)} ,{\text{ }}u \in H_{{\text{rad}}}^1({B_R}(0)).
\end{equation}
Arguing as \eqref{3.22}-\eqref{add17} in Proposition~\ref{3.4.}, we get a ${u_{{\varepsilon _j} ,k}} \in H_{{\text{rad}}}^1({B_R}(0))$ such that ${u_n} \to {u_{{\varepsilon _j} ,k}}$ in $H_{\text{0}}^1({B_R}(0))$ and ${u_{{\varepsilon _j} ,k}}$ is a critical point of \eqref{3.124} with $R = ({d_{{\varepsilon _j},k}} + {\beta _{3,k}}({\varepsilon _j}))/{\varepsilon _j}$. Moreover, in view of ``Appendix: A lower energy estimate in Proposition~3.1" in \cite{b1} or Proposition~4.3 in \cite{b}, we have the lower energy estimate, that is for any ${\delta _1} > 1$,
\begin{equation}\label{3.125}
{c_{({d_{{\varepsilon _j},k}} + {\beta _{3,k}}({\varepsilon _j}))/{\varepsilon _j}}} \ge I({u_{{\varepsilon _j} ,k}}) \ge c + {C_{{\delta _1}}}\exp \Bigl( { - \frac{{2{\delta _1}}}
{{\varepsilon _j} }({d_{{\varepsilon _j},k}} + {\beta _{3,k}}({\varepsilon _j})} \Bigr).
\end{equation}
By \eqref{3.101}, \eqref{3.116}, \eqref{3.119} and \eqref{3.125}, we get
\[
\begin{array}{ll}
  {E_{{\varepsilon _j}}}({v_{{\varepsilon _j}}}) & \ge \dis Kc + {C_{{\delta _1}}}\sum\limits_{k = 1}^K {\exp \Bigl( { - \frac{{2{\delta _1}}}
{{{\varepsilon _j}}}({d_{{\varepsilon _j},k}} + {\beta _{3,k}}({\varepsilon _j}))} \Bigr)}  \hfill \\
   & \dis \quad - {C_\delta }\sum\limits_{k = 1}^K {\exp \Bigl( { - \frac{{2\delta }}
{{{\varepsilon _j}}}({d_{{\varepsilon _j},k}} + {\beta _{2,k}})} \Bigr)}  - {C_\delta }\sum\limits_{k = 1}^K {\exp \Bigl( { - \frac{{4\delta }}
{{{\varepsilon _j}}}{d_{{\varepsilon _j},k}}} \Bigr)} . \hfill \\
\end{array}
\]
For ${\varepsilon _j} > 0$ small and $0 < \delta  < 1 < {\delta _1}$ with $\delta$ and ${\delta _1}$ close to $1$, we see that
\begin{equation}\label{3.126}
{E_{{\varepsilon _j}}}({v_{{\varepsilon _j}}}) \ge Kc + {C_{{\delta _1}}}\sum\limits_{k = 1}^K {\exp \Bigl( { - \frac{{2{\delta _1}}}
{{{\varepsilon _j}}}({d_{{\varepsilon _j},k}} + {\beta _{3,k}}({\varepsilon _j}))} \Bigr)}.
\end{equation}
Recalling that ${v_\varepsilon } \in J_\varepsilon ^{\widetilde{{c_\varepsilon }} + \sum\nolimits_{k = 1}^K {\exp ( - 2{D_k}/\varepsilon )} } \cap X_\varepsilon ^{{d_0}}$, by Lemma~\ref{3.1.}(i), Lemma~\ref{3.2.} and \eqref{3.126}, we get \eqref{3.100}. \eqref{3.100} and $(H_1)$ imply that only the situation which is a  analogous to \textbf{Case~A} will occur, that is, for each $1 \le k \le K$,
\begin{equation}\label{3.127}
{\varepsilon _j}y_{{\varepsilon _j}}^k \in {\Lambda ^k}{\text{ and }}\mathop {\lim }\limits_{j \to \infty } d(y_{{\varepsilon _j}}^k,\partial {({\Lambda ^k})_{{\varepsilon _j}}}) =  + \infty .
\end{equation}
\textbf{Step 4}. We prove the existence and concentration results in Theorem~\ref{1.1.}. Choosing $\delta  = \bar \delta  > 0$ small in \eqref{3.97} to ensure that $f(t) < \frac{1}{2}t$ for $t \in (0,\bar \delta )$, by \eqref{3.127}, we see that $\mathop  \cup \nolimits_{k = 1}^K {B_{{R_{\bar \delta }}}}(y_{{\varepsilon _j}}^k) \subset {\Lambda _{{\varepsilon _j}}}$ for ${\varepsilon _j} > 0$ small. Thus, ${v_{{\varepsilon _j}}}$ is a solution to \eqref{3.1}, i.e. ${u_{{\varepsilon _j}}}(x) = {v_{{\varepsilon _j}}}(x/{\varepsilon _j})$ is a solution to \eqref{1.1}.

From \eqref{3.78} and elliptic estimates, we see that for each $1 \le k \le K$, $\{ {v_{{\varepsilon _j}}}(x + y_{{\varepsilon _j}}^k)\} _{j = 1}^\infty $ is bounded in $C_{{\text{loc}}}^{1,\alpha }({\mathbb{R}^N})$ for some $0 < \alpha  < 1$. It follows from Arzel\'{a}-Ascoli theorem and \eqref{3.77} that
\begin{equation}\label{3.128}
{v_{{\varepsilon _j}}}(x + y_{{\varepsilon _j}}^k) \to {U^k}(x){\text{ in }}C_{{\text{loc}}}^1{\text{(}}{\mathbb{R}^N}{\text{) as }}j \to \infty .
\end{equation}
Letting $x_{{\varepsilon _j}}^k$ be a maximum point of ${u_{{\varepsilon _j}}}$ in $\overline {{\Lambda ^k}} $, we see from \eqref{3.128} that for ${\varepsilon _j} > 0$ small, ${u_{{\varepsilon _j}}}(x_{{\varepsilon _j}}^k) = {v_{{\varepsilon _j}}}(x_{{\varepsilon _j}}^k/{\varepsilon _j}) \ge {v_{{\varepsilon _j}}}(y_{{\varepsilon _j}}^k) \ge {U^k}(0)/2 > 0$. Choosing $\delta  = \tilde \delta : = \mathop {\min }\nolimits_{1 \le k \le K} \{ {U^k}(0)/2\}  > 0$ in \eqref{3.97}, then there exists ${R_{\tilde \delta }} > 0$ such that $|{v_{{\varepsilon _j}}}(x)| < \tilde \delta $ for all $x \in {\Omega _{{\varepsilon _j}}}\backslash \mathop  \cup \nolimits_{k = 1}^K {B_{{R_{\tilde \delta }}}}(y_{{\varepsilon _j}}^k)$. Thus, $|(x_{{\varepsilon _j}}^k/{\varepsilon _j}) - y_{{\varepsilon _j}}^k| \le {R_{\tilde \delta }}$, by \eqref{3.100}, we see that $d(x_{{\varepsilon _j}}^k,\partial \Omega ) \to {D_k}$ as $j \to \infty $. By the arbitrariness of the sequence $\{ {\varepsilon _j}\} _{j = 1}^\infty $, we obtain the existence and concentration results in Theorem~\ref{1.1.}.

The exponential decay of ${u_\varepsilon }$ follows immediately from \eqref{3.98}. This finishes the proof.
\end{proof}

\end{document}